\documentclass{amsart}
\usepackage{amssymb,latexsym}
\usepackage{amsfonts}
\usepackage{amsmath}
\usepackage{enumitem}

\newtheorem{theorem}{Theorem}[section]
\newtheorem{lemma}{Lemma}[section]
\newtheorem{corollary}{Corollary}[section]
\newtheorem{remark}{Remark}[section]
\newtheorem{prop}{Proposition}[section]
\newtheorem{definition}{Definition}[section]

\numberwithin{equation}{section}

\newcommand{\ba}{\begin{array}}
\newcommand{\ea}{\end{array}}

\newcommand{\rank}{{\rm rank}}

\newcommand{\Id}{\textrm{Id}}
\newcommand{\Vol}{{\rm Vol}}

\def \qed{\cqfd}

\def\qed{\vbox{\hrule
\hbox{\vrule\hbox to 5pt{\vbox to 8pt{\vfil}\hfil}\vrule}\hrule}}

\newcommand{\beg}{\begin{eqnarray*}}
\newcommand{\begn}{\begin{eqnarray}}
\newcommand{\en}{\end{eqnarray*}}
\newcommand{\enn}{\end{eqnarray}}

\newcommand{\tr}{\mbox{\rm tr\,}}

\begin{document}

\vspace*{0.5cm} \begin{center}\noindent {\LARGE \bf Projectively flat bundles and semi-stable Higgs bundles  }\\[0.7cm]
Changpeng Pan\\[0.2cm]
School of Mathematical Sciences\\
University of Science and Technology of China\\
Hefei, 230026,P.R. China\\
E-mail: pcp1995@mail.ustc.edu.cn\\[0.5cm]
Chuanjing Zhang\\[0.2cm]School of Mathematical Sciences\\
University of Science and Technology of China\\
Hefei, 230026,P.R. China\\
E-mail: chjzhang@mail.ustc.edu.cn\\[0.5cm]
Xi Zhang \footnote[1]{The third author is the corresponding
author who was supported in part by NSF in China, No.11625106, 11571332 and 11721101.}\\
School of Mathematical Sciences\\
University of Science and Technology of China\\
Hefei, 230026,P.R. China\\
E-mail: mathzx@ustc.edu.cn\\[1cm]
\end{center}


{\bf Abstract.}The Corlette-Donaldson-Hitchin-Simpson's correspondence states that, on a compact K\"ahler manifold $(X, \omega )$, there is  a one-to-one correspondence between the moduli space of semisimple flat complex vector bundles and the moduli space of poly-stable Higgs bundles with vanishing Chern numbers. In this paper, we extend this correspondence to the projectively flat bundles case. We prove that there is an equivalence of categories between the category of $\omega$-semi-stable (poly-stable) Higgs bundles $(E, \overline{\partial}_{E}, \phi )$ with $(2rc_{2}(E)-(r-1)c_{1}^{2}(E))\cdot [\omega]^{n-2}=0
$  and the category of (semi-simple) projectively flat bundles $(E, D)$ with $\sqrt{-1}F_{D}=\alpha \otimes \Id_{E}$ for some real (1,1)-form $\alpha$. Furthermore, we also establish the above correspondence on some compact non-K\"ahler manifolds. As its application, we obtain a vanishing theorem of characteristic classes of  projectively flat bundles.
\\[1cm]


\newpage

\section{Introduction}

Let $E$ be a complex vector bundle of rank $r$ over a compact complex manifold $X$. A connection $D$ in $E$ is said to be projectively flat if the induced connection in the principal $PGL(r; \mathbb{C})$-bundle is flat, or equivalently, the curvature $F_{D}$ takes values in scalar multiples of the identity endomorphism of $E$. The bundle $E$ is projectively flat if and only if it admits a projectively flat connection. The moduli space of fundamental group representations and the moduli space of flat vector bundles are related by the Riemann-Hilbert correspondence. A Higgs bundle $(E, \overline{\partial}_{E}, \theta )$ is a holomorphic bundle $(E, \overline{\partial}_{E})$ coupled with a Higgs field $\theta \in \Omega^{1, 0}(\mbox{End}E)$ satisfying $\overline{\partial}_{E} \theta =0$ and $\theta \wedge \theta =0 $. Higgs bundles were introduced by Hitchin (\cite{H}) and developed by Simpson (\cite{S1,S2}). Under the assumption that $(E, D)$ is a semi-simple flat bundle over a compact K\"ahler manifold $(M, \omega )$, the theorem of
 Corlette (\cite{Cor}) and Donaldson (\cite{Don3}) on the existence of harmonic metric implies that there exists a poly-stable Higgs structure $ (\overline{\partial}_{E}, \theta )$ on $E$. On the other hand, by the work of  Hitchin (\cite{H}) and Simpson (\cite{S2}) on the Hitchin-Kobayashi correspondence for Higgs bundles, one has a  correspondence between  the moduli space for semi-simple flat bundles and the moduli space for poly-stable Higgs bundles with vanishing Chern numbers. The motivation of this paper is to extend the Corlette-Donaldson-Hitchin-Simpson's correspondence to the projectively flat bundles case and the non-K\"ahler case.

Suppose $E$ is a projectively flat complex vector bundle over a compact complex manifold $X$ with rank $r$ and $c_{1}(E)\cap \Omega^{1,1}(X, \mathbb{C})\cap \Omega^{2}(X, \mathbb{R})\neq \emptyset $. Then we have a connection $D$ in $E$ such that its curvature satisfies
\begin{equation}\label{condition1}
\sqrt{-1}F_{D}=\alpha \otimes \Id_{E},
\end{equation}
where $\alpha \in \frac{2\pi}{ r} c_{1}(E)$ is a real $(1, 1)$-form.
A subbundle $S\subset E$ is said to be $D$-invariant if $De\in \Omega^{1}(X, S)$ for any $e\in \Omega^{0}(X, S)$. We say $(E,D)$ is simple if it has no proper $D$-invariant subbundle and  $(E,D)$ is semi-simple if it is a direct sum of $D$-invariant subbundles.  Given a Hermitian metric $H$ on $E$, there is a unique decomposition
\begin{equation}
D=D_{H}+\psi_{H},
\end{equation}
where $D_{H}$ is a unitary connection and $\psi_{H}\in \Omega^{1}(\mbox{End}(E))$ is self-adjoint with respect to $H$.
In the following, we denote the $(1,0)$-part (resp. $(0,1)$-part) of $D_{H}$ by $\partial_{H}$ (resp. $\bar{\partial}_{H}$) and define
\begin{equation}\label{de1}
D_{H}^{''}=\bar{\partial}_{H}+\psi_{H}^{1,0},\ \ D^{'}_{H}=\partial_{H}+\psi_{H}^{0,1},
\end{equation}
\begin{equation}\label{de2}
G_{H}=(D_{H}^{''})^{2}=\bar{\partial}_{H}^{2}+\bar{\partial}_{H}\psi_{H}^{1,0}+\psi_{H}^{1,0}\wedge\psi_{H}^{1,0}.
\end{equation}
One can check that $\sqrt{-1}\tr(\bar{\partial}_{H}\psi_{H}^{1,0})=\sqrt{-1}\tr G_{H}$ is a real $\partial$-closed and $\bar{\partial}$-closed $(1,1)$-form. Furthermore, for any other Hermitian metric $K$ on $E$, we have
\begin{equation}\label{GHK01}
G_{H}-G_{K}=\frac{1}{4}D(h^{-1}D_{K}^{c}(h)),
\end{equation}
and
\begin{equation}
\sqrt{-1}\tr(\bar{\partial}_{H}\psi_{H}^{1,0})-\sqrt{-1}\tr(\bar{\partial}_{K}\psi_{K}^{1,0})=\frac{\sqrt{-1}}{2}\partial \bar{\partial} \log \det(h),
\end{equation}
where $h=K^{-1}H$, $D_{K}^{c}=D_{K}^{''}-D_{K}^{'}$. Hence we can define a fixed class in the Bott-Chern cohomology,
\begin{equation}
BC(E,D)=[\sqrt{-1}\tr(\bar{\partial}_{H}\psi_{H}^{1,0})]_{BC}\in H_{BC}^{1,1}(X,\mathbb{R}).
\end{equation}
Since $\sqrt{-1}\tr(\bar{\partial}_{H}\psi_{H}^{1,0})$ is $d$-exact, if $X$ satisfies the $\partial \bar{\partial}$-lemma (specially $X$ is K\"ahler), then
\begin{equation}
BC(E,D)=0.
\end{equation}
A natural question arises as follows. For  any real $(1,1)$-from $\zeta \in BC(E, D)$, does there exist a Hermitian metric $H$ on $(E, D)$ such that
\begin{equation}\label{Q1}
\sqrt{-1}G_{H} =\frac{1}{r}\zeta \otimes \Id_{E}?
\end{equation}
In particular, if $BC(E, D)=0$, does there exist a Hermitian metric $H$ on $(E, D)$ such that $G_{H}=0$?
To solve this problem, we obtain the following theorem.

\medskip

\begin{theorem}\label{thm0}
Let $(X, \omega )$ be a compact Hermitian manifold and $(E, D)$  a simple projectively flat complex vector bundle over $X$ satisfying the condition (\ref{condition1}). Then there must exist a unique Hermitian metric $H$ on $E$ such that
\begin{equation}\label{h0}
\sqrt{-1}\Lambda_{\omega }G_{H}^{\perp}=0,
\end{equation}
where $G_{H}^{\perp}$ is the trace-free part of $G_{H}$, i.e. $G_{H}^{\perp}=G_{H}-\frac{1}{r}\tr G_{H} \otimes \Id_{E}$. Furthermore, if $\omega $ is astheno-K\"ahler, i.e. $\partial\bar{\partial}\omega^{n-2}=0$, then we have
\begin{equation}
G_{H}^{\perp}=0.
\end{equation}
\end{theorem}

\medskip

\begin{remark}
Given a real $(1,1)$-form $\zeta \in BC(E, D)$, there exists a real function $\varphi $ such that $\frac{\sqrt{-1}}{r}\tr G_{H}=\frac{1}{r}\zeta +\sqrt{-1}\partial \bar{\partial }\varphi $. Letting $\tilde{H}=e^{-2\varphi }H$, according to formula (\ref{GHK01}), we have
\begin{equation}
\sqrt{-1}G_{H}^{\perp}=\sqrt{-1}G_{\tilde{H}}-\frac{1}{r}\zeta \otimes \Id_{E}.
\end{equation}
Thus, by Theorem \ref{thm0} and conformal transformation, (\ref{Q1}) can be solved if $X$ admits an  astheno-K\"ahler metric $\omega$.
\end{remark}

A Hermitian metric $H$ on $(E, D)$ is called harmonic if it satisfies:
\begin{equation}\label{h1}
\sqrt{-1}\Lambda _{\omega} G_{H}=0.
\end{equation}
 In the case that $(X, \omega )$ is K\"ahler  and $D$ is a flat connection, (\ref{h1}) is equivalent to $D_{H}^{\ast }\psi_{H}=0$,   the existence of harmonic metric was proved by Donaldson (\cite{Don3}) and Corlette (\cite{Cor}). If $BC(E, D)=0$ and $(X, \omega )$ is astheno-K\"ahler, Theorem \ref{thm0} implies that $(E, D)$ must admit a harmonic metric $H$ with $G_{H}=0$. Different to Corlette and Donaldson's arguments by using the heat flow, we will use the continuous method to solve the equation (\ref{h0}), see section 3 for details.

\medskip

In this paper, we also study some characteristic classes of projectively flat bundle $(E, D)$ with $\sqrt{-1}F_{D}=\alpha \otimes \Id_{E}$, where $\alpha $ is a real $(1, 1)$-form. Set
\begin{equation}
v_{2j+1}(E, D, H)=(2\pi \sqrt{-1})^{-j}\tr \psi_{H}^{2j+1}.
\end{equation}
Like that in the flat bundle case, one can check that $v_{2j+1}(E, D, H)$ is a real $d$-closed form and the class $[v_{2j+1}(E, D, H)]\in H^{2j+1}_{DR}(X, \mathbb{R})$ is independent of the choice of Hermitian metrics. So we define the following class:
\begin{equation}
v_{2j+1}(E, D)=[(2\pi \sqrt{-1})^{-j}\tr \psi_{H}^{2j+1}]
\end{equation}
for $j\geq 0$ . In the flat bundle case, this class was  defined by Kamber-Tondeur and Dupont (\cite{KT, DC}), and its Chern-Weil type description was given by Bismut and Lott(\cite{BL}). When $(X, \omega )$ is a K\"ahler manifold and $(E, D)$ is a semi-simple flat bundle, Reznikov (\cite{R}) proved that the classes $v_{2j+1}(E, D)$ necessarily vanish for $j\geq 1$ (a simple proof can be found in \cite{Ko}). As an application of Theorem \ref{thm0},  we generalized Reznikov's result to the projectively flat bundles on astheno-K\"ahler manifolds. In fact, we obtain the following theorem.

\medskip

\begin{theorem}\label{thm:2-4}
Let $(E, D)$ be a projectively flat complex vector bundle over a compact complex manifold $X$ with $\sqrt{-1}F_{D}=\alpha \otimes \Id_{E}$, where $\alpha $ is a real $(1, 1)$-form on $X$. If $X$ admits an  astheno-K\"ahler metric $\omega$, then for any $j\geq 1$, the class $v_{2j+1}(E, D)$ must vanish.
\end{theorem}

\medskip

 In the following, we denote the space of connections in $E$ by $\mathcal{A}_{E}$, and set the moduli spaces
\begin{equation}
\mathcal{C}_{DR}(E, \alpha )=\{D\in \mathcal{A}_{E}\ | \ \sqrt{-1}F_{D}=\alpha \otimes \Id_{E} \}/\mathcal{G},
\end{equation}
 \begin{equation}
\mathcal{C}_{DR}^{s}(E, \alpha )=\{\text{Semi-simple}\ D\in \mathcal{A}_{E}\ | \ \sqrt{-1}F_{D}=\alpha \otimes \Id_{E} \}/\mathcal{G},
\end{equation}
where $\mathcal{G}$ is the gauge group of $E$.

Now, we suppose that $(X,\omega)$ is a compact Hermitian manifold of complex dimension $n$, and  the Hermitian metric $\omega $  satisfies the Gauduchon and astheno-K\"ahler conditions, i.e. $\partial\bar{\partial}\omega^{n-1}=\partial\bar{\partial}\omega^{n-2}=0$. These conditions ensure that the first and second Chern numbers of holomorphic vector  bundles are well-defined. Under the Gauduchon condition, we can also assume that $\Lambda_{\omega}\alpha =\lambda $, where $\lambda $ is a real number.
  Adding another condition that $\int_{X}\partial [\eta]\wedge\frac{\omega^{n-1}}{(n-1)!}=0$ for any Dolbeault class $[\eta]\in H^{0,1}(X)$, one can easily check that $X$ satisfies the global $\partial \bar{\partial }$-lemma, and then $BC(E, D)=0$.  By Theroem \ref{thm0}, there must exist a harmonic metric $H$ on the semi-simple projectively flat bundle $(E, D)$, and then $G_{H}=0$.
Denote $\overline{\partial }_{E}=\bar{\partial}_{H}$ and $\theta = \psi_{H}^{1,0}$, then we obtain a Higgs bundle structure $(E, \overline{\partial }_{E}, \theta )$ with
\begin{equation}\label{HE0}
\sqrt{-1}(F_{H}+[\theta,\theta^{*H}])=\alpha \otimes \Id_{E}, \quad \partial_{H}\theta=0.
\end{equation}
 It is not hard to find that (\ref{HE0}) means the Higgs bundle $(E, \overline{\partial }_{E}, \theta )$ is poly-stable and
\begin{equation}\label{CC}
\Delta (E,\bar{\partial}_{E})\cdot [\omega^{n-2}]:=(c_{2}(E,\bar{\partial}_{E})-\frac{r-1}{2r}c_{1}^{2}(E,\bar{\partial}_{E}))\cdot [\omega^{n-2}]=0.
\end{equation}

 In the following,
 $\mathcal{C}_{Dol}^{s}(E)$ (resp. $\mathcal{C}_{Dol}(E)$) denotes the moduli space of polystable (resp. semi-stable) Higgs structures $(\overline{\partial }_{E}, \theta )$  on bundle $E$ with the vanishing Chern number (\ref{CC}).
 Hitchin (\cite{H}) and Simpson (\cite{S1}) proved the Higgs version of Donaldson-Uhlenbeck-Yau theorem (\cite{NarSe,Don2,UhYau}) which states that every  poly-stable Higgs bundle over a K\"ahler manifold must admit a Hermitian-Einstein metric $H$, i.e.
  \begin{equation}
  \sqrt{-1}\Lambda_{\omega }(F_{H}+[\theta , \theta ^{\ast H}])=\lambda \Id_{E},
  \end{equation}
  and it is also valid for compact Gauduchon manifolds (\cite{bu,LY,MA}). Therefore, for any poly-stable Higgs structure $(\overline{\partial }_{E}, \theta ) $ with the vanishing Chern number (\ref{CC}), the Hitchin-Simpson connection $D_{H,\theta}=D_{H}+\theta+\theta^{*H}$ with respect to the Hermitian-Einstein metric $H$ must be  projectively flat. Then one can obtain a correspondence between  $\mathcal{C}_{Dol}^{s}(E)$ and $\mathcal{C}_{DR}^{s}(E, \alpha )$, i.e. we proved the following theorem.

 \medskip

\begin{theorem}\label{thm:2-1}
Let $(X,\omega)$ be a compact Hermitian manifold of dimension $n$ satisfying $\partial\bar{\partial}\omega^{n-1}=\partial\bar{\partial}\omega^{n-2}=0$ and $\int_{X}\partial [\eta]\wedge\frac{\omega^{n-1}}{(n-1)!}=0$ for any Dolbeault class $[\eta]\in H^{0,1}(X)$, $E$ be a complex vector bundle over $X$ with rank $r$ and $c_{1}(E)\cap \Omega^{1,1}(M, \mathbb{C})\cap \Omega^{2}(M, \mathbb{R})\neq \emptyset $, i.e. there is a real $(1, 1)$-form $\alpha \in \frac{2\pi}{ r} c_{1}(E)$.
Then there is a one-to-one correspondence between the moduli spaces $\mathcal{C}_{Dol}^{s}(E)$ and $\mathcal{C}_{DR}^{s}(E, \alpha )$.
\end{theorem}

\begin{remark}
 In \cite[Proposition 2.1]{OUV}, the authors constructed some nontrivial examples of Hermitian metric $\omega$ satisfying strongly Gauduchon and astheno-K\"ahler conditions. It is easy to see that these non-K\"ahler metrics must satisfy the conditions in Theorem (\ref{thm:2-1}).
\end{remark}

 In \cite{S2}, Simpson proved  that there is an equivalence of categories  between the category of semi-stable Higgs bundles with vanishing Chern numbers and the category of flat bundles over smooth projective varieties.  The key in the proof of Simpson's correspondence is to prove that every semi-stable Higgs bundle with vanishing Chern numbers is an extension of stable Higgs bundles with vanishing Chern numbers. This  has been generalized to the K\"ahler case by Nie and the third author (\cite{NZ1}). In this paper, we extend this key theorem to semi-stable Higgs bundles with the vanishing condition (\ref{CC}) and to the non-K\"ahler case.  We combine the continuous method and the heat flow method to solve this problem. The key technique here is to use $\epsilon$-regularity theorem (Theorem \ref{thm:5}) of the Hermitian-Yang-Mills flow. Let's first consider the solution $H_{\epsilon}$ of  the perturbed equation (\ref{eq}) for every $0<\epsilon \leq 1$. In \cite{NZ}, under the assumption that the Higgs bundle $(E,\bar{\partial}_{E},\theta)$ is semistable, Nie and the third author have proved that \begin{equation}\sup_X|\sqrt{-1}\Lambda_{\omega }(F_{H_{\epsilon}}+[\theta,\theta^{*{H_{\epsilon}}}])-\lambda \cdot \textmd{Id}_E|_{H_{\epsilon}}\rightarrow 0\end{equation} as $\epsilon \rightarrow 0$. The Chern-Weil theory and the  vanishing  condition (\ref{CC}) imply (see (\ref{Bo1}))
 \begin{equation}
2||\partial_{H_{\epsilon}}\theta ||_{L^{2}}^{2}+||(F_{H_{\epsilon},\theta}^{1,1})^{\perp}||_{L^{2}}^{2}\rightarrow 0,
\end{equation}
as $\epsilon \rightarrow 0$, where $(F_{H,\theta}^{1,1})^{\perp}$ is the trace free part of $F_{H,\theta}^{1,1}$. Now, let's  evolve the metrics $H_{\epsilon}$ along the Hermitian-Yang-Mills flow (\ref{hymf2}), and denote $H_{\epsilon}(t)$ is the long time solution of (\ref{hymf2}).  By the $\epsilon$-regularity theorem (Theorem \ref{thm:5}) and the parabolic inequality (\ref{k02}), we can show that for some positive $t_{0}$,
 \begin{equation}
2||\partial_{H_{\epsilon}(t_{0})}\theta ||_{L^{\infty }}^{2}+||(F_{H_{\epsilon}(t_{0}),\theta}^{1,1})^{\perp}||_{L^{\infty }}^{2}\rightarrow 0,
\end{equation}
 as $\epsilon \rightarrow 0$, i.e. we obtain the existence of approximate Higgs-projectively flat structure. Using this and following  Demailly-Peternell-Schneider's argument (\cite{DPM}), we can show that the  Harder-Narasimhan-Seshadri filtration of the Higgs bundle $(E,\bar{\partial}_{E},\theta)$ is by subbundles. So
 we prove that, on some non-K\"ahler manifolds, every semi-stable Higgs bundle with the vanishing  condition (\ref{CC}) must be an extension of stable Higgs bundles, i.e. we conclude the following theorem.

\begin{theorem}\label{thm:1}
Let $(X,\omega)$ be a compact Hermitian manifold of dimension $n$ satisfying $\partial\bar{\partial}\omega^{n-1}=\partial\bar{\partial}\omega^{n-2}=0$, and $\mathfrak{E}=(E,\bar{\partial}_{E},\theta)$ be a Higgs bundle over $X$. Then $\mathfrak{E}$ is semi-stable with the vanishing  condition (\ref{CC}) if and only if $\mathfrak{E}$ admits a filtration
\begin{equation}
0=\mathfrak{E}_{0}\subset \mathfrak{E}_{1}\subset\cdots\subset \mathfrak{E}_{l}=\mathfrak{E}
\end{equation}
by Higgs subbundles such that the quotients $\mathfrak{Q}_{k}=(Q_{k}, \overline{\partial}_{Q_{k}}, \theta_{k})=\mathfrak{E}_{k}/\mathfrak{E}_{k-1}$ are stable Higgs bundles with vanishing Chern number (\ref{CC}) and $\frac{1}{\mbox{rank}(Q_{k})}c_{1}(Q_{k}, \overline{\partial}_{Q_{k}})=\frac{1}{\mbox{rank}(E)}c_{1}(E, \overline{\partial}_{E})$ for every $0\leq k\leq l-1$.
\end{theorem}

As an application, we establish a correspondence between  $\mathcal{C}_{Dol}(E)$ and $\mathcal{C}_{DR}(E, \alpha )$, i.e. we prove the following theorem.

\begin{theorem}\label{thm:2-2}
Let $(X,\omega)$ be a compact Hermitian manifold of dimension $n$ satisfying $\partial\bar{\partial}\omega^{n-1}=\partial\bar{\partial}\omega^{n-2}=0$ and $\int_{X}\partial [\eta]\wedge\frac{\omega^{n-1}}{(n-1)!}=0$ for any Dolbeault class $[\eta]\in H^{0,1}(X)$, $E$ be a complex vector bundle over $X$ with rank $r$ and $c_{1}(E)\cap \Omega^{1,1}(M, \mathbb{C})\cap \Omega^{2}(M, \mathbb{R})\neq \emptyset $, i.e. there is a real $(1, 1)$-form $\alpha \in \frac{2\pi}{ r} c_{1}(E)$.
Then there is a one-to-one correspondence between the moduli spaces $\mathcal{C}_{Dol}(E)$ and $\mathcal{C}_{DR}(E, \alpha )$ in the case $\mbox{rank}(E)=2$.
\end{theorem}

We also find a nature one-to-one map between $\mathcal{C}_{Dol}(E)$ and $\mathcal{C}_{DR}(E, \alpha )$ for higher rank case, but as pointed in Remark \ref{rmk:7-1},  this map only exists in K\"ahler case, i.e. we deduce the following result.

\begin{theorem}\label{thm:2-3}
Let $(X,\omega)$ be a K\"ahler manifold of dimension $n$. Then there is a one-to-one map between $\mathcal{C}_{Dol}(E)$ and $\mathcal{C}_{DR}(E, \alpha )$.
\end{theorem}

Actually, using Theorem \ref{thm:1}, one may follow Simpson's argument in (\cite{S2}) to get a correspondence between $\mathcal{C}_{Dol}(E)$ and $\mathcal{C}_{DR}(E, \alpha )$ over K\"ahler manifolds. However, Simpson's argument (\cite{S2}) is totally algebraic and highly abstract. We construct the correspondence map more directly, so we write down here. It should be pointed out that, by using Nie and the third author's result (\cite{NZ1}), Deng also gives a constructing proof of this correspondence for the flat bundle case in his thesis (\cite{De}).

\medskip

This paper is organized as follows. In Section 2, we introduce some basic concepts and results about projectively flat bundles. In section 3, we study the equation (\ref{h0}) and give a proof of Theorem \ref{thm0}. In Section 4, we consider the characteristic classes of projectively flat bundles and give a proof of Theorem  \ref{thm:2-4}. In Section 5, we recall some basic results of Higgs bundles over non-K\"ahler manifolds and give a proof of Theorem \ref{thm:2-1}. In Section 6, we give some basic estimates and an $\epsilon$-regularity theorem about the  Yang-Mills-Higgs flow in the non-K\"ahler case. In Section 7, we  prove Theorem \ref{thm:1}. In Section 8 and Section 9, we  prove Theorem \ref{thm:2-2} and Theorem \ref{thm:2-3}.
\medskip

{\bf  Acknowledgement:}  The authors would like to thank Professor Jixiang Fu and Professor Xiangwen Zhang to point out the examples satisfying our assumptions in Theorem \ref{thm:2-1}. The authors are partially supported by NSF in China No.11625106, 11571332 and 11721101. The second author is also supported by NSF in China No.11801535, the China Postdoctoral Science Foundation (No.2018M642515) and the Fundamental Research Funds for the Central Universities.

\medskip

\section{Preliminaries}

Let $(X, \omega )$ be a compact Hermitian manifold of dimension $n$, $E$ be a projectively flat bundle over $X$ with  rank $r$ and $c_{1}(E)\cap \Omega^{1,1}(M, \mathbb{C})\cap \Omega^{2}(M, \mathbb{R})\neq \emptyset $.   We can choose a real $(1, 1)$-form $\alpha \in \frac{2\pi}{ r} c_{1}(E)$. Furthermore, if $\partial\bar{\partial}\omega^{n-1}=0$, we can assume
\begin{equation}\Lambda_{\omega}\alpha =\lambda , \end{equation} where $\lambda $ is a real number. It is easy to check that there is a connection in $E$ such that:
\begin{equation}\label{condition2}
\sqrt{-1}F_{D}=\alpha \otimes \Id_{E}
\end{equation}
and
\begin{equation}
2c_{2}(E, D)-\frac{r-1}{r}c_{1}^{2}(E,D)=0.
\end{equation}

For any metric $H$ on $(E, D)$, there is a unique decomposition
\begin{equation}
D=D_{H}+\psi_{H},
\end{equation}
where $D_{H}$ is a unitary connection, $\psi_{H}\in \Omega^{1}(\mbox{End}(E))$ is self-adjoint and
\begin{equation}\label{ex1}
H(\psi_{H}e_{1}, e_{2})=\frac{1}{2}\{H(De_{1}, e_{2})+H(e_{1}, De_{2})-dH(e_{1}, e_{2})\}
\end{equation}
for any $e_{1}, e_{2}\in \Gamma (E)$.
 We rewrite (\ref{condition2}) as
\begin{equation}\label{condition3}
\alpha \otimes \Id_{E}=\sqrt{-1}F_{D}=\sqrt{-1}(D_{H}^{2}+\psi_{H}\wedge \psi_{H}+D_{H}\circ \psi_{H}+ \psi_{H}\circ D_{H}).
\end{equation}
Considering the self-adjoint and anti-self-adjoint parts of the above identity, since $\alpha $ is a real $(1,1)$-form, we have
\begin{equation}\label{co1}
D_{H}(\psi_{H})=0,
\end{equation}
and
\begin{equation}
\sqrt{-1}(D_{H}^{2}+\psi_{H}\wedge \psi_{H})=\alpha \otimes \Id_{E}.
\end{equation}
Then, we get
\begin{equation}\label{dd:1}
\left\{\begin{split}
&\bar{\partial}_{H}^{2}+\psi_{H}^{0,1}\wedge\psi_{H}^{0,1}=0, \quad \partial_{H}^{2}+\psi_{H}^{1, 0}\wedge\psi_{H}^{1,0}=0,\\
&\bar{\partial}_{H}\psi_{H}^{1,0}+\partial_{H}\psi_{H}^{0,1}=0,\\
&\partial_{H}\psi_{H}^{1,0}=0, \quad \bar{\partial}_{H}\psi_{H}^{0,1}=0, \\
&\partial_{H}\bar{\partial}_{H}+\bar{\partial}_{H}\partial_{H}+[\psi_{H}^{1,0},\psi_{H}^{0,1}]=-\sqrt{-1}\alpha \otimes \Id_{E}.
\end{split}\right.
\end{equation}

Denote by $\partial_{H}$ (resp. $\bar{\partial}_{H}$) the $(1,0)$-part (resp. $(0,1)$-part) of $D_{H}$. Define $D_{H}^{''}$, $D_{H}^{'}$ and $G_{H}$ as that in (\ref{de1}) and (\ref{de2}). Obviously (\ref{dd:1}) implies
\begin{equation}\label{bc1}
(G_{H}^{1,1})^{\ast H}=-G_{H}^{1,1}, \quad (G_{H}^{2,0})^{\ast H}=G_{H}^{0,2},
\end{equation}
\begin{equation}\label{bc2}
\tr G_{H}=\tr (\bar{\partial }_{H}\psi_{H}^{1,0})=\bar{\partial}\tr (\psi_{H}^{1,0})=d \tr (\psi_{H}^{1,0}),
\end{equation}
\begin{equation}\label{bc3}
\partial \tr (\bar{\partial }_{H}\psi_{H}^{1,0})=-\bar{\partial} \tr (\partial_{H}\psi_{H}^{1,0})=0.
\end{equation}
If $G_{H}=0$, then $(E,\bar{\partial}_{H},\psi_{H}^{1,0})$ will be a Higgs bundle and $H$ is a Higgs-Hermitian projectively flat metric.

\begin{lemma}\label{lem:x1}
Assume $\bar{\partial}\partial\omega^{n-2}=0$, then
\begin{equation}\sqrt{-1}\Lambda_{\omega}(G_{H})=0\iff G_{H}=0\end{equation}
and
\begin{equation}\sqrt{-1}\Lambda_{\omega}(G_{H}^{\perp})=0\iff G_{H}^{\perp}=0.\end{equation}
\end{lemma}
\begin{proof}
The Riemann bilinear relations assert that
\begin{equation}
\int_{X}\tr(G_{H}^{\perp}\wedge G_{H}^{\perp})\wedge\frac{\omega^{n-2}}{(n-2)!}
=\int_{X}(|G_{H}^{\perp}|^{2}-|\Lambda_{\omega}G_{H}^{\perp}|^{2})\frac{\omega^{n}}{n!}.
\end{equation}
On the other hand, by (\ref{dd:1}), we have
\begin{equation}
\begin{split}
&\tr (\bar{\partial }_{H}\psi_{H}^{1,0})\wedge \tr (\bar{\partial }_{H}\psi_{H}^{1,0})\\
=&-\tr (\partial _{H}\psi_{H}^{0, 1})\wedge \tr (\bar{\partial }_{H}\psi_{H}^{1,0})\\
=&\partial \bar{\partial }(\tr \psi_{H}^{0, 1} \wedge \tr \psi_{H}^{1, 0} ),\\
\end{split}
\end{equation}
and
\begin{equation}
\begin{split}
&\int_{X}\tr(G_{H}^{\perp}\wedge G_{H}^{\perp})\wedge\frac{\omega^{n-2}}{(n-2)!}\\
=&\int_{X}\tr(G_{H}\wedge G_{H})\wedge\frac{\omega^{n-2}}{(n-2)!}\\
=&\int_{X}\{-2\tr(\psi_{H}^{1,0}\wedge\psi_{H}^{1,0}\wedge\psi_{H}^{0,1}\wedge\psi_{H}^{0,1})-\tr(\bar{\partial}_{H}\psi_{H}^{1,0}\wedge\partial_{H}\psi_{H}^{0,1})\}\wedge\frac{\omega^{n-2}}{(n-2)!}\\
=&\int_{X}\{\bar{\partial}\partial \tr(\psi_{H}^{1,0}\wedge\psi_{H}^{0,1})-2\tr(\psi_{H}^{1,0}\wedge\psi_{H}^{1,0}\wedge\psi_{H}^{0,1}\wedge\psi_{H}^{0,1})\}\wedge\frac{\omega^{n-2}}{(n-2)!}\\
 & +\int_{X}\tr(\partial_{H}(\bar{\partial}_{H}\psi_{H}^{1,0})\wedge\psi_{H}^{0,1})\wedge\frac{\omega^{n-2}}{(n-2)!}\\
=& -\int_{X}\tr((\sqrt{-1}\alpha \otimes \Id_{E}\wedge \psi_{H}^{1,0}-\psi_{H}^{1,0}\wedge (\sqrt{-1}\alpha \otimes \Id_{E}))\wedge\psi_{H}^{0,1})\wedge\frac{\omega^{n-2}}{(n-2)!}\\
=&0.
\end{split}
\end{equation}
\end{proof}


For any two Hermitian metrics $H$ and $K$, let $h=K^{-1}H$ and $D_{K}^{c}=D_{K}^{''}-D_{K}^{'}$. By the expression (\ref{ex1}), one can deduce that
\begin{equation}\label{ex2}
\psi_{H}=\frac{1}{2}h^{-1}\circ \psi_{K}\circ h +\frac{1}{2}\psi_{K}+\frac{1}{2}(D_{K}-h^{-1}\circ D_{K} \circ h),
\end{equation}
and
\begin{equation}
D_{H}^{''}-D_{K}^{''}=\frac{1}{2}h^{-1}D_{K}^{c}h.
\end{equation}
Using the equalities (\ref{dd:1}), we know
\begin{equation}
D_{K}^{''}\circ D_{K}^{''}+ D_{K}^{'}\circ D_{K}^{'}=0,
\end{equation}
\begin{equation}
D_{K}^{''}\circ D_{K}^{'}+ D_{K}^{''}\circ D_{K}^{'}=-\sqrt{-1}\alpha \otimes \Id_{E},
\end{equation}
and
\begin{equation}
D_{K}^{c}( D_{K}^{c}h)=D_{K}^{c}\circ D_{K}^{c}\circ h- h\circ D_{K}^{c}\circ D_{K}^{c}=\sqrt{-1}(\alpha  h-h\alpha)=0.
\end{equation}
Then
\begin{equation}\label{dd:2}
\begin{split}
&G_{H}-G_{K}=D_{H}^{''}\circ D_{H}^{''}-D_{K}^{''}\circ D_{K}^{''}\\
=&\frac{1}{4}\{2D_{K}^{''}(h^{-1}D_{K}^{c}h)-D_{K}^{c}(h^{-1}D_{K}^{c}h)+h^{-1}D_{K}^{c}(D_{K}^{c}h)\}\\
=&\frac{1}{4}D(h^{-1}D_{K}^{c}h).\\
\end{split}
\end{equation}
This together with (\ref{ex2}) gives us that
\begin{equation}\label{bc4}
\tr (\psi_{H})=\tr (\psi_{K})-\frac{1}{2}d\log \det(h)
\end{equation}
and
\begin{equation}\label{bc5}
\tr(\bar{\partial}_{H}\psi_{H}^{1,0})-\tr(\bar{\partial}_{K}\psi_{K}^{1,0})=\tr(G_{H})-\tr(G_{K})=\frac{1}{2}\partial\bar{\partial}\log\det h.
\end{equation}
By (\ref{bc1}), (\ref{bc2}), (\ref{bc3}) and (\ref{bc5}),
we can define the following class in the Bott-Chern cohomology.

\medskip

\begin{definition} Let $(E, D)$ be a projectively flat complex vector bundle satisfying the condition (\ref{condition1}) over a compact complex manifold $X$. We define
\begin{equation}
BC(E,D)=[\sqrt{-1}\tr(\bar{\partial}_{H}\psi_{H}^{1,0})]_{BC}\in H_{BC}^{1,1}(X,\mathbb{R}).
\end{equation}
\end{definition}

\medskip

Set $s=\log(h)=\log (K^{-1}H)$, then by a similar discussion with \cite{NZ} (or \cite{ZZZ}), we have
\begin{prop}\label{d1:3}
\begin{equation}\label{key1}
\begin{split}
\langle\sqrt{-1}\Lambda_{\omega}D(h^{-1}D_{K}^{c}h),s\rangle_{K}=-\langle\Theta(s)(Ds),Ds\rangle_{K}+\sqrt{-1}\Lambda_{\omega}\partial\bar{\partial}|s|_{K}^{2},
\end{split}
\end{equation}
where
\begin{equation}
\Theta(x,y)=
\left\{\begin{split}
&\frac{e^{y-x}-1}{y-x},   &x\neq y;\\
&\ \ \ \  1,  &x=y.
\end{split}\right.
\end{equation}
\end{prop}
\begin{proof}
In an open dense subset $U$ of $X$, choosing a suitable orthonormal basis $\{e_{1},\cdots,e_{r}\}$ of $E$ with respect to the metric $K$, we can assume
\begin{equation}
h=\sum_{\alpha=1}^{r}e^{\lambda_{\alpha}}e_{\alpha}\otimes e^{\alpha}
\end{equation}
and
\begin{equation}
s=\sum_{\alpha=1}^{r}\lambda_{\alpha}e_{\alpha}\otimes e^{\alpha}.
\end{equation}
Let $D=d+A$, then $D_{K}=d+\frac{1}{2}(A-\bar{A}^{T})$ and $\psi_{K}=\frac{1}{2}(A+\bar{A}^{T})$. Denote the $(1,0)$-part and $(0,1)$-part of $A$ by $B$ and $C$. Computing directly shows
\begin{equation}
\begin{split}
h^{-1}D_{K}^{c}h&=\sum_{\alpha=1}^{r}d^{c}\lambda_{\alpha}e_{\alpha}\otimes e^{\alpha}+\sum_{\alpha\neq\beta}(e^{\lambda_{\beta}-\lambda_{\alpha}}-1)(\overline{C_{\alpha}^{\beta}}-\overline{B_{\alpha}^{\beta}})e_{\alpha}\otimes e^{\beta},\\
Ds&=\sum_{\alpha=1}^{r}d\lambda_{\alpha}e_{\alpha}\otimes e^{\alpha}+\sum_{\alpha\neq\beta}(\lambda_{\beta}-\lambda_{\alpha})(B_{\beta}^{\alpha}+C_{\beta}^{\alpha})e_{\alpha}\otimes e^{\beta}.
\end{split}
\end{equation}
Then
\begin{equation}
\begin{split}
\sqrt{-1}\Lambda_{\omega}d\tr(h^{-1}D_{K}^{c}hs)=&\sum_{\alpha=1}^{r}\sqrt{-1}\Lambda_{\omega}d(\lambda_{\alpha}d^{c}\lambda_{\alpha})\\
=&\sum_{\alpha=1}^{r}2\sqrt{-1}\Lambda_{\omega}\lambda_{\alpha}\partial\bar{\partial}\lambda_{\alpha}+\sum_{\alpha=1}^{r}2|\partial\lambda_{\alpha}|^{2}\\
=&\sqrt{-1}\Lambda_{\omega}\partial\bar{\partial}|s|_{K}^{2}
\end{split}
\end{equation}
and
\begin{equation}
\begin{split}
&\sqrt{-1}\Lambda_{\omega}\tr(h^{-1}D_{K}^{c}h\wedge Ds)\\
=&\sum_{\alpha=1}^{r}\sqrt{-1}\Lambda_{\omega}(d^{c}\lambda_{\alpha}\wedge d\lambda_{\alpha})\\
&+\sum_{\alpha\neq\beta}(e^{\lambda_{\beta}-\lambda_{\alpha}}-1)(\lambda_{\alpha}-\lambda_{\beta})\sqrt{-1}\Lambda_{\omega}(\overline{C_{\alpha}^{\beta}}-\overline{B_{\alpha}^{\beta}})\wedge(B^{\beta}_{\alpha}+C_{\alpha}^{\beta})\\
=&-\sum_{\alpha=1}^{r}|d\lambda_{\alpha}|^{2}-\sum_{\alpha\neq\beta}\frac{(e^{\lambda_{\beta}-\lambda_{\alpha}}-1)}{(\lambda_{\beta}-\lambda_{\alpha})}(\lambda_{\beta}-\lambda_{\alpha})^{2}(|B_{\alpha}^{\beta}|^{2}+|C_{\alpha}^{\beta}|^{2})\\
=&-\sum_{\alpha,\beta}\Theta(\lambda_{\alpha},\lambda_{\beta})|(Ds)_{\alpha}^{\beta}|^{2}.
\end{split}
\end{equation}
Combining all of the above, we complete this proof.
\end{proof}

\begin{prop}\label{pp:0}For any real number $\lambda $, we have
\begin{equation}
\sqrt{-1}\Lambda_{\omega}dd^{c}\log(\tr(h)+\tr(h^{-1}))\geq -4(|\sqrt{-1}\Lambda_{\omega}G_{H}-\lambda \Id_{E}|_{H}+|\sqrt{-1}\Lambda_{\omega}G_{K}-\lambda \Id_{E}|_{K})
\end{equation}
and $\sup_{X}|\log(h)|_{K}\leq \grave{C}_{1}\|\log(h)\|_{L^{2}(X,K)}+\grave{C}_{2}$, where $\grave{C}_{1}$ and $\grave{C}_{2}$ are positive constants depending only on $\sup_{X}|\sqrt{-1}\Lambda_{\omega}G_{H}-\lambda \Id_{E}|_{H}$, $\sup_{X}|\sqrt{-1}\Lambda_{\omega}G_{K}-\lambda \Id_{E}|_{K}$ and the geometry of $(X, \omega)$.
\end{prop}
\begin{proof}
Clearly (\ref{dd:2}) yields
\begin{equation}\label{unique00}
\begin{split}
h(\sqrt{-1}\Lambda_{\omega}(G_{H}-G_{K}))=&\frac{1}{4}\sqrt{-1}\Lambda_{\omega}DD^{c}_{K}h-\frac{1}{4}\sqrt{-1}\Lambda_{\omega}D(h)h^{-1}D_{K}^{c}(h).
\end{split}
\end{equation}
Taking the trace on both sides, we see
\begin{equation}\label{unique10}
\begin{split}
&\frac{1}{4}\sqrt{-1}\Lambda_{\omega}dd^{c}\tr(h)-\frac{1}{4}|Dh\cdot h^{-1/2}|_{K}^{2}\\
\geq &-\tr(h)(|\sqrt{-1}\Lambda_{\omega}G_{H}-\lambda \Id_{E}|_{H}+|\sqrt{-1}\Lambda_{\omega}G_{K}-\lambda \Id_{E}|_{K}).
\end{split}
\end{equation}
On the other hand, it holds that
\begin{equation}
\begin{split}
&(\tr(h)+\tr(h^{-1}))\sqrt{-1}\Lambda_{\omega}dd^{c}\log(\tr(h)+\tr(h^{-1}))\\
=&-\frac{|\tr(Dh)+\tr(Dh^{-1})|^{2}}{\tr(h)+\tr(h^{-1})}+\sqrt{-1}\Lambda_{\omega}dd^{c}(\tr(h)+\tr(h^{-1})).
\end{split}
\end{equation}
According to Young's inequality, we have
\begin{equation}
\begin{split}
\frac{|\tr(Dh)+\tr(Dh^{-1})|^{2}}{\tr(h)+\tr(h^{-1})}\leq &\frac{(1+\varepsilon)|\tr(Dh)|^{2}+(1+\frac{1}{\varepsilon})|\tr(Dh^{-1})|^{2}}{\tr(h)+\tr(h^{-1})}.
\end{split}
\end{equation}
Take $\varepsilon=\tr(h^{-1})/\tr(h)$, then
\begin{equation}
\begin{split}
\frac{|\tr(Dh)+\tr(Dh^{-1})|^{2}}{\tr(h)+\tr(h^{-1})}\leq&|Dh\cdot h^{-1/2}|_{K}^{2}+|Dh^{-1}\cdot h^{1/2}|_{H}^{2} .
\end{split}
\end{equation}
So
\begin{equation}
\sqrt{-1}\Lambda_{\omega}dd^{c}\log(\tr(h)+\tr(h^{-1}))\geq -4(|\sqrt{-1}\Lambda_{\omega}G_{H}-\lambda \Id_{E}|_{H}+|\sqrt{-1}\Lambda_{\omega}G_{K}-\lambda \Id_{E}|_{K}).
\end{equation}
Noting
\begin{equation}
\log\big(\frac{1}{2r}(\tr(h)+\tr(h^{-1}))\big)\leq |\log(h)|_{K}\leq r^{1/2}\log(\tr(h)+\tr(h^{-1}))
\end{equation}
and applying Moser's iteration, we finish the proof.
\end{proof}

\medskip


\section{Proof of Theorem \ref{thm0}}
We first study the equation (\ref{h0}) by using the continuous method. Let $(E,D)$  be a rank $r$ projectively flat bundle over  a compact Hermitian manifold $(X, \omega )$,  and satisfy $\sqrt{-1}F_{D}=\alpha \otimes \Id_{E}$ for a real $(1, 1)$-form $\alpha $. Given a Hermitian metric $H$ on $E$, we set
\begin{equation}
\mbox{Herm}(E,H)=\{\eta\in\mbox{End}(E)\ | \ \eta^{*H}=\eta\}
\end{equation}
and
\begin{equation}
\mbox{Herm}^{+}(E,H)=\{\rho\in \mbox{Herm}(E,H)\ | \ \rho\ \text{is positive definite}\}.
\end{equation}

If the metric $\omega$ is Gauduchon, i.e. $\partial\bar{\partial}\omega^{n-1}=0$, we can define
\begin{equation}
\deg_{\omega}(E,D)=\int_{X}BC(E,D)\wedge\frac{\omega^{n-1}}{(n-1)!}.
\end{equation}
Using conformal transformation, we have the following lemma.

\medskip

\begin{lemma}\label{lemma31}
Suppose $(X, \omega )$ is a compact Gauduchon manifold, then there exists a Hermitian metric $\tilde{K}$ on $(E, D)$ such that \begin{equation}\tr(\sqrt{-1}\Lambda_{\omega}G_{\tilde{K}})=r\lambda ,\end{equation} where the constant $\lambda=\frac{\deg_{\omega }(E,D)}{r\Vol(X,\omega)}$.
\end{lemma}

\medskip

Let $K$ be a fixed background metric on $E$. Consider the following perturbed equation
\begin{equation}\label{dd:3}
L_{\epsilon}(h):=4\sqrt{-1}\Lambda_{\omega}G_{K}+\sqrt{-1}\Lambda_{\omega}D(h^{-1}D_{K}^{c}h)-\epsilon\log(h)-4\lambda \Id_{E}=0,
\end{equation}
where $h\in \mbox{Herm}^{+}(E,K)$ and $\epsilon\in(0,1]$.



\begin{lemma}\label{lemma32}
	Assume that there exists a Hermitian metric $\tilde{K}$ on $(E, D)$ with $\tr(\sqrt{-1}\Lambda_{\omega}G_{\tilde{K}})=r\lambda $, then we can choose a background metric $K$ such that $\tr(\sqrt{-1}\Lambda_{\omega}G_{K})=r\lambda $, $\det (K^{-1}\tilde{K})=1$ and the equation (\ref{dd:3}) has a solution for $\epsilon=1$.
\end{lemma}
\begin{proof}
 Define $h_{1}=\exp(4\sqrt{-1}\Lambda_{\omega}G_{\tilde{K}}-4\lambda \Id_{E})$ and choose $K=\tilde{K}h_{1}^{-1}$ as the fixed background metric, then
 \begin{equation}
 \log \det (K^{-1}\tilde{K})=\log \det (h_{1})=\tr (4\sqrt{-1}\Lambda_{\omega}G_{\tilde{K}}-4\lambda \Id_{E})=0,
 \end{equation}
\begin{equation}
\begin{split} \tr(\sqrt{-1}\Lambda_{\omega}G_{K})&=\tr(\sqrt{-1}\Lambda_{\omega}G_{\tilde{K}})+\frac{1}{2}\sqrt{-1}\Lambda_{\omega}\partial\bar{\partial}\log\det(h_{1}^{-1})\\ 
&=r\lambda
\end{split}
\end{equation}
and
\begin{equation} 4\sqrt{-1}\Lambda_{\omega}G_{K}+\sqrt{-1}\Lambda_{\omega}D(h_{1}^{-1}D_{K}^{c}h_{1})-\log(h_{1})-4\lambda \Id_{E}=0.
\end{equation}
\end{proof}

\begin{prop}\label{p01}
	Let $\tr(\sqrt{-1}\Lambda_{\omega}G_{K})=r\lambda $. If $h$ is a solution of equation (\ref{dd:3}) for $\epsilon\in(0,1]$. Then $\det(h)=1$.
\end{prop}
\begin{proof}
	Taking the trace on both sides of equation (\ref{dd:3}) yields
	\begin{equation}
	2\sqrt{-1}\Lambda_{\omega}\partial\bar{\partial}\log\det(h)-\epsilon\log\det(h)=0.
	\end{equation}
	Due to the maximum principle, we have $\log\det(h)=0$ and $\det(h)=1$.
\end{proof}

\medskip

For any $0<\delta<1$, define
\begin{equation}
J_{\delta}=\{\epsilon\in[\delta,1]\ |\  \text{there is a smooth solution to}\ L_{\epsilon}(h)=0 \}
\end{equation}
and
\begin{equation}
J=\{\epsilon\in(0,1]\ |\  \text{there is a smooth solution to}\ L_{\epsilon}(h)=0 \}.
\end{equation}
We will use the continuity method to show that $J_{\delta}=[\delta,1]$ for any $0<\delta<1$, hence $J=(0,1]$. It is obvious that $J_{\delta}\neq\emptyset$. By the Fredholmness of an elliptic operator over a compact manifold and the implicit function theorem, we know that $J_{\delta}$ is open. Next, we conclude that $J_{\delta}$ is closed.
\begin{lemma}\label{pp:1}
	Let $h$ be the solution of $L_{\epsilon}(h)=0$ for $\epsilon\in[\delta,1]$. Then
	\begin{equation}
	\begin{split}
	\sqrt{-1}\Lambda_{\omega}\partial\bar{\partial}|\log(h)|_{K}^{2}\geq \epsilon|\log(h)|_{K}^{2}-4|\sqrt{-1}\Lambda_{\omega}G_{K}-\lambda \Id_{E}|_{K}|\log(h)|_{K},
	\end{split}
	\end{equation}
	and \begin{equation}\sup_{X}|\log(h)|_{K}\leq\frac{4}{\epsilon}\sup_{X}|\sqrt{-1}\Lambda_{\omega}G_{K}-\lambda \Id_{E}|_{K}\leq C_{12},\end{equation} where $C_{12}$ is a constant depending only on $\delta^{-1}$, $\lambda $ and $K$.
\end{lemma}
\begin{proof}
	Since the term $\langle\Theta(s)(Ds),Ds\rangle_{K}$ in (\ref{key1}) is nonnegative, this lemma comes from Proposition \ref{d1:3} and the maximum principle.
\end{proof}

\medskip

Using the above $C^{0}$-estimate, we can derive the $C^{1}$-estimate and $L_{2}^{p}$-estimate by Donaldson's arguments in \cite[Lemma 19]{Don2}. But now we will give another proof by the maximum principle.
\begin{lemma}\label{pp:2}
	Let $h$ be the solution of $L_{\epsilon}(h)=0$ for $\epsilon\in[\delta,1]$. Then
	\begin{equation}
	\sup_{X}|\psi_{H}^{1,0}|_{H}\leq C_{13},
	\end{equation}
	where $H=Kh$, $C_{13}$ is a constant depending only on the background metric $K$, $\lambda $, the bound of $\sup_{X}|\log{h}|_{K}$ and the geometry of $(X, \omega)$.
\end{lemma}
\begin{proof}
	After a straightforward calculation, we can show
\begin{equation}
\begin{split}
\sqrt{-1}\Lambda_{\omega}\partial\bar{\partial}|\psi_{H}^{1,0}|_{H}^{2}\geq &|\nabla \psi_{H}^{1,0}|_{H}^{2}+|[\psi_{H}^{1,0},\psi_{H}^{0,1}]|_{H}^{2}+2|\psi_{H}^{1,0}\wedge\psi_{H}^{1,0}|_{H}^{2}\\ &-2Re\langle\partial_{H}(\sqrt{-1}\Lambda_{\omega}G_{H}),\psi_{H}^{1,0}\rangle_{H}-\breve{C}|\nabla\psi_{H}^{1,0}|_{H}|\psi_{H}^{1,0}|_{H}-|R|\cdot|\psi_{H}^{1,0}|_{H}^{2},
\end{split}
\end{equation}
where $\breve{C}$ is a constant depending only  on the geometry of $(X, \omega )$, $R$ is the curvature of Chern connection. According to the equation (\ref{dd:3}), we obtain
	\begin{equation}
	\begin{split}
	\sqrt{-1}\Lambda_{\omega}\partial\bar{\partial}|\psi_{H}^{1,0}|_{H}^{2}\geq &-\breve{C}_{1}|\psi_{H}^{1,0}|_{H}^{2}-\frac{\epsilon}{2}|\partial_{H}(\log(h))|_{H}|\psi_{H}^{1,0}|_{H},
	\end{split}
	\end{equation}
where $\breve{C}_{1}$ is a uniform constant depending only on $(X, \omega ) $.
	Notice that
	\begin{equation}
	|\partial_{H}(\log(h))|_{H}\leq \breve{C}_{2}|h^{-1}\partial_{H}h|_{H}\leq \breve{C}_{3}|\psi_{H}^{1,0}|_{H}+\breve{C}_{4},
	\end{equation}
	then
	\begin{equation}
	\begin{split}
	\sqrt{-1}\Lambda_{\omega}\partial\bar{\partial}|\psi_{H}^{1,0}|_{H}^{2}\geq &-\breve{C}_{5}|\psi_{H}^{1,0}|_{H}^{2}-\breve{C}_{6}|\psi_{H}^{1,0}|_{H},
	\end{split}
	\end{equation}
	 where $\breve{C}_{2},\breve{C}_{3}$, $\breve{C}_{4}$, $\breve{C}_{5}$ and $\breve{C}_{6}$ are the constants depending only on $\sup_{X}|\log{h}|_{K}$ and $K$.  On the other hand, we have
	\begin{equation}
	\begin{split}
	\sqrt{-1}\Lambda_{\omega}dd^{c}\tr(h)=&\sqrt{-1}\Lambda_{\omega}\tr(Dh\wedge h^{-1}D_{K}^{c}h)+\sqrt{-1}\Lambda_{\omega}\tr(hD(h^{-1}D_{K}^{c}h))\\
	=&|Dh\cdot h^{-1/2}|_{K}^{2}-4\tr(h(\sqrt{-1}\Lambda_{\omega}G_{K}-\lambda \Id_{E}))+\epsilon \tr(h\log(h)),
	\end{split}
	\end{equation}
	and
	\begin{equation}
	|Dh\cdot h^{-1/2}|_{K}^{2}\geq \breve{C}_{7}|Dh\cdot h^{-1}|_{H}^{2}\geq \breve{C}_{8}|\psi_{H}^{1,0}|_{H}^{2}-\breve{C}_{9},
	\end{equation}
	hence it follows that
	\begin{equation}
	\begin{split}
	\sqrt{-1}\Lambda_{\omega}dd^{c}\tr(h)\geq &\breve{C}_{8}|\psi_{H}^{1,0}|_{H}^{2}-\breve{C}_{10},
	\end{split}
	\end{equation}
where $\breve{C}_{7},\breve{C}_{8}$, $\breve{C}_{9}$ and $\breve{C}_{10}$ are the constants depending only on $\sup_{X}|\log{h}|_{K}$, $\lambda $ and $K$.
	Let $f=|\psi_{H}^{1,0}|_{H}^{2}+A\tr(h)$, where $A$ is large enough such that
	\begin{equation}
	\sqrt{-1}\Lambda_{\omega}\partial\bar{\partial}f\geq \breve{C}_{11}|\psi_{H}^{1,0}|_{H}^{2}-\breve{C}_{12},
	\end{equation}
	where $\breve{C}_{11}$ and $\breve{C}_{12}$ are positive constants depending only on the background metric $K$, the bound of $\sup_{X}|\log{h}|_{K}$ and the geometry of $(X, \omega)$.
	The maximum principle implies:
	\begin{equation}
	\max_{X}|\psi_{H}^{1,0}|^{2}_{H}\leq C_{13},
	\end{equation}
	where $C_{13}$ is a positive constant depending only on $\breve{C}_{11}$ and $\breve{C}_{12}$.
\end{proof}

\medskip

Combining Lemma \ref{pp:1}, Lemma \ref{pp:2} and the $L_{2}^{p}$-estimate of elliptic operator, we immediately know $h_{\epsilon}$ is uniformly bounded in $L^{p}_{2}$ for $\epsilon\in[\delta,1]$. Thus $J_{\delta}$ is closed for $[\delta,1]$ and $J=(0,1]$.

\begin{lemma}\label{lem:1}
	Let $(X,\omega)$ be an $n$-dimensional compact Gauduchon manifold,  $(E,D)$  be a  projectively flat bundle over  $(X, \omega )$ with  rank $r$  and satisfying $\sqrt{-1}F_{D}=\alpha \otimes \Id_{E}$ for a real $(1, 1)$-form $\alpha $, $h_{\epsilon}$ be a solution of equation (\ref{dd:3}), $s_{\epsilon}=\log(h_{\epsilon})$. Then
	\begin{equation}
	\begin{split}
	&-\int_{X}4\tr((\sqrt{-1}\Lambda_{\omega}G_{K}-\lambda \Id_{E})s_{\epsilon})\frac{\omega^{n}}{n!}+\int_{X}\langle\Theta(s_{\epsilon})(Ds_{\epsilon}),Ds_{\epsilon}\rangle_{K}\frac{\omega^{n}}{n!}+\epsilon\|s_{\epsilon}\|_{L^{2}}^{2}=0.
	\end{split}
	\end{equation}
\end{lemma}
\begin{proof}
	This follows directly from Proposition \ref{d1:3}.
\end{proof}

\begin{lemma}\label{lemmau}
Let $(E,D)$  be a simple projectively flat complex vector bundle over  a compact Hermitian manifold $(X, \omega )$ satisfying $\sqrt{-1}F_{D}=\alpha \otimes \Id_{E}$ for a real $(1, 1)$-form $\alpha $, $H$ and $K$ be two Hermitian metrics on $E$. If $\Lambda_{\omega}G_{H}=\Lambda_{\omega}G_{K}$, then $H=cK$ for some positive constant $c$.
\end{lemma}
\begin{proof}
 By (\ref{unique00}), we have
\begin{equation}\label{unique2}
\sqrt{-1}\Lambda_{\omega}dd^{c}\tr(h)\geq |Dh\cdot h^{-1/2}|_{K}^{2},
\end{equation}
where $h=K^{-1}H$. Hence $Dh=0$. Since $(E, D)$ is simple and $h^{\ast K}=h$, then $h=c \Id_{E}$ for some positive constant $c$.
\end{proof}

\medskip

\begin{theorem}\label{thm:x1}
	Let $(X,\omega)$ be an $n$-dimensional compact Gauduchon manifold,  $(E,D)$  be a simple projectively flat bundle over  $(X, \omega )$ with  rank $r$  and satisfying $\sqrt{-1}F_{D}=\alpha \otimes \Id_{E}$ for a real $(1, 1)$-form $\alpha $.   Then there exists a unique Hermitian metric $H$ on $(E, D)$ satisfying \begin{equation}\label{hh1}\sqrt{-1}\Lambda_{\omega}G_{H}=\lambda \Id_{E},\end{equation} where the constant $\lambda=\frac{\deg_{\omega }(E,D)}{r \Vol(X,\omega)}$.
\end{theorem}
\begin{proof}
Lemma \ref{lemmau} gives the uniqueness. Lemma \ref{lemma31} guarantees that there exists a Hermitian metric $\tilde{K}$ on $(E, D)$ with $\tr(\sqrt{-1}\Lambda_{\omega}G_{\tilde{K}})=r\lambda$.
By choosing a background metric $K$ with $\tr(\sqrt{-1}\Lambda_{\omega}G_{K})=r\lambda $,  as above we  know that there exists a solution $h_{\epsilon}$  of the equation (\ref{dd:3}) for any $\epsilon\in(0,1]$. Under the assumption that $(E,D)$ is simple, we conclude that there is a uniform constant $\breve{C}_{13}$ independent of $\epsilon$ such that \begin{equation}\label{c0}\|\log(h_{\epsilon})\|_{L^{2}}<\breve{C}_{13}.\end{equation} Then Proposition \ref{pp:0} implies that $\|\log(h_{\epsilon})\|_{L^{\infty}}$ are uniformly bounded. By Lemma \ref{pp:2}, we can get the uniform $C^{1}$-estimate and also the uniform $L_{2}^{p}$-estimate. After choosing a subsequence of $\epsilon$, we can obtain a limiting metric $H$ satisfying $\sqrt{-1}\Lambda_{\omega}G_{H}=\lambda \Id_{E}$.
	
	Suppose $\overline{\lim\limits_{\epsilon\rightarrow 0}}\|\log(h_{\epsilon})\|_{L^{2}}=+\infty$. Then there is a sequence $\epsilon_{j}\rightarrow 0$, such that
	\begin{equation}
	\|\log(h_{\epsilon_{j}})\|_{L^{2}} \rightarrow +\infty.
	\end{equation}
	Set
	\begin{equation}
	s_{j}=\log(h_{\epsilon_{j}}),\ \ l_{j}=\|s_{j}\|_{L^{2}},\ \ u_{j}=l_{j}^{-1}s_{j},
	\end{equation}
	then
	\begin{equation}
	 \|u_{j}\|_{L^{2}}=1,\ \ \|u_{j}\|_{L^{\infty}}\leq C^{*}.
	\end{equation}
On the other hand, by Proposition \ref{p01}, we have
\begin{equation}
	\tr(u_{j})=0.
	\end{equation}
	Based on Lemma \ref{lem:1}, we derive
	\begin{equation} -\int_{X}4\tr((\sqrt{-1}\Lambda_{\omega}G_{K}-\lambda \Id_{E})u_{j})\frac{\omega^{n}}{n!}+l_{j}\int_{X}\langle\Theta(l_{j}u_{j})(Du_{j}),Du_{j}\rangle_{K}\frac{\omega^{n}}{n!}+\epsilon_{j}l_{j}=0.
	\end{equation}
	According to the argument in \cite[page 637]{NZ}, if $\Upsilon: \mathbb{R}\times \mathbb{R} \rightarrow \mathbb{R}$ is a positive smooth function such that $\Upsilon(\lambda_{1},\lambda_{2})<(\lambda_{1}-\lambda_{2})^{-1}$ whenever $\lambda_{1}>\lambda_{2}$, it holds that
	\begin{equation}
	\begin{split}
	&-\int_{X}4\tr((\sqrt{-1}\Lambda_{\omega}G_{K}-\lambda \Id_{E})u_{j})\frac{\omega^{n}}{n!}+\int_{X}\langle\Upsilon(u_{j})(Du_{j}),Du_{j}\rangle_{K}\frac{\omega^{n}}{n!}+\epsilon_{j}l_{j}\leq 0,\ \ j\gg 0.
	\end{split}
	\end{equation}
	Then we see $u_{j}$ is uniformly bounded in $L_{1}^{2}$. By choosing a subsequence which is also denoted by $u_{j}$, we deduce $u_{j}\rightarrow u_{\infty}$ weakly in $L_{1}^{2}$ as $j\rightarrow +\infty$. Then
	\begin{equation}
	\tr(u_{\infty})=0, \ \ \|u_{\infty}\|_{L^{2}}=1,\ \ \|u_{\infty}\|_{L^{\infty}}\leq C^{*}.
	\end{equation}
	Following Simpson's argument in \cite[Lemma 5.5]{S1}, we know the eigenvalues of $u_{\infty}$ are constants almost everywhere. Let $\lambda_{1}<\lambda_{2}<\cdots<\lambda_{b}$ denote the distinct eigenvalues. The fact that $\tr(u_{\infty})=0$ and $\|u_{\infty}\|_{L^{2}}=1$ yields $b\geq 2$. Define
\begin{equation*}
p_{a}=\left\{\begin{split} 1,\quad & \text{$x\leq \lambda_{a}$}, \\
0,\quad & \text{$x\geq\lambda_{a+1}$}. \end{split}\right.
\end{equation*}
Let $\pi_{a}=p_{a}(u_{\infty})$. From \cite[Lemma 5.6]{S1}, one can see that
\begin{itemize}
\item[(1)]$\pi_{a}\in L_{1}^{2}$;
\item[(2)]$\pi_{a}^{2}=\pi_{a}^{*}=\pi_{a}$;
\item[(3)]$(\Id-\pi_{a})D\pi_{a}=0$.
\end{itemize}
Set $V_{a}=\pi_{a}(E)$. By the condition that  $\sqrt{-1}F_{D}=\alpha \otimes \Id_{E}$ for a real $(1, 1)$-form $\alpha $, we know that $(D^{0, 1})^{2}=0$, i.e. $D^{0,1}$ determines a holomorphic structure on $E$.  Since $(\Id-\pi_{a})D^{0,1}\pi_{a}=0$, the Uhlenbeck-Yau's regularity theorem of $L_{1}^{2}$-subbundle (\cite{UhYau}) states that $V_{a}$ is smooth outside an analytic subset $\Sigma$ which is of co-dimension at least $2$.
	Then $V_{a}$ is a $D$-invariant subbundle of $(E, D)$ on $X\setminus\Sigma$.  In the following, we can extend $V_{a}$ to the whole $X$ as a $D$-invariant subbundle of $E$, which is contradict with the simpleness of $E$. Then, we get the uniform $C^{0}$-estimate (\ref{c0}).

For any $x\in \Sigma$, there exists a domain $\hat{B}\subset X$ containing $x$ such that $\alpha= \sqrt{-1}d\beta$ for some $1$-form $\beta$ defined on $\hat{B}$. From
\begin{equation}
\sqrt{-1}(D- \beta\otimes \Id_E)^2= \sqrt{-1}D^2- \sqrt{-1}d\beta\otimes \Id_E^2= 0,
\end{equation}
we know that $D- \beta\otimes \Id_E$ is a flat connection on $E|_{\hat{B}}$.
On a small domain $B$ containing $x$, we can choose a local basis $\{e_1, \cdots, e_r\}$ of $E$ such that
\begin{equation}\label{IMPO2}
D\left(\begin{array}{c}
e_1 \\ \vdots \\ e_r
\end{array}\right)
=\beta\left(\begin{array}{c}
e_1 \\ \vdots \\ e_r
\end{array}\right).
\end{equation}
On $X\setminus \Sigma$, we have the following  exact sequence
\begin{equation}
0\longrightarrow V_a \longrightarrow E\longrightarrow Q_a \longrightarrow 0,
\end{equation}
and
\begin{equation}
D=\left(\begin{array}{cc}
D_{V_a} &   \eta\\
0 &  D_{Q_a}
\end{array}\right),
\end{equation}
\begin{equation}
D^2=\left(\begin{array}{cc}
D_{V_a}^2 &   D_{V_a}\circ\eta+ \eta\circ D_{V_a}\\
0 &  D_{Q_a}^2
\end{array}\right)
=-\sqrt{-1}\alpha\otimes \Id_{E},
\end{equation}
where $\eta$ is a $1$-form valued in $Q_{a}^{\ast}\otimes V_{a }$. So  $D_{V_{a }}- \beta\otimes \Id_{V_{a }}$ is also a flat connection on $V_{a }|_{\hat{B}\setminus \Sigma}$.
For any point $y\in B\setminus \Sigma$, there exist a small domain $\tilde{B}_y\subset B\setminus \Sigma$, and  a basis $\{\tilde{e}_1, \cdots, \tilde{e}_l\}$ of $V_a$ on $\tilde{B}_y$, where $l= \mbox{rank}(V_a)$, such that
\begin{equation}\label{IMPO3}
D_{V_a}\left(\begin{array}{c}
\tilde{e}_1 \\ \vdots \\ \tilde{e}_l
\end{array}\right)
=D\left(\begin{array}{c}
\tilde{e}_1 \\ \vdots \\ \tilde{e}_l
\end{array}\right)
=\beta\left(\begin{array}{c}
\tilde{e}_1 \\ \vdots \\ \tilde{e}_l
\end{array}\right).
\end{equation}
On the other hand, since $V_{a }\subset E$, we can suppose
\begin{equation}
\left(\begin{array}{c}
\tilde{e}_1 \\ \vdots \\ \tilde{e}_l
\end{array}\right)
=A_{l\times r}\left(\begin{array}{c}
e_1 \\ \vdots \\ e_r
\end{array}\right).
\end{equation}
Then (\ref{IMPO2})   and (\ref{IMPO3}) imply that $dA=0$, i.e. $A$ is a constant matrix. For another point $y_{1}\in B\setminus \Sigma$, if  $\tilde{B}_{y_1}\cap\tilde{B}_{y}\neq \emptyset$, it is easy to see that $A_{y_{1}}=C\cdot A$ for some constant matrix $C$, and then
$
A\left(\begin{array}{c}
e_1 \\ \vdots \\ e_r
\end{array}\right)
$
is a basis of $V_a$ on $\tilde{B}_{y_1}\cup\tilde{B}_{y}$.
Since $B\setminus \Sigma$ is connected,
$
A\left(\begin{array}{c}
e_1 \\ \vdots \\ e_r
\end{array}\right)
$
is also a basis of $V_a$ on  $B\setminus \Sigma$. Then $V_a |_{B\setminus \Sigma}$ can be extended to $B$.
Because $x$ is arbitrary,  $V_a$ can be extended to the whole $X$.
\end{proof}

\medskip

\begin{proof}[Proof of Theorem \ref{thm0}]
It has been proved by Gauduchon (\cite{Gaud}) that  there must exist a Gauduchon metric $\tilde{\omega}$  in the conformal class of $\omega $. By Theorem \ref{thm:x1}, we have a  metric $H$ such that $\sqrt{-1}\Lambda_{\tilde{\omega }}(G_{H}^{\perp})=0$, and equivalently \begin{equation}\sqrt{-1}\Lambda_{\omega}(G_{H}^{\perp})=0.\end{equation} If $\partial \overline{\partial } \omega ^{n-2}=0$, from Lemma \ref{lem:x1}, we know that the metric $H$ satisfies $G_{H}^{\perp}=0$.

\end{proof}

\medskip

Taking conformal transformation, we have the following corollary.

\medskip

\begin{corollary}\label{q1}
Let $(X, \omega )$ be a compact Hermitian manifold and $(E, D)$  a simple projectively flat complex vector bundle over $X$ satisfying $\sqrt{-1}F_{D}=\alpha \otimes \Id_{E}$ for a real $(1, 1)$-form $\alpha $. Then for any real $(1,1)$-form $\zeta \in BC(E, D)$, there must exist a unique Hermitian metric $H$ on $E$ such that	
\begin{equation}
\sqrt{-1}\Lambda_{\omega}(G_{H}+\frac{\sqrt{-1}}{r}\zeta \otimes \Id_{E})=0.
\end{equation}
 If $\partial \overline{\partial } \omega ^{n-2}=0$ or $\dim^{\mathbb{C}} X=2$, then the metric $H$ satisfies
 \begin{equation}
\sqrt{-1}G_{H}=\frac{1}{r}\zeta \otimes \Id_{E}.
\end{equation}
\end{corollary}

\medskip

Let $(E, D)$ be an $SL(r, \mathbb{C})$-flat bundle over a compact Hermitian manifold $(X, \omega)$, i.e. the corresponding representation $\rho : \pi_{1}(X)\rightarrow GL(r, \mathbb{C})$ has image in $SL(r, \mathbb{C})$. There must exist a Hermitian metric $K$ on $(E, D)$ such that $\tr \psi_{K}=0$, and hence $\tr(\sqrt{-1}\Lambda_{\omega}G_{K})=0$ and  $BC(E, D)=0$.  Then the following corollary holds.

\medskip

\begin{corollary}\label{coro1}
	Let $(E, D)$ be a simple $SL(r, \mathbb{C})$-flat bundle over a compact Hermitian manifold $(X, \omega)$, then there must exist a unique Hermitian metric $H$ such that $\sqrt{-1}\Lambda_{\omega}G_{H}=0$. If $\partial \overline{\partial } \omega ^{n-2}=0$ or $\dim^{\mathbb{C}} X=2$, then the metric $H$ satisfies $G_{H}=0$, and there exists a Hermitian-flat Higgs bundle $(E, D_{H}^{0,1}, \psi_{H}^{1, 0})$ with $\tr \psi_{H}^{1, 0}=0$.
\end{corollary}

\medskip

\section{Characteristic classes of projectively flat bundle}

Let $(E, D)$ be a projectively flat vector bundle over a compact complex  manifold $X$ with $\sqrt{-1}F_{D}=\alpha \otimes \Id_{E}$, where $\alpha $ is a real $(1, 1)$-form.  Letting $H$ be a Hermitian metric on the projectively flat bundle $(E, D)$, we have the unique decomposition $D=D_{H}+\psi_{D, H}$. On the other hand, one can define $D^{\ast H}$ by
\begin{equation}
dH(X, Y)=H(DX, Y)+H(X, D^{\ast H}Y)
\end{equation}
for any $X, Y \in \Gamma(E)$, and then
\begin{equation}\label{ch1}
\psi_{D, H}=\frac{1}{2}(D-D^{\ast H}).
\end{equation}
Setting
\begin{equation}
v_{2j+1}(E, D, H)=(2\pi \sqrt{-1})^{-j}\tr \psi_{D, H}^{2j+1},
\end{equation}
and recalling that $\psi_{D,H}$ is self-adjoint and (\ref{co1}), we know that $v_{2j+1}(E, D, H)$ is a closed real form.

For any two Hermitian metrics $H$ and $K$ on $E$, we choose  a smooth path $H_{t}$ connecting $K$ and $H$ with $H_0=K$ and $H_1=H$, let $h_{t}=K^{-1}H_{t}$. From (\ref{ex2}), it follows that
\begin{equation}\label{ex03}
\psi_{D, H_{t}}=\frac{1}{2}h_{t}^{-1}\circ \psi_{D, K}\circ h_{t} +\frac{1}{2}\psi_{D,K}+\frac{1}{2}(D_{K}-h_{t}^{-1}\circ D_{K} \circ h_{t}),
\end{equation}
and
\begin{equation}\label{ex04}
\frac{\partial }{\partial t}\psi_{D, H_{t}}=\frac{1}{2}\{\psi_{D, H_{t}} \circ h_{t}^{-1}\frac{\partial h_{t}}{\partial t}-h_{t}^{-1}\frac{\partial h_{t}}{\partial t}\circ \psi_{D, H_{t}} -D_{H_{t}}(h_{t}^{-1}\frac{\partial h_{t}}{\partial t})\}.
\end{equation}
By direct calculation, we obtain
\begin{equation}
\begin{split}
&\frac{\partial }{\partial t} v_{2j+1}(E, D, H_{t})=(2\pi \sqrt{-1})^{-j}\tr (\frac{\partial }{\partial t}\psi_{D, H_{t}}^{2j+1})\\
=& (2j+1)(2\pi \sqrt{-1})^{-j}\tr (\frac{\partial \psi_{D, H_{t}}}{\partial t}\wedge \psi_{D, H_{t}}^{2j})\\
=& -\frac{2j+1}{2}(2\pi \sqrt{-1})^{-j}\tr (D_{H_{t}}(h_{t}^{-1}\frac{\partial h_{t}}{\partial t})\wedge \psi_{D, H_{t}}^{2j})\\
=& -\frac{2j+1}{2}(2\pi \sqrt{-1})^{-j}d\tr (h_{t}^{-1}\frac{\partial h_{t}}{\partial t}\wedge \psi_{D, H_{t}}^{2j}),\\
\end{split}
\end{equation}
\begin{equation}\label{ccc1}
v_{2j+1}(E, D, H)-v_{2j+1}(E, D, K)=
 -\frac{2j+1}{2}(2\pi \sqrt{-1})^{-j}d\int_{0}^{1}\tr (h_{t}^{-1}\frac{\partial h_{t}}{\partial t}\wedge \psi_{D, H_{t}}^{2j})dt
\end{equation}
and the class $[v_{2j+1}(E, D, H)]$ is independent of the choice of Hermitian metrics. So we define the following class:
\begin{equation}
v_{2j+1}(E, D)=[(2\pi \sqrt{-1})^{-j}\tr \psi_{H}^{2j+1}]
\end{equation}
for $j\geq 0$ .

Let $f:(\tilde{E}, \tilde{D})\rightarrow (E, D)$ be an isomorphism, i.e. $\tilde{D}=f^{-1}\circ D \circ f$. Given a Hermitian metric on $E$, we set
a Hermitian metric $\tilde{H}$ on $\tilde{E}$ by
\begin{equation}
\tilde{H}(\cdot , \cdot )=H (f(\cdot ) , f(\cdot ) ).
\end{equation}
According to the definition, we have
\begin{equation}
\begin{split}
&\tilde{H}(\psi_{\tilde{D}, \tilde{H}}(X) , Y )=\frac{1}{2}\{\tilde{H}(\tilde{D}X , Y )+\tilde{H}(X, \tilde{D}Y )-d\tilde{H}(X , Y )\}\\
=& \frac{1}{2}\{H(D\circ f (X) , f(Y) )+H(f(X), D\circ f(Y) )-dH(f(X) , f(Y) )\}\\
=& H(\psi_{D, H}\circ f(X) , f(Y) )\\
=& \tilde{H}(f^{-1}\circ \psi_{D, H}\circ f(X) , Y ),
\end{split}
\end{equation}
i.e.
\begin{equation}
\psi_{\tilde{D}, \tilde{H}}= f^{-1}\circ \psi_{D, H}\circ f,
\end{equation}
and then
\begin{equation}\label{ccc2}
v_{2j+1}(E, D, H)=v_{2j+1}(\tilde{E}, \tilde{D}, \tilde{H})
\end{equation}
for $j\geq 0$ .

\begin{proof}[Proof of Theorem \ref{thm:2-4}]
If the projectively flat bundle $(E,D)$ is simple and $\omega$ is astheno-K\"ahler, by Theorem \ref{thm0}, there is a Hermitian metric H such that $G_{H}^{\perp}=0$, i.e.
\begin{equation}
(D_{H}^{0, 1})^{2}=0, \quad D_{H}^{0, 1}\psi_{D, H}^{1,0}=\frac{1}{r}\bar{\partial}\tr \psi_{D, H}^{1,0}, \quad and \quad \psi_{D, H}^{1,0} \wedge \psi_{D, H}^{1,0}=0.
\end{equation}
Then
\begin{equation}\label{ch2}
\begin{split}
&v_{2j+1}(E, D, H)=(2\pi \sqrt{-1})^{-j}\tr (\psi_{D, H}^{1,0}+\psi_{D, H}^{0,1})^{2j+1}\\
=& (2\pi \sqrt{-1})^{-j}\tr \{ ((\psi_{D, H}^{1,0}\wedge \psi_{D, H}^{0,1})^{j}+(\psi_{D, H}^{0,1}\wedge \psi_{D, H}^{1,0})^{j})\wedge (\psi_{D, H}^{1,0}+\psi_{D, H}^{0,1})\}\\
=& 0.
\end{split}
\end{equation}

When $(E, D)$ is not simple, we choose a $D$-invariant subbundle $S$ of minimal rank and we have  the following exact sequence of bundles:
\begin{equation}
0\rightarrow S\rightarrow E\rightarrow Q\rightarrow 0.
\end{equation}
Take  a bundle isomorphism $f:S\oplus Q\rightarrow E$, then the pull-back of $D$ can be expressed as
\begin{equation}
f^{*}(D)=\left(\begin{split}
  &D_{S}  &\beta \\
  &0  &D_{Q}
\end{split}\right),
\end{equation}
and one can check that
\begin{equation}\label{ccc3}
\sqrt{-1}D_{S}^{2}=\alpha \otimes \Id_{S}, \quad \sqrt{-1}D_{Q}^{2}=\alpha \otimes \Id_{Q}, \quad D_{Q^{*}\otimes S}\beta=0.\end{equation}
Since $S$ is of minimal rank, it must be simple and then
\begin{equation}
v_{2j+1}(S, D_{S})=0.
\end{equation}
In the following, denote
\begin{equation}
\tilde{E}=S\oplus Q, \quad \tilde{D}=f^{*}(D), \quad H=(f^{-1})^{\ast}\tilde{H}, \quad  \tilde{H}=\left(\begin{split}
  &H_{S}  &0 \\
  &0  &H_{Q}
\end{split}\right),
\end{equation}
where $H_{S}$  and $H_{Q}$ are Hermitian metrics on S and $Q$.
It is easy to get that
\begin{equation}
\tilde{D}^{\ast \tilde{H}}=\left(\begin{split}
  &D_{S}^{\ast H_{S}}  &0 \\
  &-\beta^{\ast \tilde{H}}  &D_{Q}^{\ast H_{Q}}
\end{split}\right),
\end{equation}
and
\begin{equation}
\psi_{\tilde{D}, \tilde{H}}=\frac{1}{2}(\tilde{D}-\tilde{D}^{\ast \tilde{H}})=\left(\begin{split}
  &\psi_{D_{S}, H_{S}}  &\frac{1}{2}\beta \\
  &\frac{1}{2}\beta^{\ast \tilde{H}}  &\psi_{D_{Q}, H_{Q}}
\end{split}\right).
\end{equation}
Setting
\begin{equation}
\tilde{H}_{t}=\left(\begin{split}
  &tH_{S}  &0 \\
  &0  &\frac{1}{t}H_{Q}
\end{split}\right),
\end{equation}
one can find that
\begin{equation}
\beta^{\ast \tilde{H}_{t}}=t^{2}\beta^{\ast \tilde{H}}
\end{equation}
and
\begin{equation}
\psi_{\tilde{D}, \tilde{H}_{t}}=\left(\begin{split}
  &\psi_{D_{S}, H_{S}}  &\frac{1}{2}\beta \\
  &\frac{1}{2}t^{2}\beta^{\ast \tilde{H}}  &\psi_{D_{Q}, H_{Q}}
\end{split}\right).
\end{equation}
For simplicity, we denote:
\begin{equation}
\Xi =\left(\begin{split}
  &\psi_{D_{S}, H_{S}}  &0 \\
  &0  &\psi_{D_{Q}, H_{Q}}
\end{split}\right),
\end{equation}
\begin{equation}
M_{t}=\left(\begin{split}
  &0  &\frac{1}{2}\beta \\
  &\frac{1}{2}t^{2}\beta^{\ast \tilde{H}}  &0
\end{split}\right)=\left(\begin{split}
  &1  &0 \\
  &0  &t^{2}
\end{split}\right)\left(\begin{split}
  &0  &\frac{1}{2}\beta \\
  &\frac{1}{2}\beta^{\ast \tilde{H}}  &0
\end{split}\right)=\left(\begin{split}
  &0  &\frac{1}{2}\beta \\
  &\frac{1}{2}\beta^{\ast \tilde{H}}  &0
\end{split}\right)\left(\begin{split}
  &t^{2}  &0 \\
  &0  &1
\end{split}\right),
\end{equation}
\begin{equation}
\Psi _{2j+1, 2k, t}=\sum_{\alpha_{1}+\cdots +\alpha_{2k+1}=2j-2k+1}\Xi^{\alpha_{1}}M_{t}\cdots \Xi^{\alpha_{2k}}M_{t}\cdot \Xi^{\alpha_{2k+1}}
\end{equation}
and
\begin{equation}
\Psi _{2j+1, 2k+1, t}=\sum_{\alpha_{1}+\cdots +\alpha_{2k+2}=2j-2k}\Xi^{\alpha_{1}}M_{t}\cdots \Xi^{\alpha_{2k+1}}M_{t}\cdot \Xi^{\alpha_{2k+2}},
\end{equation}
where $\alpha_{i}$ is a nonnegative integer.
It is not hard to observe that
\begin{equation}
\Xi^{\alpha_{i}}M_{t}\cdot \Xi^{\alpha_{i+1}}M_{t}=t^{2}\Xi^{\alpha_{i}}M_{1}\cdot \Xi^{\alpha_{i+1}}M_{1},
\end{equation}
\begin{equation}
\Psi _{2j+1, 2k, t}=t^{2k}\Psi _{2j+1, 2k, 1},
\end{equation}
and
\begin{equation}
\tr (\Psi _{2j+1, 2k+1, t})=0.
\end{equation}
Then
\begin{equation}
\psi_{\tilde{D}, \tilde{H}_{t}}^{2j+1}=\sum_{k=0}^{j}\Psi _{2j+1, 2k, t}+\sum_{k=0}^{j}\Psi _{2j+1, 2k+1, t},
\end{equation}
and
\begin{equation}
\tr \{\psi_{\tilde{D}, \tilde{H}_{t}}^{2j+1}\}=\tr \{\psi_{D_{s}, H_{s}}^{2j+1}\}+\tr \{\psi_{D_{Q}, H_{Q}}^{2j+1}\}+\sum_{k=1}^{j}t^{2k}\tr \{\Psi _{2j+1, 2k, 1}\}.
\end{equation}
By (\ref{ccc1}) and (\ref{ccc2}), we know that
\begin{equation}
[\tr \{\psi_{D, H}^{2j+1}\}]=[\tr \{\psi_{\tilde{D}, \tilde{H}_{t}}^{2j+1}\}]=[\tr \{\psi_{\tilde{D}, \tilde{H}}^{2j+1}\}],
\end{equation}
and
\begin{equation}
\sum_{k=1}^{j}(t^{2k}-1)[\tr \{\Psi _{2j+1, 2k, 1}\}]=0
\end{equation}
for any $t>0$. This implies
\begin{equation}
[\tr \{\Psi _{2j+1, 2k, 1}\}]=0
\end{equation}
for all $1\leq k\leq j$, then
\begin{equation}
[\tr \{\psi_{D, H}^{2j+1}\}]=[\tr \{\psi_{D_{s}, H_{s}}^{2j+1}\}]+[\tr \{\psi_{D_{Q}, H_{Q}}^{2j+1}\}],
\end{equation}
i.e.
\begin{equation}
v_{2j+1}(E, D)=v_{2j+1}(S, D_{S})+v_{2j+1}(Q, D_{Q})=v_{2j+1}(Q, D_{Q})
\end{equation}
for $j\geq 0$ . Using (\ref{ccc3}),  we can prove the vanishing of $v_{2j+1}(E, D)$  by induction on the rank.
\end{proof}

\section{Proof of Theorem \ref{thm:2-1}}

Let $(X,\omega)$ be a compact Hermitian manifold of dimension $n$ and  suppose that $\partial\bar{\partial}\omega^{n-1}=0$ and $\partial\bar{\partial}\omega^{n-2}=0$. Let $(E,\bar{\partial}_{E},\theta)$ be a rank $r$ Higgs bundle over $X$, $H$ be a Hermitian metric on $E$. Define the Hitchin-Simpson connection by $D_{H,\theta}=D_{H}+\theta+\theta^{*H}$, where $D_{H}$ is the Chern connection and $\theta^{*H}$ is the adjoint of $\theta$ with respect to $H$. The related curvature of $D_{H,\theta}$ is
\begin{equation}
F_{H,\theta}=F_{H}+[\theta,\theta^{*H}]+\partial_{H}\theta+\bar{\partial}_{E}\theta^{*H}.
\end{equation}
The Chern character forms $ch_{k}(E,\bar{\partial}_{E},H)\in\mathcal{A}^{k,k}(X)$ are defined by
\begin{equation}
ch_{k}(E,\bar{\partial}_{E},H)=\frac{1}{k!}\tr((\frac{\sqrt{-1}}{2\pi}F_{H})^{k}).
\end{equation}
Donaldson (\cite[Proposition 6]{Don2}) proved that, given two metrics $H_{1}$ and $H_{2}$ on $E$,  there exists $R_{k-1}(H_1,H_2)\in A^{k-1,k-1}(X)$ such that
\begin{equation}
ch_k(E,\bar{\partial}_{E}, H_1)-ch_k(E,\bar{\partial}_{E}, H_2)=\sqrt{-1}\bar{\partial}\partial R_{k-1}(H_1, H_2),
\end{equation} for every $1\leq k\leq \min{\{r,n\}}$. This means that every $ch_{k}(E,\bar{\partial}_{E},H)$ determines a Bott-Chern cohomology class $ch_{k}(E,\bar{\partial}_{E})\in H_{BC}^{k, k}(X)$.

Since $\partial\bar{\partial}\omega^{n-1}=\partial\bar{\partial}\omega^{n-2}=0$, the following  first and second Chern numbers are well defined,
\begin{equation}
ch_{1}(E,\bar{\partial}_{E})[\omega^{n-1}]=\int_{X}ch_{1}(E,\bar{\partial}_{E},H)\wedge\frac{\omega^{n-1}}{(n-1)!}
\end{equation}
and
\begin{equation}
ch_{2}(E,\bar{\partial}_{E})[\omega^{n-2}]=\int_{X}ch_{2}(E,\bar{\partial}_{E},H)\wedge\frac{\omega^{n-2}}{(n-2)!}.
\end{equation}

By the direct calculations, we have
\begin{equation}
ch_{1}(E,\bar{\partial}_{E})[\omega^{n-1}]=\int_{X}\frac{\sqrt{-1}}{2\pi}\tr(F_{H,\theta})\wedge\frac{\omega^{n-1}}{(n-1)!},
\end{equation}
\begin{equation}
\begin{split}
&\frac{-1}{8\pi^{2}}\int_{X}\tr(F_{H,\theta}\wedge F_{H,\theta})\wedge\frac{\omega^{n-2}}{(n-2)!}\\
=&\frac{-1}{8\pi^{2}}\int_{X}\{\tr(F_{H}\wedge F_{H})+2\tr(F_{H}\wedge[\theta,\theta^{*H}])\\
&+\tr([\theta,\theta^{*H}]\wedge[\theta,\theta^{*H}])+2\tr(\partial_{H}\theta\wedge\bar{\partial}_{E}\theta^{*H})\}\wedge\frac{\omega^{n-2}}{(n-2)!}\\
=&\frac{-1}{8\pi^{2}}\int_{X}\{\tr(F_{H}\wedge F_{H})+2\tr(F_{H}\wedge[\theta,\theta^{*H}])\\
&-2\partial\bar{\partial}\tr(\theta\wedge\theta^{*H})+2\tr(\theta\wedge[F_{H},\theta^{*H}])\}\wedge\frac{\omega^{n-2}}{(n-2)!}\\
=&ch_{2}(E,\bar{\partial}_{E})[\omega^{n-2}],
\end{split}
\end{equation}
and
\begin{equation}\label{Bo1}
\begin{split}
&8\pi^{2}(c_{2}(E,\bar{\partial}_{E})-\frac{r-1}{2r}c_{1}^{2}(E,\bar{\partial}_{E}))\cdot [\omega^{n-2}]\\
= &\int_{X}(2|\partial_{H}\theta |^{2}+|(F_{H,\theta}^{1,1})^{\perp}|^{2}-|\Lambda_{\omega}(F_{H,\theta}^{1,1})^{\perp}|^{2})\frac{\omega^{n}}{n!},
\end{split}
\end{equation}
where $(F_{H,\theta}^{1,1})^{\perp}$ is the trace free part of $F_{H,\theta}^{1,1}$.

For any torsion free coherent Higgs sheaf $(\mathcal{F},\theta_{\mathcal{F}})$, define the $\omega$-degree and $\omega$-slope by
\begin{equation}
\deg_{\omega}(\mathcal{F})=\int_{X}c_{1}(\det{\mathcal{F}})\wedge\frac{\omega^{n-1}}{(n-1)!}
\end{equation}
and
\begin{equation}
\mu_{\omega}(\mathcal{F})=\frac{\deg_{\omega}(\mathcal{F})}{\mbox{rank}(\mathcal{F})}.
\end{equation}
We call $(E,\bar{\partial}_{E},\theta)$ is stable (semi-stable) if $\mu_{\omega}(\mathcal{F})<(\leq)\mu_{\omega}(E)$ for every proper coherent Higgs subsheaf $\mathcal{F}\subset E$. A Higgs bundle $(E,\bar{\partial}_{E},\theta)$ is said to be admitting an approximate Hermitian-Einstein structure, if for every $\delta >0$, there exists a Hermitian metric $H_{\delta}$ such that
\begin{equation}
\sup_X|\sqrt{-1}\Lambda_{\omega }(F_{H_{\delta}}+[\theta,\theta^{*{H_{\delta}}}])-\lambda \cdot \textmd{Id}_E|_{H_{\delta}}<\delta.
\end{equation}

Let's consider the following perturbed equation on $(X, \omega)$:
\begin{equation} \label{eq}
\sqrt{-1}\Lambda_{\omega } (F_{H}+[\theta,\theta^{*H}])-\lambda \cdot \textmd{Id}_E+\epsilon \log (K^{-1}H)=0,
\end{equation}
where $K$ is a fixed Hermitian metric on $E$.
Making use of the continuous method in \cite{UhYau} and applying the Fredholmness of the elliptic operators,  one can see that the above perturbed equation can be solved for any $\epsilon \in (0, 1]$. Letting $H_{\epsilon}$ be the solution of the equation (\ref{eq}), we can conclude that (\cite[Lemma 2.2]{NZ})
\begin{equation}
\frac{1}{2}\sqrt{-1}\Lambda_{\omega }\partial \overline{\partial }(|\log (K^{-1}H_{\epsilon})|^2_{K})\geq \epsilon |\log (K^{-1}H_{\epsilon})|^2_{K}-|\sqrt{-1}\Lambda_{\omega }F_{K, \theta}-\lambda \Id_{E}|_{K}|\log (K^{-1}H_{\epsilon})|_{K},
\end{equation}
and then
\begin{equation}\label{t2}
\sup_{X} \epsilon |\log (K^{-1}H_{\epsilon})|_{K}\leq \sup_{X}|\sqrt{-1}\Lambda_{\omega }F_{K, \theta}-\lambda \Id_{E}|_{K}.
\end{equation}
By conformal transformation, we can assume that the background metric $K$ satisfies
\begin{equation}
\tr (\sqrt{-1}\Lambda_{\omega } (F_{K}+[\theta,\theta^{*K}])-\lambda \cdot \textmd{Id}_E)=0.
\end{equation}
Then, the equation (\ref{eq}) and the maximum principle imply
\begin{equation}\label{det1}
\det (K^{-1}H_{\epsilon })=1
\end{equation}
for every $\epsilon \in (0, 1]$.

If the Higgs bundle $(E,\bar{\partial}_{E},\theta)$ is semi-stable, we have (\cite[Theorem 3.2]{NZ})
\begin{equation}
\sup_{X} \epsilon |\log (K^{-1}H_{\epsilon})|_{K}=\sup_{X} \epsilon |\log (K^{-1}H_{\epsilon})|_{H_{\epsilon}} \rightarrow 0,
\end{equation}
as $\epsilon \rightarrow 0$. Furthermore, we have the following theorem:

\medskip

\begin{theorem}[\cite{NZ}]\label{thm:3}
Let $(X,\omega)$ be a compact Hermitian manifold of dimension $n$ satisfying $\partial\bar{\partial}\omega^{n-1}=0$, and $(E,\bar{\partial}_{E},\theta)$ be a Higgs bundle over $X$. Then $(E,\bar{\partial}_{E},\theta)$ is semi-stable if and only if it admits an approximate Hermitian-Einstein structure.
\end{theorem}

\medskip

For any Hermitian metric $H$ on $E$, applying Lemma 2.7 in \cite{S2}, we obtain
\begin{equation}\label{eqn:2}
|\sqrt{-1}\Lambda_\omega[\theta, \theta^{\ast H}]|_{H}=|[\theta , \theta^{\ast H}]|_{H, \omega}\geq a_1|\theta |_{H, \omega}^2-a_2|\theta |_{K, \omega}^2,
\end{equation}
where $a_1$ and $a_2$ are positive constants depending only on $r$ and $n$. By choosing  complex coordinates $\{z^{1}, \cdots , z^{n}\}$ at the considered point, we deduce:
\begin{equation}\label{t1}
\begin{split}
\sqrt{-1}\Lambda_{\omega }\partial \bar{\partial }|\theta |_{H_{\epsilon}, \omega }^{2} =&|\nabla \theta |_{H_{\epsilon}, \omega }^{2} +Re \langle  g^{\alpha \bar{\beta}}(\nabla_{\bar{\partial}_{\beta }}\nabla_{\partial_{\alpha } }-\nabla_{\partial_{\alpha } }\nabla_{\bar{\partial}_{\beta }})\theta  , \theta\rangle \\
&+ 2Re \langle  g^{\alpha \bar{\beta}}\nabla_{\partial_{\alpha } }\nabla_{\bar{\partial}_{\beta }} \theta  , \theta \rangle \\
=& Re\langle-[\sqrt{-1}\Lambda_{\omega}F_{H_{\epsilon}},\theta],\theta\rangle-Re\langle\theta\sharp R,\theta\rangle +|\nabla\theta|_{H_{\epsilon},\omega }^{2}\\
& + 2Re \langle  g^{\alpha \bar{\beta}}\nabla_{\partial_{\alpha } }\nabla_{\overline{\partial}_{\beta }} \theta  , \theta
\rangle \\
\geq & Re \langle [\sqrt{-1}\Lambda_{\omega}[\theta , \theta^{\ast H_{\epsilon }}],\theta ] , \theta \rangle+Re \langle [\epsilon \log (K^{-1}H_{\epsilon}),\theta ] , \theta \rangle\\
&+|\nabla\theta |_{H_{\epsilon}, \omega }^{2}-\check{C}_{1}|\theta |_{H_{\epsilon}, \omega }^{2}-\check{C}_{2}|\nabla\theta |_{H_{\epsilon}, \omega }|\theta |_{H_{\epsilon}, \omega }, \\
\end{split}
\end{equation}
where we have used that $H_{\epsilon}$ satisfies the perturbed equation (\ref{eq}) and $\overline{\partial}_{E}\theta =0$, $\check{C}_{1}$ and $\check{C}_{2}$ are positive constants depending only on the geometry of $(X, \omega )$. Of course (\ref{eqn:2}) implies:
\begin{equation}\label{eqn:3}
\begin{split}
Re \langle [\sqrt{-1}\Lambda_{\omega}[\theta , \theta^{\ast H_{\epsilon }}],
\theta ] , \theta \rangle &=
|\sqrt{-1}\Lambda_\omega[\theta, \theta^{\ast H_{\epsilon}}]|_{H_{\epsilon}, \omega }^{2}\\
&\geq  \frac{a_1^2}{2}|\theta |_{H_{\epsilon }, \omega}^4-a_2^2|\theta |_{K, \omega}^4.
\end{split}
\end{equation}
Combining this and (\ref{t1}), (\ref{t2}) gives us that
\begin{equation}\label{t3}
\sqrt{-1}\Lambda_{\omega }\partial \overline{\partial }|\theta |_{H_{\epsilon}, \omega }^{2}
\geq \frac{1}{2}|\nabla\theta |_{H_{\epsilon}, \omega }^{2}+\frac{a_1^2}{2}|\theta |_{H_{\epsilon }, \omega}^4-\check{C}_{3}|\theta |_{H_{\epsilon}, \omega }^{2}-a_2^2|\theta |_{K, \omega}^4,
\end{equation}
where $\check{C}_{3}$ is a positive constant depending only on $\sup_{X}|\Lambda_{\omega }F_{K, \theta}|_{K}$ and the geometry of $(X, \omega )$. Then the maximum principle means:
\begin{equation}\label{t4}
\sup_{X}|\theta |_{H_{\epsilon}, \omega }^{2}
\leq \check{C}_{4},
\end{equation}
where $\check{C}_{4}$ is a positive constant depending only on $\sup_{X}|\theta |_{K, \omega}$, $\sup_{X}|\Lambda_{\omega }F_{K, \theta}|_{K}$ and the geometry of $(X, \omega )$.

For a Higgs bundle $(\tilde{E},\bar{\partial}_{\tilde{E}},\tilde{\theta })$ and a Hermitian metric $\tilde{H}$ on $\tilde{E}$, set $D_{\tilde{E}}^{''}=\bar{\partial}_{\tilde{E}}+\tilde{\theta }$ and $D_{\tilde{H}}^{'}=\partial_{\tilde{H}}+\tilde{\theta }^{*\tilde{H}}$, sometimes we omit the subscript for simplicity. Then $(D_{\tilde{E}}^{''})^{2}=0$ and $(D_{\tilde{H}}^{'})^{2}=0$. The Dolbeault cohomology groups are defined by
\begin{equation}
H^{i}_{Dol}(\tilde{E})=\frac{\mbox{Ker} (D_{\tilde{E}}^{''}:\mathcal{A}^{i}(\tilde{E})\rightarrow\mathcal{A}^{i+1}(\tilde{E}))}{\mbox{Im} (D_{\tilde{E}}^{''}:\mathcal{A}^{i-1}(\tilde{E})\rightarrow\mathcal{A}^{i}(\tilde{E}))}.
\end{equation}
For a flat bundle $(V,D)$, the de Rham cohomology groups are defined by
\begin{equation}
H^{i}_{DR}(V)=\frac{\mbox{Ker} (D:\mathcal{A}^{i}(V)\rightarrow\mathcal{A}^{i+1}(V))}{\mbox{Im} (D:\mathcal{A}^{i-1}(V)\rightarrow\mathcal{A}^{i}(V))}.
\end{equation}

\begin{lemma}\label{lem:l1}
Suppose $(\tilde{E},\bar{\partial}_{\tilde{E}},\tilde{\theta }, \tilde{H})$ is a Higgs bundle over $(X, \omega )$ with $F_{\tilde{H},\tilde{\theta }}=(D_{\tilde{H},\tilde{\theta}})^{2}=0$. If $\partial\bar{\partial}\omega^{n-1}=0$, then $H_{DR}^{0}(X,\tilde{E})= H_{Dol}^{0}(X,\tilde{E})$.
\end{lemma}
\begin{proof}
For any $e\in H^{0}_{DR}(X,\tilde{E})$, we have
\begin{equation}\label{x:2}
\sqrt{-1}\Lambda_{\omega}\partial\bar{\partial}|e|_{\tilde{H}}^{2}=|D_{E}^{''}e|_{\tilde{H}}^{2}+|D^{'}_{\tilde{H}}e|_{\tilde{H}}^{2}.
\end{equation}
Integrating both sides of this equation derives $e\in H^{0}_{Dol}(X,\tilde{E})$. Similarly, if $e\in H^{0}_{Dol}(X,\tilde{E})$, we get
\begin{equation}
\sqrt{-1}\Lambda_{\omega}\partial\bar{\partial}|e|_{\tilde{H}}^{2}=|\partial_{\tilde{H}}e|_{\tilde{H}}^{2}+|\theta^{*}(e)|_{\tilde{H}}^{2}.
\end{equation}
Hence $e\in H_{DR}^{0}(X,\tilde{E})$.
\end{proof}

\medskip

\begin{proof}[Proof of Theorem \ref{thm:2-1}]
	Suppose $(E,\bar{\partial}_{E},\theta)$ is a polystable Higgs bundle with the vanishing Chern number (\ref{CC}). By the Hitchin-Kobayashi correspondence for Higgs bundles over compact Gauduchon manifolds(\cite{bu,LY,MA}),  there is a Hermitian-Einstein metric $H$ on $(E,\bar{\partial}_{E},\theta)$. The identity (\ref{Bo1}) implies $\partial _{H}\theta =0$ and $F_{H}+[\theta , \theta ^{\ast H}]=\frac{1}{r}\tr F_{H}\otimes \Id_{E}$, then the curvature of the related Hitchin-Simpson connection $D_{H,\theta}$ satisfies
\begin{equation}
\sqrt{-1}F_{D_{H, \theta }}=\frac{\sqrt{-1}}{r}\tr F_{H}\otimes \Id_{E}.
\end{equation}
Since $[\alpha]_{DR}=[\frac{\sqrt{-1}}{r}\tr F_{H}]_{DR} = \frac{2\pi}{ r} c_{1}(E)$, there exists a $1$-form $\eta $ such that $\alpha = \frac{\sqrt{-1}}{r}\tr F_{H}+ d\eta $. Set $\hat{D}_{H, \theta }=D_{H,\theta} -\sqrt{-1}\eta \otimes \Id_{E}$, then
\begin{equation}
\sqrt{-1}F_{\hat{D}_{H, \theta }}=\alpha \otimes \Id_{E}.
\end{equation}
If assume that $\Lambda_{\omega}\alpha =\lambda $ in advance, we have \begin{equation}\label{alpha}\frac{\sqrt{-1}}{r}\tr F_{H}=\alpha .\end{equation} In fact, the conditions that $\partial\bar{\partial}\omega^{n-1}=\partial\bar{\partial}\omega^{n-2}=0$ and $\int_{X}\partial [\eta]\wedge\frac{\omega^{n-1}}{(n-1)!}=0$ for any Dolbeault class $[\eta]\in H^{0,1}(X)$ ensure the global $\partial \overline{\partial }$-lemma holds on $(X, \omega)$. Because $\frac{\sqrt{-1}}{r}\tr F_{H}$ and $\alpha$ are both real $(1, 1)$-forms in the same class,  there exists a real function $f$ such that
\begin{equation}\frac{\sqrt{-1}}{r}\tr F_{H}=\alpha +\sqrt{-1} \partial \overline{\partial }f.\end{equation}
The Hermitian-Einstein equation means that $f$ is constant, so we have (\ref{alpha}).

	Conversely, assume $(E,D)$ is a semi-simple projectively flat bundle. Since $(X, \omega )$ satisfies the  global $\partial \overline{\partial }$-lemma,  $BC(E,D)=0$. By Theorem \ref{thm0}, there is a unique harmonic metric $H$ such that $G_{H}=0$. Thus $(E,\bar{\partial}_{H},\psi_{H}^{1,0})$ is a polystable Higgs bundle with vanishing Chern number (\ref{CC}).
	
	Let $(\bar{\partial}_{1}, \theta_{1} )$ and $(\bar{\partial}_{2}, \theta_{2} )$ be two poly-stable Higgs structures on $E$ with vanishing Chern number (\ref{CC}), $H_{1}$ and $H_{2}$ be the corresponding Hermitian-Einstein metrics, $D_{1}$ and $D_{2}$ be the corresponding Hitchin-Simpson connections.  Set the induced Higgs bundle \begin{equation}(\tilde{E}, \bar{\partial}_{\tilde{E}}, \tilde{\theta})=(E, \bar{\partial}_{1}, \theta_{1} )\otimes (E, \bar{\partial}_{2}, \theta_{2} )^{\ast},\end{equation}
and the induced bundle
\begin{equation}(\tilde{E},  \tilde{D})=(E, D_{1} )\otimes (E, D_{2} )^{\ast}.\end{equation}
It is easy to see that $(\tilde{E},  \tilde{D})$ is a flat bundle. From Lemma \ref{lem:l1}, we know $H^{0}_{DR}(X,\tilde{E})=H^{0}_{Dol}(X,\tilde{E})$. Hence there is an equivalence of categories between the category of poly-stable Higgs bundles with vanishing Chern number (\ref{CC}) and the category of semi-simple projectively flat bundles. So we have constructed a one-to-one correspondence between the moduli spaces $\mathcal{C}_{Dol}^{s}(E)$ and $\mathcal{C}_{DR}^{s}(E, \alpha )$.
\end{proof}

\section{The Yang-Mills-Higgs flow}
Suppose $(X, J, g)$ is a compact Hermitian manifold of dimension $n$ with the associated K\"ahler form $\omega$ satisfying $\partial\bar{\partial}\omega^{n-1}=\partial\bar{\partial}\omega^{n-2}=0$. Let $(E,H_{0})$ be a Hermitian vector bundle over $X$,
${\mathcal A}_{H_{0}}$ be the space of connections on $E$
compatible with the metric $H_{0}$, and ${\mathcal A}^{1,1}_{H_{0}}$
be the space of unitary integrable connections on $E$. Set
\begin{equation}
\mathcal{B}_{E,H_{0}}=\{(A,\phi)\in\mathcal{A}^{1,1}_{0}\times\Omega^{1,0}(\mbox{End}(E))\ | \ \bar{\partial}_{A}\phi=0,\phi\wedge\phi=0\}.
\end{equation}
 The element $(A,\phi)\in\mathcal{B}_{E,H_{0}}$ is called a Higgs pair which determines a Higgs bundle $(E, D^{0,1}_{A}, \phi )$.
Define the Yang-Mills-Higgs functional on  the space of Higgs pairs $\mathcal{B}_{E,H_{0}}$ by
\begin{equation}
\textrm{YMH}(A,\phi)=\int_{X}|F_{A}+[\phi,\phi^{*H_{0}}]|_{H_{0},\omega}^{2}+2|\partial_{A}\phi|_{H_{0},\omega}^{2}dv_{g},
\end{equation}
where $dv_g$ is the volume form.

Consider the following modified Yang-Mills-Higgs flow (it is the gradient flow of Yang-Mills-Higgs functional in K\"ahler manifolds case):
\begin{equation}\label{mymh}
\left\{
\begin{split}
 \frac{\partial A(t)}{\partial t}&=-\sqrt{-1}(\partial_{A(t)}-\bar{\partial}_{A(t)})\Lambda_{\omega}(F_{A(t)}+[\phi(t),\phi^{*H_{0}}(t)]),  \\
 \frac{\partial\phi(t)}{\partial t}&=-[\sqrt{-1}\Lambda_{\omega}(F_{A(t)}+[\phi(t),\phi^{*H_{0}}(t)]),\phi(t)],  \\
 A(0)&=A_{0},\ \ \phi(0)=\phi_{0}.
\end{split}\right.
\end{equation}
As that in \cite{LZ}, we can obtain the long time existence and uniqueness of solution to the modified Yang-Mills-Higgs flow. Indeed, let $H(t)$ be the long time solution of the following Hermitian-Yang-Mills flow on the Higgs bundle $(E,\bar{\partial}_{A_{0}},\phi_{0})$
\begin{equation}\label{HYMF}
\left\{
\begin{split}
 &H^{-1}(t)\frac{\partial H(t)}{\partial t}=-2(\sqrt{-1}\Lambda_{\omega}(F_{H(t)}+[\phi_{0},\phi_{0}^{*H(t)}])-\lambda \Id),  \\
 &H(0)=H_{0},
\end{split}
\right.
\end{equation}
where $\lambda=\frac{2\pi}{\mbox{Vol}(X,\omega)}\mu_{\omega}(E,\bar{\partial}_{A_{0}}).$
There is a family of complex gauge transformations $\sigma(t)$ satisfying $\sigma^{*H_{0}}(t)\sigma(t)=H_{0}^{-1}H(t)$, such that $(A(t),\phi(t))=\sigma(t)(A_{0},\phi_{0})$ is a solution of the modified Yang-Mills-Higgs flow, where
 \begin{equation}
\overline{\partial }_{\sigma(A_{0})}=\sigma \circ \overline{\partial
}_{A_{0}}\circ \sigma^{-1}, \quad \partial _{\sigma (A_{0})}=(\sigma^{\ast
H_{0}})^{-1} \circ \partial _{A_{0}}\circ \sigma^{\ast H_{0}};
\end{equation}
\begin{equation}
\sigma (\phi_{0} )=\sigma \circ \phi_{0} \circ \sigma^{-1}.
\end{equation}
When there is no confusion, we denote $(A(t),\phi(t))$, $H(t)$ and $\sigma(t)$ by $(A,\phi)$, $H$ and $\sigma$ for simplicity. It is straightforward to check that
\begin{equation}
\begin{split}
\bar{\partial}_{A}\phi^{*H_{0}}&=\sigma\circ\bar{\partial}_{A_{0}}\phi_{0}^{*H}\circ \sigma^{-1},\\
\partial_{A}\phi&=\sigma\circ\partial_{H}\phi_{0}\circ \sigma^{-1},\\
F_{A}+[\phi,\phi^{*H_{0}}]&=\sigma\circ(F_{H}+[\phi_{0},\phi_{0}^{*H}])\circ \sigma^{-1},\\
D_{A,\phi}(\Lambda_{\omega}(F_{A}+[\phi,\phi^{*H_{0}}]))&=\sigma\circ(D_{H,\phi_{0}}(\Lambda_{\omega}(F_{H}+[\phi_{0},\phi_{0}^{*H}])))\circ \sigma^{-1}.
\end{split}
\end{equation}

Along the flow, we have
\begin{equation}\label{re1}
\begin{split}
&\textrm{YMH}(A(t),\phi(t))\\
=&\int_{X}|F_{A(t)}+[\phi(t),\phi^{*H_{0}}(t)]|_{H_{0}}^{2}+2|\partial_{A(t)}\phi(t)|_{H_{0}}^{2}dv_{g}\\
=&\int_{X}|F_{H(t)}+[\phi_{0},\phi_{0}^{*H(t)}]|_{H(t)}^{2}+2|\partial_{A_{0}}\phi_{0}|_{H(t)}^{2}dv_{g}\\
=&\int_{X}|\Lambda_{\omega}(F_{H(t)}+[\phi_{0},\phi_{0}^{*H(t)}])|_{H(t)}^{2}dv_{g}-8\pi^{2}ch_{2}(E,\bar{\partial}_{A_{0}})[\omega^{n-2}]\\
=&\int_{X}|\sqrt{-1}\Lambda_{\omega}(F_{H(t)}+[\phi_{0},\phi_{0}^{*H(t)}])-\lambda \Id_{E}|_{H(t)}^{2}dv_{g}\\
&+\lambda^{2}\mbox{rank}(E)\mbox{Vol}(X,\omega)-8\pi^{2}ch_{2}(E,\bar{\partial}_{A_{0}})[\omega^{n-2}]\\
=&\int_{X}|\sqrt{-1}\Lambda_{\omega}(F_{A(t)}+[\phi(t),\phi^{*H_{0}}(t)])-\lambda \Id_{E}|_{H_{0}}^{2}dv_{g}\\
&+\lambda^{2}\mbox{rank}(E)\mbox{Vol}(X,\omega)-8\pi^{2}ch_{2}(E,\bar{\partial}_{A_{0}})[\omega^{n-2}].
\end{split}
\end{equation}
Furthermore, if $H(t)$ is a solution of the Hermitian-Yang-Mills flow (\ref{HYMF}), by the same discussion as that in \cite[Lemma 6.1]{S1}, we get (\cite[Lemma 1]{Jac1} or \cite[Proposition 2.1]{ZZZ})
\begin{equation}\label{mc10}
(\frac{\partial}{\partial t}-2\sqrt{-1}\Lambda_{\omega}\partial\bar{\partial})\tr (\Psi (t))=0,
\end{equation}
\begin{equation}\label{mc1}
(\frac{\partial}{\partial t}-2\sqrt{-1}\Lambda_{\omega}\partial\bar{\partial})|\Psi (t)|_{H(t)}^{2}=-2|D_{H(t),\phi_{0}}\Psi (t)|_{H(t)}^{2},
\end{equation}
where we have set $\Psi (t)=\sqrt{-1}\Lambda_{\omega}(F_{H(t)}+[\phi_{0},\phi_{0}^{*H(t)}])-\lambda \Id$ for simplicity. By the conformal transformation, we can suppose that $\tr \Psi (0)=0$, then we have
\begin{equation}\label{tr01}
\tr \Psi (t)=0,
\end{equation}
\begin{equation}\label{tr02}
\det (H_{0}^{-1}H(t))=1, \quad \text{and} \quad \tr F_{H(t)}=\tr F_{H_{0}}
\end{equation}
for all $t\geq 0$.
Based on (\ref{re1}) and (\ref{mc1}), we establish the following proposition.

\begin{prop}\label{p}
Let $(A,\phi)$ be a solution of the modified Yang-Mills-Higgs flow (\ref{mymh}), then
$
\sup_{X}|\sqrt{-1}\Lambda_{\omega}(F_{A}+[\phi,\phi^{*H_{0}}])-\lambda \Id|_{H_{0}}
$ is uniformly bounded
and $\rm{YMH}(t)$ is decreasing along the flow. Furthermore, if $\tr \Psi (0)=0$, then $\tr F_{A(t)}=\tr F_{A_{0}}$ for all $t\geq 0$.
\end{prop}

It is well known that there are two connections on the tangent bundle $TX$: Chern connection and Levi-Civita connection. They are coincide on K\"ahler manifolds, but they are different in non-K\"ahler manifolds case. In the rest of this section, we use $\hat{\nabla}$ ($\nabla$) to denote the Chern connection (the Levi-Civita connection) on $X$ with respect to $\omega $, and $\hat{\nabla}_{A}$ ($\nabla_{A} $) to denote the induced connection in $\Lambda ^{\bullet} X \otimes \mbox{End}(E)$ by $D_{A}$ and $\hat{\nabla}$ ($\nabla$). Notice that for a smooth function $f$ on a non-K\"ahler manifold, there holds
\begin{equation}\label{la1}
\Delta f=2\sqrt{-1}\Lambda_{\omega}\partial\bar{\partial} f+\langle df,V\rangle,
\end{equation}
where $V=*(-\sqrt{-1}(\bar{\partial}-\partial)\omega^{n-1})$ and $\Delta $ is the Beltrami-Laplacian. One has already known that
\begin{equation}
\hat{\nabla } g=0, \quad \hat{\nabla }J=0, \quad and \quad T^{1,1}=0,
\end{equation}
where $T$ is the torsion tensor of $\hat{\nabla }$. In a local complex coordinate $\{z^{1}, \cdots , z^{n}\}$, we have
\begin{equation}
\hat{\nabla}_{\frac{\partial }{\partial \bar{z}^{\beta }}}\frac{\partial }{\partial z^{\alpha }}=\hat{\nabla}_{\frac{\partial }{\partial z^{\alpha }}}\frac{\partial }{\partial \bar{z}^{\beta }}=0.
\end{equation}
Let us recall the Bianchi identity $D_{A}F_{A}=0$, equivalently we have
\begin{equation}\label{bianchi1}
\sum \hat{\nabla}_{A,X}F_{A}(Y, Z)=-\sum F_{A}(T(X, Y), Z),
\end{equation}
where $\sum$ is the rotation sum of $X, Y, Z$ and $\hat{\nabla}_{A, X}$ is the covariant derivative in the direction $X$ with respect to the connection $\hat{\nabla}_{A}$. In fact, for any $\theta \in \Omega^{2}(X)$, it holds that
\begin{equation}\label{bianchi2}
d\theta (X, Y, Z)= \sum \hat{\nabla}_{X}\theta(Y, Z)+\sum \theta (T(X, Y), Z).
\end{equation}
For every Higgs pair $(A, \phi )$, we know that $\overline{\partial }_{A}\phi=0$, this is equivalent to
\begin{equation}\label{bianchi3}
 \hat{\nabla}_{A}^{0,1}\phi =0.
\end{equation}
For this reason, in the following calculations, we will choose the Chern connection $\hat{\nabla }$ on the base manifold.

As that in  \cite[Section 2.2]{LZ}, the following result comes from a direct computation.
\begin{prop}\label{p:10}
Let $(A,\phi)$ be a solution of the modified Yang-Mills-Higgs flow (\ref{mymh}), then
\begin{equation}
(2\sqrt{-1}\Lambda_{\omega}\partial\bar{\partial}-\frac{\partial}{\partial t})|\phi|^{2}\geq2|\hat{\nabla}_{A}\phi|^{2}+C_{1}(|\phi|^{2}+1)^{2}-C_{2}(|\phi|^{2}+1)-C_{3}|\hat{\nabla}_{A}\phi|\cdot|\phi|,
\end{equation}
where the constants $C_{1},C_{2},C_{3}$ depend only on $\sup_{X}|\phi_{0}|_{H_{0}}$ and the geometry of $(X,\omega)$. So we get $\sup_{X}|\phi|^{2}\leq \max\{\sup_{X}|\phi_{0}|^{2},C_{2}/C_{1}+C_{3}^{2}/4C_{1}\}$ for every $t\geq 0$.
\end{prop}

\medskip

\begin{prop}\label{p:1}
Let $(A,\phi)$ be a solution of the modified Yang-Mills-Higgs flow (\ref{mymh}), and $\xi$ be a closed $(1, 1)$-form on $(X, \omega)$ with $\Lambda_{\omega }\xi =\lambda $ for some constant $\lambda$. Then
\begin{equation}\label{k01}
\begin{split}
&(2\sqrt{-1}\Lambda_{\omega}\partial\bar{\partial}-\frac{\partial}{\partial t})|\hat{\nabla}_{A}\phi|^{2}-2|\hat{\nabla}_{A}\hat{\nabla}_{A}\phi|^{2}\\
\geq&-C_{4}(|F_{A}|+|R_{\omega }|+|\phi|^{2})|\hat{\nabla}_{A}\phi|^{2}\\
&-C_{5}(|\hat{\nabla} R_{\omega }|+|T|\cdot |F_{A}|)|\phi|\cdot |\hat{\nabla}_{A}\phi|,\\
\end{split}
\end{equation}
 and
\begin{equation}\label{k02}
\begin{split}
&(2\sqrt{-1}\Lambda_{\omega}\partial\bar{\partial}-\frac{\partial}{\partial t})(|F_{A}+[\phi,\phi^{*H_{0}}]-\xi\otimes \Id_{E}|^{2}+2|\partial_{A}\phi|^{2})\\
\geq &2|\hat{\nabla}_{A}\partial_{A}\phi|^{2}+|\hat{\nabla}_{A}(F_{A}+[\phi,\phi^{*H_{0}}]-\xi\otimes \Id_{E})|^{2}\\
&-C_{6}(|T|^{2}+|\hat{\nabla}_{A}T|+|F_{A}|+|R_{\omega}|+|\hat{\nabla}_{A}\phi|+|\phi|^{2})(|F_{A}+[\phi,\phi^{*H_{0}}]-\xi\otimes \Id_{E}|^{2}+2|\partial_{A}\phi|^{2}),
\end{split}
\end{equation}
where the constants $C_{4},C_{5},C_{6}$ depend only on the dimension of $X$ and the rank of $E$, $R_{\omega }$ is the curvature tensor of the Chern connection $\hat{\nabla }$ with respect to $\omega$.
\end{prop}
\begin{proof}
For any point $p\in X$, we choose a local complex coordinate $\{z^{1}, \cdots , z^{n}\}$ centered at $p$ and with $g_{\alpha \bar{\beta}}(p)=\delta_{\alpha \beta}$. Using the Bianchi identity (\ref{bianchi1}), $\overline{\partial}_{A}\phi =0$ and $\phi \wedge \phi =0$, we obtain
\begin{equation}\label{k011}
\begin{split}
&2\sqrt{-1}\Lambda_{\omega}\partial\bar{\partial}(|\hat{\nabla}_{A}\phi|^{2})\\
=&2|\hat{\nabla}_{A}\hat{\nabla}_{A}\phi|^{2}-2Re\langle[\sqrt{-1}\Lambda_{\omega}F_{A},\hat\nabla_{A, \gamma }\phi],\hat\nabla_{A, \gamma }\phi\rangle\\
&-4Re\langle[\partial_{A,\gamma }(\sqrt{-1}\Lambda_{\omega}(F_{A}+[\phi,\phi^{*H_{0}}])),\phi],\hat\nabla_{A, \gamma }\phi\rangle\\
&+4Re\langle[[\hat{\nabla}_{A,\gamma }\phi_{\alpha},\phi_{\bar{\alpha}}^{*H_{0}}],\phi],\hat\nabla_{A, \gamma }\phi\rangle\\
&+2Re\langle[F_{l\bar{\alpha}},\phi]T_{\alpha\gamma}^{l},\hat\nabla_{A, \gamma }\phi\rangle -4Re\langle[F_{\gamma \bar{\alpha}},\hat\nabla_{A, \alpha }\phi],\hat\nabla_{A, \gamma }\phi\rangle\\
&+2Re\langle(\hat\nabla_{A, \alpha }(\phi_{l}R^{l}_{\beta \gamma \bar{\alpha }})+\hat\nabla_{A, \alpha }\phi_{l}R^{l}_{\beta \gamma \bar{\alpha }}+\hat\nabla_{A, l }\phi_{\beta}R^{l}_{\alpha \gamma \bar{\alpha }}) dz^{\beta},\hat\nabla_{A, \gamma }\phi\rangle,\\
\end{split}
\end{equation}
\begin{equation}\label{k021}
\begin{split}
&2\sqrt{-1}\Lambda_{\omega}\partial\bar{\partial}(2|\partial_{A}\phi|^{2})\\
=&4|\hat{\nabla}_{A}\partial_{A}\phi|^{2}-4Re\langle [\sqrt{-1}\Lambda_{\omega}F_{A},\partial_{A}\phi],\partial_{A}\phi\rangle+4Re\langle(D_{\gamma}^{A}\phi_{\beta}-D_{\beta}^{A}\phi_{\gamma})\hat{\nabla}_{\bar{\alpha}}\hat{\nabla}_{\alpha}dz^{\alpha}\wedge dz^{\beta},\partial_{A}\phi\rangle\\
&-8Re\langle[\partial_{A}(\sqrt{-1}\Lambda_{\omega}(F_{A}+[\phi,\phi^{*H_{0}}])),\phi],\partial_{A}\phi\rangle+8Re\langle[[\phi_{\gamma},\phi_{\bar{\alpha}}^{*}],D_{\alpha}^{A}\phi_{\beta}-D_{\beta}^{A}\phi_{\alpha}]dz^{\gamma}\wedge dz^{\beta},\partial_{A}\phi\rangle\\
&+8Re\langle[F_{l\bar{\alpha}}+[\phi_{l},\phi_{\bar{\alpha}}^{*}]-\xi_{l\bar{\alpha}}\otimes \Id_{E},\phi_{\beta}]T_{\alpha\gamma}^{l}dz^{\gamma}\wedge dz^{\beta},\partial_{A}\phi\rangle\\
&-8Re\langle[(F_{\gamma\bar{\alpha}}+[\phi_{\gamma},\phi_{\bar{\alpha}}^{*}]-\xi_{\gamma\bar{\alpha}}\otimes \Id_{E}),\hat{\nabla}_{A, \alpha}\phi_{\beta}]dz^{\gamma}\wedge dz^{\beta},\partial_{A}\phi\rangle\\
&-8Re\langle[[\phi_{\alpha},\partial_{A}\phi],\phi_{\bar{\alpha}}^{*}],\partial_{A}\phi\rangle ,
\end{split}
\end{equation}
\begin{equation}\label{k022}
\begin{split}
&2\sqrt{-1}\Lambda_{\omega}\partial\bar{\partial}|F_{A}+[\phi,\phi^{*}]-\xi\otimes \Id_{E}|^{2}\\
=&2|\hat{\nabla}_{A}(F_{A}+[\phi,\phi^{*}]-\xi\otimes \Id_{E})|^{2}\\
&+2Re\langle(\hat{\nabla}_{A,\alpha}\hat{\nabla}_{A,\bar{\alpha}}+\hat{\nabla}_{A,\bar{\alpha}}\hat{\nabla}_{A,\alpha})(F_{A}+[\phi,\phi^{*H_{0}}]-\xi\otimes \Id_{E}), F_{A}+[\phi,\phi^{*}]-\xi\otimes \Id_{E}\rangle ,\\
\end{split}
\end{equation}
\begin{equation}\label{k023}
\begin{split}
&\hat{\nabla}_{A,\alpha}\hat{\nabla}_{A,\bar{\alpha}}(F_{A}+[\phi,\phi^{*H_{0}}]-\xi\otimes \Id_{E})\\
=&-\hat{\nabla}_{A,\alpha}((F_{\beta\bar{l}}+[\phi_{\beta},\phi_{\bar{l}}^{*}]-\xi_{\beta\bar{l}}\otimes \Id_{E})T_{\bar{\alpha}\bar{\gamma}}^{\bar{l}})dz^{\beta}\wedge d\bar{z}^{\gamma}\\
&+\hat{\nabla}_{A,\alpha}([\phi_{\beta},D_{\bar{\alpha}}^{A}\phi_{\bar{\gamma}}^{*}-D_{\bar{\gamma}}^{A}\phi_{\bar{\alpha}}^{*}])dz^{\beta}\wedge d\bar{z}^{\gamma}\\
&+\hat{\nabla}_{A,\bar{\gamma}}([D_{\alpha}^{A}\phi_{\beta}-D_{\beta}^{A}\phi_{\alpha},\phi_{\bar{\alpha}}^{*}])dz^{\beta}\wedge d\bar{z}^{\gamma}\\
&+ [F_{\alpha\bar{\gamma}},(F_{\beta\bar{\alpha}}+[\phi_{\beta},\phi_{\bar{\alpha}}^{*}]-\xi_{\beta\bar{\alpha}}\otimes \Id_{E})]dz^{\beta}\wedge d\bar{z}^{\gamma}\\
&- R_{\beta\alpha\bar{\gamma}}^{l}(F_{l\bar{\alpha}}+[\phi_{l},\phi_{\bar{\alpha}}^{*}]-\xi_{l\bar{\alpha}}\otimes \Id_{E})dz^{\beta}\wedge d\bar{z}^{\gamma}\\
&- R_{\bar{\alpha}\alpha\bar{\gamma}}^{\bar{l}}(F_{\beta\bar{l}}+[\phi_{\beta},\phi_{\bar{l}}^{*}]-\xi_{\beta\bar{l}}\otimes \Id_{E})dz^{\beta}\wedge d\bar{z}^{\gamma}\\
&+ \hat{\nabla}_{A,\bar{\gamma}}(\hat{\nabla}_{A,\beta}(F_{\alpha\bar{\alpha}}+[\phi_{\alpha},\phi_{\bar{\alpha}}^{*}]-\xi_{\alpha\bar{\alpha}}\otimes \Id_{E}))dz^{\beta}\wedge d\bar{z}^{\gamma}\\
&- \hat{\nabla}_{A,\bar{\gamma}}(T_{\alpha\beta}^{l}(F_{l\bar{\alpha}}+[\phi_{l},\phi_{\bar{\alpha}}^{*}]-\xi_{l\bar{\alpha}}\otimes \Id_{E}))dz^{\beta}\wedge d\bar{z}^{\gamma},\\
\end{split}
\end{equation}
and
\begin{equation}\label{k024}
\begin{split}
&\hat{\nabla}_{A,\bar{\alpha}}\hat{\nabla}_{A,\alpha}(F_{A}+[\phi,\phi^{*H_{0}}]-\xi\otimes \Id_{E})\\
=& \hat{\nabla}_{A,\bar{\alpha}}(T_{\alpha\beta}^{l}(F_{l\bar{\gamma}}+[\phi_{l},\phi_{\bar{\gamma}}^{*}]-\xi_{l\bar{\gamma}}\otimes \Id_{E}))dz^{\beta}\wedge d\bar{z}^{\gamma}\\
&+ \hat{\nabla}_{A,\bar{\alpha}}([D_{\alpha}^{A}\phi_{\beta}-D_{\beta}^{A}\phi_{\alpha},\phi_{\bar{\gamma}}^{*}])dz^{\beta}\wedge d\bar{z}^{\gamma}\\
&- [F_{\beta\bar{\alpha}},(F_{\alpha\bar{\gamma}}+[\phi_{\alpha},\phi_{\bar{\gamma}}^{*}]-\xi_{\alpha\bar{\gamma}}\otimes \Id_{E})]dz^{\beta}\wedge d\bar{z}^{\gamma}\\
&+ R_{\alpha\beta\bar{\alpha}}^{l}(F_{l\bar{\gamma}}+[\phi_{l},\phi_{\bar{\gamma}}^{*}]-\xi_{l\bar{\gamma}}\otimes \Id_{E})dz^{\beta}\wedge d\bar{z}^{\gamma}\\
&- R_{\bar{\gamma}\beta\bar{\alpha}}^{\bar{l}}(F_{\alpha\bar{l}}+[\phi_{\alpha},\phi_{\bar{l}}^{*}]-\xi_{\alpha\bar{l}}\otimes \Id_{E})dz^{\beta}\wedge d\bar{z}^{\gamma}\\
&+ \hat{\nabla}_{A,\beta}(\hat{\nabla}_{A,\bar{\gamma}}(F_{\alpha\bar{\alpha}}+[\phi_{\alpha},\phi_{\bar{\alpha}}^{*}]-\xi_{\alpha\bar{\alpha}}\otimes \Id_{E}))dz^{\beta}\wedge d\bar{z}^{\gamma}\\
&- \hat{\nabla}_{A,\beta}(T_{\bar{\alpha}\bar{\gamma}}^{\bar{l}}(F_{\alpha\bar{l}}+[\phi_{\alpha},\phi_{\bar{l}}^{*}]-\xi_{\alpha\bar{l}}\otimes \Id_{E}))dz^{\beta}\wedge d\bar{z}^{\gamma}\\
&+ \hat{\nabla}_{A,\beta}([\phi_{\alpha},D_{\bar{\alpha}}^{A}\phi_{\bar{\gamma}}^{*}-D_{\bar{\gamma}}^{A}\phi_{\bar{\alpha}}^{*}])dz^{\beta}\wedge d\bar{z}^{\gamma}.\\
\end{split}
\end{equation}
On the other hand, by the Yang-Mills-Higgs flow (\ref{mymh}), we have
\begin{equation}\label{k012}
\begin{split}
&\frac{\partial }{\partial t}|\hat{\nabla}_{A}\phi|^{2}
=-4Re\langle[\partial_{A,\gamma }(\sqrt{-1}\Lambda_{\omega}(F_{A}+[\phi,\phi^{*H_{0}}])),\phi],\hat\nabla_{A, \gamma }\phi\rangle\\
&-2Re\langle[\sqrt{-1}\Lambda_{\omega}(F_{A}+[\phi , \phi^{\ast H_{0}}]),\hat\nabla_{A, \gamma }\phi],\hat\nabla_{A, \gamma }\phi\rangle,\\
\end{split}
\end{equation}
and
\begin{equation}\label{k025}
\begin{split}
&\frac{\partial}{\partial t}(|F_{A}+[\phi,\phi^{*H_{0}}]-\xi\otimes \Id_{E}|^{2}+2|\partial_{A}\phi|^{2})\\
=&-2Re\langle\sqrt{-1}(\bar{\partial}_{A}\partial_{A}-\partial_{A}\bar{\partial}_{A})\Lambda_{\omega}(F_{A}+[\phi,\phi^{*H_{0}}]),F_{A}+[\phi,\phi^{*H_{0}}]-\xi\otimes \Id_{E}\rangle\\
&-4Re\langle[[\sqrt{-1}\Lambda_{\omega}(F_{A}+[\phi,\phi^{*H_{0}}]),\phi],\phi^{*H_{0}}],F_{A}+[\phi,\phi^{*H_{0}}]-\xi\otimes \Id_{E}\rangle\\
&+2Re\langle[\sqrt{-1}\Lambda_{\omega}(F_{A}+[\phi,\phi^{*H_{0}}]),[\phi,\phi^{*H_{0}}]],F_{A}+[\phi,\phi^{*H_{0}}]-\xi\otimes \Id_{E}\rangle\\
&-4Re\langle 2[\partial_{A}(\sqrt{-1}\Lambda_{\omega}(F_{A}+[\phi,\phi^{*H_{0}}])),\phi]+[\sqrt{-1}\Lambda_{\omega}(F_{A}+[\phi,\phi^{*H_{0}}]),\partial_{A}\phi],\partial_{A}\phi\rangle .
\end{split}
\end{equation}
Then, (\ref{k011}) and (\ref{k012}) mean (\ref{k01}), at the same time, (\ref{k021}), (\ref{k022}), (\ref{k023}), (\ref{k024}) and (\ref{k025}) imply (\ref{k02}).
\end{proof}

\subsection{Monotonicity inequality and $\epsilon$-regularity}
Regard $X$ as a compact Riemannian manifold of dimension $2n$. For any $x_{0}\in X$, choose normal geodesic coordinates $\{x^{i}\}_{i=1}^{2n}$ in the geodesic ball $B_{r}(x_{0})$($r\leq i_X$, where $i_X$ is the injective radius of $X$) such that $x_{0}=(0,\cdots,0)$. Then it follows that
\begin{equation}\label{a}
|g_{ij}(x)-\delta_{ij}|\leq C|x|^{2},\ \ |\partial_{k}g_{ij}(x)|\leq C|x|,\ \ \forall x\in B_{r}(x_0),
\end{equation}
where $C$ is a positive constant depending only on $x_{0}$.

Let $u=(x,t)\in X\times\mathbb{R}$. For any $u_{0}=(x_{0},t_{0})\in X\times\mathbb{R}^{+}$, set
\begin{equation}
\begin{split}
&S_{r}(u_{0})=X\times\{t=t_{0}-r^{2}\},\\
&T_{r}(u_{0})=X\times [t_{0}-4r^{2},t_{0}-r^{2}],\\
&P_{r}(u_{0})=B_{r}(x_{0})\times [t_{0}-r^{2},t_{0}+r^{2}].
\end{split}
\end{equation}
For simplicity, we denote $S_{r}(0,0),T_{r}(0,0),P_{r}(0,0)$ by $S_{r},T_{r},P_{r}$.

\begin{prop}\label{p:2}
Let $(A,\phi)$ be a solution of the modified Yang-Mills-Higgs flow with the initial value $(A_{0},\phi_{0})$. For any $u_{0}=(x_{0},t_{0})\in X\times\mathbb{R}^{+}$, $r\in(0,\min\{\sqrt{t_{0}},i_{X}/2\})$, we have
\begin{equation}
\int_{P_{r}(u_{0})}|\hat{\nabla}_{A}\phi|^{2}dv_{g}dt\leq C_{9}r^{2n},
\end{equation}
where the constant $C_{9}$ depends only on the geometry of $(X,\omega)$ and $\sup_{X}|\phi_{0}|_{H_{0}}$.
\end{prop}
\begin{proof}
Choose $f\in C^{\infty}_{0}(B_{2r}(x_{0}))$ satisfying $0\leq f\leq 1$, $f\equiv 1$ on $B_{r}(x_{0})$, $|d f|\leq 4/r$ and $|\Delta f|\leq 16/r^{2}$ on $B_{2r}(x_{0})\setminus B_{r}(x_{0})$. According to Proposition \ref{p:10}, we obtain
\begin{equation}
\begin{split}
&(\Delta-\frac{\partial}{\partial t})f^{2}|\phi|^{2}\\
=&2|df|^{2}|\phi|^{2}+2f|\phi|^{2}\Delta f+4f\langle d f,d|\phi|^{2}\rangle+f^{2}(\Delta-\frac{\partial}{\partial t})|\phi|^{2}\\
\geq&2|df|^{2}|\phi|^{2}+2f|\phi|^{2}\Delta f-8f|df|\cdot |\phi|\cdot |\hat{\nabla}_{A}\phi|+2f^{2}|\hat{\nabla}_{A}\phi|^{2}\\
&-C_{2}f^{2}(|\phi|^{2}+1)-C_{3}f^{2}|\phi|\cdot |\hat{\nabla}_{A}\phi|-f^{2}|V|\cdot |\phi|\cdot |\hat{\nabla}_{A}\phi|\\
\geq& f^{2}|\hat{\nabla}_{A} \phi|^{2}-\bar{C}_{1}-\bar{C}_{2}(|df|^{2}+|\Delta f|),
\end{split}
\end{equation}
where the constants $\bar{C}_{1},\bar{C}_{2}$ depend only on the geometry of $(X,\omega)$ and $\sup_{X}|\phi_{0}|_{H_{0}}$, and $V$ is the $1$-form defined in (\ref{la1}). One can conclude this proposition by integrating both sides of the inequality over $X\times[t_{0}-r^{2},t_{0}+r^{2}]$.
\end{proof}

The fundamental solution of (backward) heat equation with singularity at $u_{0}=(x_{0}, t_{0})$ is
\begin{equation}
G_{u_{0}}(x,t)=G_{(x_{0},t_{0})}(x,t)=\frac{1}{(4\pi(t_{0}-t))^{2n}}\exp\Big(-\frac{|x-x_{0}|^{2}}{4(t_{0}-t)}\Big),\ \ t\leq t_{0}.
\end{equation}
For simplicity, denote $G_{(0,0)}(x,t)$ by $G(x,t)$.

Given $0<R\leq i_{X}$, we take $f\in\ C^{\infty}_{0}(B_{R})$ satisfying $0\leq f\leq 1$, $f\equiv 1$ on $B_{R/2}$ and $|\nabla f|\leq 2/R$ on $B_{R}\setminus B_{R/2}$. Let $(A,\phi)$ be a solution of the modified Yang-Mills-Higgs flow with initial value $(A_{0},\phi_{0})$ and $A^{'}=A+\phi+\phi^{*}$. For any $(x,t)\in X\times [0,+\infty)$, set
\begin{equation}
e(A,\phi)(x,t)=|F_{A^{'}}|^{2}
\end{equation}
and
\begin{equation}
\Phi(r)=r^{2}\int_{T_{r}(u_{0})}e(A,\phi)f^{2}G_{u_{0}}dv_{g}dt.
\end{equation}

\begin{theorem}\label{thm:4}
Let $(A,\phi)$ be a solution of the modified Yang-Mills-Higgs flow. For any $u_{0}=(x_{0},t_{0})\in X\times [0,T]$ and $0<r_{1}\leq r_{2}\leq \min\{R/2,\sqrt{t_{0}}/2\}$, we have
\begin{equation}
\begin{split}
\Phi(r_{1})\leq &C_{10}\exp(C_{10}(r_{2}-r_{1}))\Phi(r_{2})+C_{10}(r_{2}^{2}-r_{1}^{2})\rm{YMH}(A_{0},\phi_{0})\\
&+C_{10}R^{2-2n}\int_{P_{R}(u_{0})}e(A,\phi)dv_{g}dt,
\end{split}
\end{equation}
where the constant $C_{10}$ depends only on the geometry of $(X,\omega)$.
\end{theorem}

\begin{proof}
Choose normal geodesic coordinates $\{x^{i}\}_{i=1}^{2n}$ in the geodesic ball $B_{R}(x_{0})$. Let $x=r\tilde{x}$, $t=t_{0}+r^{2}\tilde{t}$. There holds that
\begin{equation}
\begin{split}
\Phi(r)&=r^{2}\int_{T_{r}(u_{0})}e(A,\phi)f^{2}G_{u_{0}}dv_{g}dt\\
&=r^{2}\int_{t_{0}-4r^{2}}^{t_{0}-r^{2}}\int_{\mathbb{R}^{2n}}e(A,\phi)(x,t)f^{2}(x)G_{u_{0}}(x,t)\sqrt{\det{(g_{ij})}}(x)dxdt\\
&=r^{4}\int_{-4}^{-1}\int_{\mathbb{R}^{2n}}e(A,\phi)(r\tilde{x},t_{0}+r^{2}\tilde{t})f^{2}(r\tilde{x})G(\tilde{x},\tilde{t})\sqrt{\det{(g_{ij})}}(r\tilde{x})d\tilde{x}d\tilde{t}.
\end{split}
\end{equation}
Then one can see that
\begin{equation}
\begin{split}
&\frac{d\Phi(r)}{dr}=4r^{3}\int_{-4}^{-1}\int_{\mathbb{R}^{2n}}e(A,\phi)(r\tilde{x},t_{0}+r^{2}\tilde{t})f^{2}(r\tilde{x})G(\tilde{x},\tilde{t})\sqrt{\det{(g_{ij})}}(r\tilde{x})d\tilde{x}d\tilde{t}\\
&+r^{3}\int_{-4}^{-1}\int_{\mathbb{R}^{2n}}\{x^{i}\partial_{i}e(A,\phi)(r\tilde{x},t_{0}+r^{2}\tilde{t})\}f^{2}(r\tilde{x})G(\tilde{x},\tilde{t})\sqrt{\det{(g_{ij})}}(r\tilde{x})d\tilde{x}d\tilde{t}\\
&+r^{3}\int_{-4}^{-1}\int_{\mathbb{R}^{2n}}\{2(t-t_{0})\partial_{t}e(A,\phi)(r\tilde{x},t_{0}+r^{2}\tilde{t})\}f^{2}(r\tilde{x})G(\tilde{x},\tilde{t})\sqrt{\det{(g_{ij})}}(r\tilde{x})d\tilde{x}d\tilde{t}\\
&+r^{4}\int_{-4}^{-1}\int_{\mathbb{R}^{2n}}e(A,\phi)(r\tilde{x},t_{0}+r^{2}\tilde{t})\frac{d}{dr}\{f^{2}(r\tilde{x})\sqrt{\det{(g_{ij})}}(r\tilde{x})\}G(\tilde{x},\tilde{t})d\tilde{x}d\tilde{t}\\
=&I_{1}+I_{2}+I_{3}+I_{4}.
\end{split}
\end{equation}
First, we have
\begin{equation}\label{eqn:31}
\begin{split}
I_{2}&=r\int_{T_{r}(u_{0})}\{x^{i}\partial_{i}e(A,\phi)(x,t)\}f^{2}(x)G_{u_{0}}(x,t)dv_{g}dt\\
&=2rRe\int_{T_{r}(u_{0})}\langle x^{i}\nabla_{i}F_{A^{'}},F_{A^{'}}\rangle f^{2}G_{u_{0}}dv_{g}dt\\
&=2rRe\int_{T_{r}(u_{0})}\langle x^{i}\nabla_{i}F_{A^{'}}(\partial_{j},\partial_{k})dx^{j}\wedge dx^{k}, F_{A^{'}}\rangle f^{2}G_{u_{0}}dv_{g}dt.
\end{split}
\end{equation}
From the Bianchi identity
\begin{equation}
D_{A^{'}}F_{A^{'}}=D_{A}F_{A^{'}}+[\phi+\phi^{*}, F_{A^{'}}]=0,
\end{equation}
it follows that
\begin{equation}\label{eqn:32}
\begin{split}
\nabla_{i}F_{A^{'}}(\partial_{j},\partial_{k})=&-[\phi_{i}+\phi^{*}_{i},F_{A^{'},jk}]+\nabla_{j}F_{A^{'}}(\partial_{i},\partial_{k})+[\phi_{j}+\phi^{*}_{j},F_{A^{'},ik}]\\
&-\nabla_{k}F_{A^{'}}(\partial_{i},\partial_{j})-[\phi_{k}+\phi^{*}_{k}, F_{A^{'},ij}],
\end{split}
\end{equation}
and
\begin{equation}\label{eqn:33}
\begin{split}
&\langle x^{i}\nabla_{j}F_{A^{'}}(\partial_{i},\partial_{k})dx^{j}\wedge dx^{k}, F_{A^{'}}\rangle\\
=&\langle x^{i}D_{A}(F_{A^{'},ik}dx^{k})-x^{i}(F_{A^{'}}(\nabla_{j}\partial_{i},\partial_{k})+F_{A^{'}}(\partial_{i},\nabla_{j}\partial_{k}))dx^{j}\wedge dx^{k},F_{A^{'}}\rangle\\
=&\langle D_{A}(x^{i}F_{A^{'},ik}dx^{k})-x^{i}F_{A^{'}}(\nabla_{j}\partial_{i},\partial_{k})dx^{j}\wedge dx^{k},F_{A^{'}}\rangle-|F_{A^{'}}|^{2},
\end{split}
\end{equation}
and
\begin{equation}\label{eqn:34}
\begin{split}
&-\langle x^{i}\nabla_{k}F_{A^{'}}(\partial_{i},\partial_{j})dx^{j}\wedge dx^{k}, F_{A^{'}}\rangle\\
=&\langle x^{i}D_{A}(F_{A^{'},ij}dx^{j})+x^{i}(F_{A^{'}}(\nabla_{k}\partial_{i},\partial_{j})+F_{A^{'}}(\partial_{i},\nabla_{k}\partial_{j}))dx^{j}\wedge dx^{k},F_{A^{'}}\rangle\\
=&\langle D_{A}(x^{i}F_{A^{'},ij}dx^{j})+x^{i}F_{A^{'}}(\nabla_{k}\partial_{i},\partial_{j})dx^{j}\wedge dx^{k},F_{A^{'}}\rangle-|F_{A^{'}}|^{2}.
\end{split}
\end{equation}
Under the condition $\partial\bar{\partial}\omega^{n-2}=\partial\bar{\partial}\omega^{n-1}=0$, Demailly (\cite{D}) proved
\begin{equation}\label{z:1}
\partial_{A}^{*}=\sqrt{-1}[\Lambda_{\omega},\bar{\partial}]-\tau^{*},\ \ \bar{\partial}_{A}^{*}=-\sqrt{-1}[\Lambda_{\omega},\partial_{A}]-\bar{\tau}^{*},
\end{equation}
where $\tau=[\Lambda_{\omega},\partial\omega]$. Letting $x\odot F_{A^{'}}=x^{i}F_{A^{'},ij}dx^{j}$, we derive
\begin{equation}\label{eqn:35}
\begin{split}
&Re\langle x\odot F_{A^{'}},D_{A}^{*}F_{A^{'}}\rangle \\
=&-Re\langle x\odot F_{A^{'}},\frac{\partial A}{\partial t}\rangle
+Re\langle x\odot F_{A^{'}},\sqrt{-1}\Lambda_{\omega}((\bar{\partial}_{A}-\partial_{A})[\phi,\phi^{*}])\rangle \\
&+Re\langle x\odot F_{A^{'}},\sqrt{-1}\Lambda_{\omega}([F_{A},\phi]-[F_{A},\phi^{*}])\rangle \\
&-Re\langle x\odot F_{A^{'}},(\tau^{*}+\bar{\tau}^{*})F_{A^{'}}\rangle.
\end{split}
\end{equation}
It is not hard to find that
\begin{equation}\label{eqn:36}
\begin{split}
&\langle x^{i}[\phi_{i}+\phi^{*}_{i},F_{A^{'}}],F_{A^{'}}\rangle=0.
\end{split}
\end{equation}
Because of $(F_{A}+[\phi,\phi^{*}])^{*}=-(F_{A}+[\phi,\phi^{*}]),(\partial_{A}\phi+\bar{\partial}_{A}\phi^{*})^{*}=\partial_{A}\phi+\bar{\partial}_{A}\phi^{*}$, we know
\begin{equation}\label{eqn:37}
\begin{split}
&Re\langle[\phi+\phi^{*},x\odot F_{A^{'}}],F_{A^{'}}\rangle\\
=&Re\langle[\phi+\phi^{*},x\odot(F_{A}+[\phi,\phi^{*}])],\partial_{A}\phi+\bar{\partial}_{A}\phi^{*}\rangle\\
&+Re\langle[\phi+\phi^{*},x\odot(\partial_{A}\phi+\bar{\partial}_{A}\phi^{*})],F_{A}+[\phi,\phi^{*}]\rangle,
\end{split}
\end{equation}
where
\begin{equation}\label{eqn:38(1)}
\begin{split}
\langle[\phi,x\odot (F_{A}+[\phi,\phi^{*}])],\partial_{A}\phi\rangle=&\langle x\odot (F_{A}+[\phi,\phi^{*}]),\sqrt{-1}\Lambda_{\omega}\partial_{A}[\phi,\phi^{*}]\rangle,\\
\langle[\phi^{*},x\odot (F_{A}+[\phi,\phi^{*}])],\bar{\partial}_{A}\phi^{*}\rangle=&-\langle x\odot (F_{A}+[\phi,\phi^{*}]),\sqrt{-1}\Lambda_{\omega}\bar{\partial}_{A}[\phi,\phi^{*}]\rangle\\
\end{split}
\end{equation}
and
\begin{equation}\label{eqn:38(2)}
\begin{split}
&\langle[\phi^{*},x\odot (\partial_{A}\phi+\bar{\partial}_{A}\phi^{*})],F_{A}+[\phi,\phi^{*}]\rangle\\
=&-\langle x\odot (\partial_{A}\phi+\bar{\partial}_{A}\phi^{*}),\sqrt{-1}\Lambda_{\omega}[F_{A}+[\phi,\phi^{*}],\phi]\rangle\\
&+\langle x\odot (\partial_{A}\phi+\bar{\partial}_{A}\phi^{*}),[\sqrt{-1}\Lambda_{\omega}(F_{A}+[\phi,\phi^{*}]),\phi]\rangle,\\
&\langle[\phi,x\odot (\partial_{A}\phi+\bar{\partial}_{A}\phi^{*})],F_{A}+[\phi,\phi^{*}]\rangle\\
=&\langle x\odot (\partial_{A}\phi+\bar{\partial}_{A}\phi^{*}),\sqrt{-1}\Lambda_{\omega}[F_{A}+[\phi,\phi^{*}],\phi^{*}]\rangle\\
&-\langle x\odot (\partial_{A}\phi+\bar{\partial}_{A}\phi^{*}),[\sqrt{-1}\Lambda_{\omega}(F_{A}+[\phi,\phi^{*}]),\phi^{*}]\rangle.
\end{split}
\end{equation}
So
\begin{equation}\label{eqn:39}
\begin{split}
&Re\langle[\phi+\phi^{*},x\odot F_{A^{'}}],F_{A^{'}}\rangle\\
=&-Re\langle x\odot F_{A^{'}},\sqrt{-1}\Lambda_{\omega}((\bar{\partial}_{A}-\partial_{A})[\phi,\phi^{*}])\rangle \\
&-Re\langle x\odot F_{A^{'}},\sqrt{-1}\Lambda_{\omega}([F_{A},\phi]-[F_{A},\phi^{*}])\rangle \\
&-Re\langle x\odot F_{A^{'}},\frac{\partial\phi}{\partial t}+\frac{\partial\phi^{*}}{\partial t}\rangle.
\end{split}
\end{equation}
Combining (\ref{eqn:32}), (\ref{eqn:33}), (\ref{eqn:34}), (\ref{eqn:35}), (\ref{eqn:36}) and (\ref{eqn:39}), and using $\phi\wedge\phi=0$, we deduce
\begin{equation}\label{310}
\begin{split}
I_{1}+I_{2}=&-4rRe\int_{T_{r}(u_{0})}\langle d(f^{2}G_{u_{0}})\wedge x\odot F_{A^{'}},F_{A^{'}}\rangle dv_{g}dt\\
&-4rRe\int_{T_{r}(u_{0})}\langle x\odot F_{A^{'}},\frac{\partial A^{'}}{\partial t}\rangle f^{2}G_{u_{0}}dv_{g}dt \\
&+4rRe\int_{T_{r}(u_{0})}\langle x^{i}F_{A^{'}}(\nabla_{k}\partial_{i},\partial_{j})dx^{j}\wedge dx^{k},F_{A^{'}}\rangle f^{2}G_{u_{0}}dv_{g}dt\\
&-4rRe\int_{T_{r}(u_{0})}\langle x\odot F_{A^{'}},(\tau^{*}+\bar{\tau}^{*})F_{A^{'}}\rangle f^{2}G_{u_{0}}dv_{g}dt.
\end{split}
\end{equation}
Noting that
\begin{equation}
\begin{split}
\langle F_{A^{'}},\Big[\phi,\frac{\partial A^{'}}{\partial t}\Big]\rangle=&\langle \partial_{A}\phi,\Big[\phi,\frac{\partial A^{'}}{\partial t}\Big]\rangle+\langle F_{A}+[\phi,\phi^{*}],\Big[\phi,\frac{\partial A^{'}}{\partial t}\Big]\rangle\\
=&\langle\sqrt{-1}\Lambda_{\omega}[\partial_{A}\phi,\phi^{*}],\frac{\partial A^{'}}{\partial t}\rangle+\langle\sqrt{-1}\Lambda_{\omega}[F_{A}+[\phi,\phi^{*}],\phi^{*}],\frac{\partial A^{'}}{\partial t}\rangle\\
&-\langle[\sqrt{-1}\Lambda_{\omega}(F_{A}+[\phi,\phi^{*}]),\phi^{*}],\frac{\partial A^{'}}{\partial t}\rangle,
\end{split}
\end{equation}
\begin{equation}
\begin{split}
\langle F_{A^{'}},\Big[\phi^{*},\frac{\partial A^{'}}{\partial t}\Big]\rangle=&-\langle\sqrt{-1}\Lambda_{\omega}[\bar{\partial}_{A}\phi^{*},\phi],\frac{\partial A^{'}}{\partial t}\rangle-\langle\sqrt{-1}\Lambda_{\omega}[F_{A}+[\phi,\phi^{*}],\phi],\frac{\partial A^{'}}{\partial t}\rangle\\
&+\langle[\sqrt{-1}\Lambda_{\omega}(F_{A}+[\phi,\phi^{*}]),\phi],\frac{\partial A^{'}}{\partial t}\rangle,
\end{split}
\end{equation}
and applying $\phi\wedge\phi=0$ again, one can get
\begin{equation}
\begin{split}
&Re\langle D_{A}^{*}F_{A^{'}},\frac{\partial A^{'}}{\partial t}\rangle+Re\langle F_{A^{'}},\Big[\phi+\phi^{*},\frac{\partial A^{'}}{\partial t}\Big]\rangle\\
=&-\Big|\frac{\partial A}{\partial t}\Big|^{2}-2\Big|\frac{\partial \phi}{\partial t}\Big|^{2}-Re\langle(\tau^{*}+\bar{\tau}^{*})F_{A^{'}},\frac{\partial A^{'}}{\partial t}\rangle.
\end{split}
\end{equation}
The flow and Stokes formula imply that
\begin{equation}\label{eqn:311}
\begin{split}
I_{3}=&2r\int_{T_{r}(u_{0})}(t-t_{0})\partial_{t}e(A,\phi)f^{2}G_{u_{0}}dv_{g}dt\\
=&-4rRe\int_{T_{r}(u_{0})}(t-t_{0})\Big(\Big|\frac{\partial A}{\partial t}\Big|^{2}+2\Big|\frac{\partial \phi}{\partial t}\Big|^{2}\Big)dv_{g}dt\\
&-4rRe\int_{T_{r}(u_{0})}(t-t_{0})\langle F_{A^{'}},d(f^{2}G_{u_{0}})\wedge\frac{\partial A^{'}}{\partial t}\rangle dv_{g}dt\\
&-4rRe\int_{T_{r}(u_{0})}(t-t_{0})\langle(\tau^{*}+\bar{\tau}^{*})F_{A^{'}},\frac{\partial A^{'}}{\partial t}\rangle f^{2}G_{u_{0}}dv_{g}dt.
\end{split}
\end{equation}
Let $x\cdot F_{A^{'}}=\frac{1}{2}x^{i}g^{ij}F_{A^{'},jk}dx^{k}$, and notice that $\partial_{i}G_{u_{0}}=\frac{x^{i}G_{u_{0}}}{2(t-t_{0})}$, then
\begin{equation}\label{eqn:312}
\begin{split}
&I_{1}+I_{2}+I_{3}\\
=&4rRe\int_{T_{r}(u_{0})}\frac{1}{|t-t_{0}|}\Big||t-t_{0}|\frac{\partial A^{'}}{\partial t}-x\odot F_{A^{'}}\Big|^{2}f^{2}G_{u_{0}}dv_{g}dt\\
&+4rRe\int_{T_{r}(u_{0})}\frac{1}{|t-t_{0}|}\langle x\cdot F_{A^{'}}-x\odot F_{A^{'}},x\odot F_{A^{'}}-|t-t_{0}|\frac{\partial A^{'}}{\partial t}\rangle f^{2}G_{u_{0}}dv_{g}dt \\
&+4rRe\int_{T_{r}(u_{0})}\langle 2f^{-1}\nabla f\llcorner F_{A^{'}},|t-t_{0}|\frac{\partial A}{\partial t}-x\odot F_{A^{'}}\rangle f^{2}G_{u_{0}}dv_{g}dt\\
&+4rRe\int_{T_{r}(u_{0})}\langle x^{i}F_{A^{'}}(\nabla_{k}\partial_{i},\partial_{j})dx^{j}\wedge dx^{k},F_{A^{'}}\rangle f^{2}G_{u_{0}}dv_{g}dt\\
&+4rRe\int_{T_{r}(u_{0})}\langle |t-t_{0}|\frac{\partial A^{'}}{\partial t}-x\odot F_{A^{'}},(\tau^{*}+\bar{\tau}^{*})F_{A^{'}}\rangle f^{2}G_{u_{0}}dv_{g}dt.
\end{split}
\end{equation}
Calculating straightforward, we obtain
\begin{equation}\label{eqn:313}
\begin{split}
I_{4}=&r\int_{T_{r}(u_{0})}e(A,\phi)x^{i}\partial_{i}(f^{2}\sqrt{\det(g_{ij})})G_{u_{0}}dxdt\\
=&r\int_{T_{r}(u_{0})}e(A,\phi)2fx^{i}\partial_{i}fG_{u_{0}}dv_{g}dt\\
&+\frac{r}{2}\int_{T_{r}(u_{0})}e(A,\phi)x^{i}\tr(g^{-1}\partial_{i}g)f^{2}G_{u_{0}}dv_{g}dt.
\end{split}
\end{equation}
Based on (\ref{a}) and $|\Gamma_{jk}^{i}|\leq C|x|$, there exist constants $\tilde{C}_{1}$ and $\tilde{C}_{2}$, such that
\begin{equation}\label{eqn:314}
\begin{split}
&|x\cdot F_{A^{'}}-x\odot F_{A^{'}}|^{2}=\Big|\frac{1}{2}x^{i}(\delta_{j}^{i}-g^{ij})F_{A^{'},jl}dx^{l}\Big|^{2}\leq \tilde{C}_{1}|x|^{6}|F_{A^{'}}|^{2},\\
&\langle x^{i}F_{A^{'}}(\nabla_{j}\partial_{i},\partial_{k})dx^{j}\wedge dx^{k},F_{A^{'}}\rangle\leq \tilde{C}_{1}|x|^{2}|F_{A^{'}}|^{2},\\
&\tr(g^{-1}\partial_{k}g)\leq \tilde{C}_{1}|x|,
\end{split}
\end{equation}
and $|(\tau^{*}+\bar{\tau}^{*})F_{A^{'}}|^{2}\leq \tilde{C}_{2}|F_{A^{'}}|^{2}$. Due to (\ref{eqn:312}) and (\ref{eqn:313}), it holds
\begin{equation}
\begin{split}
\frac{d\Phi(r)}{dr}\geq& r\int_{T_{r}(u_{0})}\frac{1}{|t-t_{0}|}\Big||t-t_{0}|\frac{\partial A^{'}}{\partial t}-x\odot F_{A^{'}}\Big|^{2}f^{2}G_{u_{0}}dv_{g}dt\\
&-\tilde{C}_{3}r\int_{T_{r}(u_{0})}\frac{|x|^{6}}{|t-t_{0}|}|F_{A^{'}}|^{2}f^{2}G_{u_{0}}dv_{g}dt\\
&-r\int_{T_{r}(u_{0})}|t-t_{0}|\cdot |2f^{-1}\nabla\llcorner F_{A^{'}}|^{2}f^{2}G_{u_{0}}dv_{g}dt\\
&-\tilde{C}_{3}r\int_{T_{r}(u_{0})}|t-t_{0}|\cdot |F_{A^{'}}|^{2}f^{2}G_{u_{0}}dv_{g}dt\\
&-\tilde{C}_{3}r\int_{T_{r}(u_{0})}|x|^{2}|F_{A^{'}}|^{2}f^{2}G_{u_{0}}dv_{g}dt\\
&-2r\int_{T_{r}(u_{0})}|x|\cdot |\nabla f|\cdot |f|\cdot |F_{A^{'}}|^{2}G_{u_{0}}dv_{g}dt.
\end{split}
\end{equation}
According to Chen-Struwe's arguments in \cite{CS}, we know there exists a constant $\tilde{C}_{4}>0$ such that
\begin{equation}
\begin{split}
&r^{-1}|t-t_{0}|\cdot |x|^{6}G_{u_{0}}\leq \tilde{C}_{4}(1+G_{u_{0}}),\\
&r^{-1}|x|^{2}G_{u_{0}}\leq \tilde{C}_{4}(1+G_{u_{0}})
\end{split}
\end{equation}
on $T_{r}(u_{0})$. Hence it follows that
\begin{equation}
\begin{split}
&-\tilde{C}_{3}r\int_{T_{r}(u_{0})}\Big(\frac{|x|^{6}}{|t-t_{0}|}+|t-t_{0}|+|x|^{2}\Big)|F_{A^{'}}|^{2}f^{2}G_{u_{0}}dv_{g}dt\\
&\geq -\tilde{C}_{5}\Phi(r)-\tilde{C}_{5}r\textrm{YMH}(A_{0},\phi_{0}),
\end{split}
\end{equation}
where the constant $\tilde{C}_{5}$ depends only on $\tilde{C}_{3}$ and $\tilde{C}_{4}$. As that shown in \cite{NZ2}(page 15), we have
\begin{equation}
\begin{split}
-r\int_{T_{r}(u_{0})}|t-t_{0}|\cdot |2f^{-1}\nabla f\llcorner F_{A^{'}}|^{2}f^{2}G_{u_{0}}dv_{g}dt&\geq-\frac{C(n)r}{R^{2n}}\int_{P_{R}(u_{0})}|F_{A^{'}}|^{2}dv_{g}dt,\\
-2r\int_{T_{r}(u_{0})}|x|\cdot |\nabla f|\cdot |f|\cdot |F_{A^{'}}|^{2}G_{u_{0}}dv_{g}dt&\geq-\frac{C(n)r}{R^{2n}}\int_{P_{R}(u_{0})}|F_{A^{'}}|^{2}dv_{g}dt.
\end{split}
\end{equation}
From above, it can be seen that
\begin{equation}
\frac{d\Phi(r)}{dr}\geq -\tilde{C}_{6}\Phi(r)-\tilde{C}_{6}r\rm{YMH}(A_{0},\phi_{0})-\frac{\tilde{C}_{6}r}{R^{2n}}\int_{P_{R}(u_{0})}|F_{A^{'}}|^{2}dv_{g}dt.
\end{equation}
Then integrating two sides of this inequality concludes this proof.
\end{proof}

By the monotonicity inequality, we can derive the $\epsilon$-regularity theorem.
\begin{theorem}\label{thm:5}
Let $(A(t),\phi(t))$ be a solution of the modified Yang-Mills-Higgs flow. There exist positive constants $\epsilon_{0},\delta_{0}<1/4$, such that if
\begin{equation}
R^{2-2n}\int_{P_{R}(u_{0})}e(A,\phi)dv_{g}dt<\epsilon_{0}
\end{equation}
holds for some $0<R\leq \min\{i_{X}/2,\sqrt{t_{0}}/2\}$, then for any $\delta\in(0,\delta_{0})$, we have
\begin{equation}
\sup_{P_{\delta R}(u_{0})}e(A,\phi)\leq\frac{16}{(\delta R)^{4}}
\end{equation}
and
\begin{equation}
\sup_{P_{\delta R}(u_{0})}|\hat{\nabla}_{A}\phi|^{2}\leq  C_{11},
\end{equation}
where $C_{11}$ is a positive constant depending only on $R$, $\delta_{0}$, $\sup_{X}|\phi_{0}|_{H_{0}}$ and the geometry of $(X,\omega)$.
\end{theorem}
\begin{proof}
Regard $X$ as a real manifold. The proof is very similar to the one of Theorem 2.6 in \cite{LZ1}.

For any $\delta\in(0,1/4]$, define
\begin{equation}
\begin{split}
f(r)=(2\delta R-r)^{4}\sup_{P_{r}(x_{0},t_{0})}e(A,\phi).
\end{split}
\end{equation}
Since $f(r)$ is continuous and $f(2\delta R)=0$, we know that $f(r)$ attains its maximum at some point $r_{0}\in[0,2\delta R)$. Suppose $(x_{1},t_{1})\in\bar{P}_{r_{0}}(x_{0},t_{0})$ is a point such that
\begin{equation}
e(A,\phi)(x_{1},t_{1})=\sup_{P_{r}(x_{0},t_{0})}e(A,\phi).
\end{equation}
We claim that when $\epsilon_{0},\delta_{0}$ are small enough, $f(r_{0})\leq 16$. Otherwise, set
\begin{equation}
\rho_{0}=(2\delta R-r_{0})f(r_{0})^{-1/4}=e(A,\phi)(x_{1},t_{1})^{-1/4}<\delta R-\frac{r_{0}}{2}.
\end{equation}
Rescaling the Riemannian metric $\tilde{g}=\rho_{0}^{-2}g$ and $t=t_{1}+\rho_{0}^{2}\tilde{t}$, we get
\begin{equation}
\begin{split}
&|F_{A}+[\phi,\phi^{*H_{0}}]|_{\tilde{g}}^{2}=\rho_{0}^{4}|F_{A}+[\phi,\phi^{*H_{0}}]|_{g}^{2},\\
&|\partial_{A}\phi|_{\tilde{g}}^{2}=\rho_{0}^{4}|\partial_{A}\phi|_{g}^{2},\ \ |\hat{\nabla}_{A}\phi|_{\tilde{g}}^{2}=\rho_{0}^{4}|\hat{\nabla}_{A}\phi|_{g}^{2}.
\end{split}
\end{equation}
Setting
\begin{equation}
\begin{split}
&e_{\rho_{0}}(x,\tilde{t})=|F_{A}+[\phi,\phi^{*H_{0}}]|_{\tilde{g}}^{2}+2|\partial_{A}\phi|_{\tilde{g}}^{2}=\rho_{0}^{4}e(A,\phi)(x,t_{1}+\rho_{0}^{2}\tilde{t}),\\
&b_{\rho_{0}}(x,\tilde{t})=|\hat{\nabla}_{A}\phi|_{\tilde{g}}^{2}=\rho_{0}^{4}|\hat{\nabla}_{A}\phi|_{g}^{2}(x,t_{1}+\rho_{0}^{2}\tilde{t}),\\
&\tilde{P}_{\tilde{r}}(x_{1},0)=B_{\rho_{0}\tilde{r}}(x_{1})\times[-\tilde{r}^{2},\tilde{r}^{2}],
\end{split}
\end{equation}
we have $e_{\rho_{0}}(x_{1},0)=\rho_{0}^{4}e(A,\phi)(x_{1},t_{1})=1$, and
\begin{equation}
\begin{split}
\sup_{\tilde{P}_{1}(x_{1},0)}e_{\rho_{0}}&=\rho_{0}^{4}\sup_{P_{\rho_{0}}(x_{1},t_{1})}e(A,\phi)\leq \rho_{0}^{4}\sup_{P_{\delta R+r_{0}/2}(x_{0},t_{0})}e(A,\phi)\\
&\leq \rho_{0}^{4}f(\delta R+r_{0}/2)(\delta R-r_{0}/2)^{-4}\leq 16.
\end{split}
\end{equation}
Thus
\begin{equation}\label{eqn:315}
\begin{split}
|F_{A}+[\phi,\phi^{*H_{0}}]|_{\tilde{g}}^{2}+2|\partial_{A}\phi|_{\tilde{g}}^{2}\leq 16,\ \ \ \ on\ \ \tilde{P}_{1}(x_{1},0).
\end{split}
\end{equation}
Combining this together with Proposition \ref{p:1} yields that
\begin{equation}
\begin{split}
&(\Delta_{\tilde{g}}-\frac{\partial}{\partial \tilde{t}})(b_{\rho_{0}}+\rho_{0}^{4})=\rho_{0}^{6}(\Delta_{g}-\frac{\partial}{\partial t})(|\hat{\nabla}_{A}\phi|_{g}^{2}+1)\\
\geq &2\rho_{0}^{6}|\hat{\nabla}_{A}\hat{\nabla}_{A}\phi|_{g}^{2}-C_{4}\rho_{0}^{6}(|F_{A}|_{g}+|R_{\omega}|_{g}+|\phi|_{g}^{2})|\hat{\nabla}_{A}\phi|_{g}^{2}\\
&-C_{5}\rho_{0}^{6}(|\hat{\nabla} R_{\omega }|_{g}+|T|_{g}|F_{A}|_{g})|\phi|_{g}|\hat{\nabla}_{A}\phi|_{g}\\
&-C_{6}\rho_{0}^{6}|T|_{g}|\hat{\nabla}_{A}\phi|_{g}|\hat{\nabla}_{A}\hat{\nabla}_{A}\phi|_{g}\\
\geq&-\tilde{C}_{1}(b_{\rho_{0}}+\rho_{0}^{4})
\end{split}
\end{equation}
on $\tilde{P}_{1}(x_{1},0)$, where the constant $\tilde{C}_{1}$ depends only on the geometry of $(X,\omega)$ and $\sup_{X}|\phi_{0}|_{H_{0}}$. Then by the parabolic mean value inequality and Proposition \ref{p:2}, we observe
\begin{equation}\label{eqn:316}
\begin{split}
\sup_{\tilde{P}_{1/2}(x_{1},0)}(b_{\rho_{0}}+\rho_{0}^{4})&\leq C\int_{\tilde{P}_{1}(x_{1},0)}(b_{\rho_{0}}+\rho_{0}^{4})dv_{\tilde{g}}d\tilde{t}\\
&=C\rho_{0}^{2-2n}\int_{P_{\rho_{0}}(x_{1},t_{1})}(|\hat{\nabla}_{A}\phi|_{g}^{2}+1)dv_{g}dt\\
&\leq C^{*}\rho_{0}^{2}\leq \tilde{C}_{2},
\end{split}
\end{equation}
where the constant $\tilde{C}_{2}$ depends only on the geometry of $(X,\omega)$ and $\sup_{X}|\phi_{0}|_{H_{0}}$. Similarly, combining (\ref{eqn:316}), (\ref{eqn:315}) and Proposition \ref{p:1}, one knows that
\begin{equation}
(\Delta_{\tilde{g}}-\frac{\partial}{\partial \tilde{t}})e_{\rho_{0}}\geq -\tilde{C}_{3}e_{\rho_{0}}
\end{equation}
on $\tilde{P}_{1/2}(x_{1},0)$, where the constant $\tilde{C}_{3}$ depends only on the geometry of $(X,\omega)$ and $\sup_{X}|\phi_{0}|_{H_{0}}$. Using the parabolic mean value inequality again, we can see that
\begin{equation}
\begin{split}
1=e_{\rho_{0}}(x_{1},0)\leq\sup_{\tilde{P}_{1/4}(x_{1},0)}e_{\rho_{0}}&\leq C\int_{\tilde{P}_{1/2}(x_{1},0)}e_{\rho_{0}}dv_{\tilde{g}}d\tilde{t}\\
&\leq \tilde{C}_{4}\rho_{0}^{2-2n}\int_{P_{\rho_{0}}(x_{1},t_{1})}e(A,\phi)dv_{g}dt,\\
\end{split}
\end{equation}
where the constant $\tilde{C}_{4}$ depends only on the geometry of $(X,\omega)$ and $\sup_{X}|\phi_{0}|_{H_{0}}$.

Choose normal geodesic coordinates centred at $x_{1}$, and a smooth cut-off function $f\in C^{\infty}_{0}(B_{R/2}(x_{1}))$ such that $0\leq f\leq 1$, $f\equiv 1$ on $B_{R/4}(x_{1})$, $|d f|\leq 8/R$ on $B_{R/2}(x_{1})\setminus B_{R/4}(x_{1})$. Taking $r_{1}=\rho_{0}$ and $r_{2}=\delta_{0} R$, applying the monotonicity inequality, we obtain
\begin{equation}
\begin{split}
&\rho_{0}^{2-2n}\int_{P_{\rho_{0}}(x_{1},t_{1})}e(A,\phi)dv_{g}dt\\
\leq& C\rho_{0}^{2}\int_{P_{\rho_{0}}(x_{1},t_{1})}e(A,\phi)G_{(x_{1},t_{1}+2\rho_{0}^{2})}f^{2}dv_{g}dt\\
\leq& C\rho_{0}^{2}\int_{T_{\rho_{0}}(x_{1},t_{1}+2\rho_{0}^{2})}e(A,\phi)G_{(x_{1},t_{1}+2\rho_{0}^{2})}f^{2}dv_{g}dt\\
\leq& C_{*}r_{2}^{2}\int_{T_{r_{2}}(x_{1},t_{1}+2\rho_{0}^{2})}e(A,\phi)G_{(x_{1},t_{1}+2\rho_{0}^{2})}f^{2}dv_{g}dt+C_{*}\delta_{0}^{2}R^{2}\rm{YMH}(A_{0},\phi_{0})\\
&+C_{*}(R/2)^{2-2n}\int_{P_{R/2}(x_{1},t_{1})}e(A,\phi)dv_{g}dt\\
\leq& C_{*}\delta_{0}^{2-2n}R^{2-2n}\int_{P_{R}(x_{0},t_{0})}e(A,\phi)dv_{g}dt+C_{*}\delta_{0}^{2}R^{2}\rm{YMH}(A_{0},\phi_{0})\\
\leq& \tilde{C}_{5}(\delta_{0}^{2-2n}\epsilon_{0}+\delta_{0}^{2}R^{2}\rm{YMH}(A_{0},\phi_{0})),
\end{split}
\end{equation}
where the constant $\tilde{C}_{5}$ depends only on the geometry of $(X,\omega)$ and $\sup_{X}|\phi_{0}|_{H_{0}}$. Choose $\epsilon_{0},\delta_{0}$ properly such that $\tilde{C}_{4}\tilde{C}_{5}(\delta_{0}^{2-2n}\epsilon_{0}+\delta_{0}^{2}R^{2}\rm{YMH}(A_{0},\phi_{0}))< 1$, and then a contradiction occurs. So $f(r_{0})\leq 16$, which implies
\begin{equation}
\sup_{P_{\delta R}(u_{0})}e(A,\phi)\leq 16/(\delta R)^{4}.
\end{equation}
Let $r=3\delta R/2$, then
\begin{equation}
\sup_{P_{3\delta R/2}(u_{0})}e(A,\phi)=f(3\delta R/2)(\delta R/2)^{-1/4}\leq 256/(\delta R)^{4}.
\end{equation}
On $P_{3\delta R/2}(u_{0})$, from Proposition \ref{p:1}, Proposition \ref{p:2} and  the parabolic mean value inequality, we derive
\begin{equation}
\begin{split}
\sup_{P_{\delta R}(u_{0})}(|\hat{\nabla}_{A}\phi|_{g}^{2}+1)&\leq C\int_{P_{3\delta R/2}(u_{0})}(|\hat{\nabla}_{A}\phi|_{g}^{2}+1)dv_{g}dt\\
&\leq \tilde{C}_{6},
\end{split}
\end{equation}
where the constant $\tilde{C}_{6}$ depends only on $R$, $\delta_{0}$, $\sup_{X}|\phi_{0}|_{H_{0}}$ and the geometry of $(X,\omega)$.
\end{proof}

\section{Proof of Theorem \ref{thm:1}}
\begin{theorem}\label{app01}
Let $(X,\omega)$ be a compact Hermitian manifold of dimension $n$ satisfying $\partial\bar{\partial}\omega^{n-1}=\partial\bar{\partial}\omega^{n-2}=0$,  and $(E,\bar{\partial}_{E},\theta)$ be a semi-stable Higgs bundle of rank $r$ with the vanishing  condition (\ref{CC}). Then $E$ admits a Hermitian metric $K$ and a family of Hermitian metrics ${H_{s}}$ for $s\in (0, 1]$ such that $\sqrt{-1}\Lambda_{\omega }\tr F_{K}$ is a constant and
 \begin{equation}\label{k040}
\sup_{X}(|F_{H_{s}}+[\theta ,\theta^{*H_{s}}]-\frac{1}{r}\tr F_{K}\otimes \Id_{E}|_{H_{s}}^{2}+2|\partial_{H_{s}}\theta |_{H_{s}}^{2})\rightarrow 0
\end{equation}
as $s \rightarrow 0$.
\end{theorem}
\begin{proof}

By conformal transformation, we can choose a background metric $K$ satisfying
\begin{equation}
\tr \{\sqrt{-1}\Lambda_{\omega } (F_{K}+[\theta,\theta^{*K}])-\lambda \cdot \textmd{Id}_E\}=0,
\end{equation}
where the constant $\lambda =(r\text{Vol}(X, \omega))^{-1}ch_{1}(E,\bar{\partial}_{E})[\omega^{n-1}]$. Let $H_{\epsilon}$ be the solution of the equation (\ref{eq}) for $\epsilon \in (0, 1]$. From (\ref{det1}), we see that $\det (K^{-1}H_{\epsilon })=1$ and then
\begin{equation}\label{det2}
\tr F_{H_{\epsilon}}=\tr F_{K}
\end{equation}
for every $\epsilon \in (0, 1]$.
Since the Higgs bundle $(E,\bar{\partial}_{E},\theta)$ is semi-stable, by Theorem \ref{thm:3}, we know that
\begin{equation}
\sup_{X}|\sqrt{-1}\Lambda_{\omega}F_{H_{\epsilon},\theta}-\lambda \Id_{E}|_{H_{\epsilon}}\rightarrow 0
\end{equation}
as $\epsilon \rightarrow 0$. According to the identity (\ref{Bo1}) and (\ref{det2}), we have
\begin{equation}\label{Bo2}
\begin{split}
&\int_{X}|F_{H_{\epsilon},\theta}-\frac{1}{r}\tr F_{K}\otimes \Id_{E}|^{2}\frac{\omega^{n}}{n!}\\
= &\int_{X}(2|\partial_{H_{\epsilon }}\theta |^{2}+|(F_{H_{\epsilon },\theta}^{1,1})^{\perp}|^{2})\frac{\omega^{n}}{n!}\\
= &\int_{X}|\sqrt{-1}\Lambda_{\omega}(F_{H_{\epsilon},\theta})-\lambda \Id_{E}|^{2}\frac{\omega^{n}}{n!}.\\
\end{split}
\end{equation}
Consider the Hermitian-Yang-Mills flow with initial metric $H_{\epsilon}$
\begin{equation}\label{hymf2}
\left\{
\begin{split}
      &H_{\epsilon}^{-1}(t)\frac{\partial H_{\epsilon}(t)}{\partial t}=-2(\sqrt{-1}\Lambda_{\omega}(F_{H_{\epsilon}(t)}+[\theta,\theta^{*H_{\epsilon}(t)}])-\lambda \Id),  \\
      &H_{\epsilon}(0)=H_{\epsilon}.
\end{split}
\right.
\end{equation}
It is gauge equivalent to the following Yang-Mills-Higgs flow
\begin{equation}\label{ymhf2}
\left\{
\begin{split}
             \frac{\partial A_{\epsilon}(t)}{\partial t}&=-\sqrt{-1}(\partial_{A_{\epsilon}(t)}-\bar{\partial}_{A_{\epsilon}(t)})\Lambda_{\omega}(F_{A_{\epsilon}(t)}+[\phi_{\epsilon}(t),\phi_{\epsilon}^{*H_{\epsilon}}(t)]),   \\
             \frac{\partial\phi_{\epsilon}(t)}{\partial t}&=-[\sqrt{-1}\Lambda_{\omega}(F_{A_{\epsilon}(t)}+[\phi_{\epsilon}(t),\phi_{\epsilon}^{*H_{\epsilon}}(t)]),\phi_{\epsilon}(t)],  \\
             A_{\epsilon}(0)&=D_{H_{\epsilon},\bar{\partial}_{E}},\ \ \phi_{\epsilon}(0)=\theta.
\end{split}
\right.
\end{equation}
One can easily  see that
\begin{equation}
\det (K^{-1}H_{\epsilon }(t))=\det (H_{\epsilon }^{-1}H_{\epsilon }(t))=1\end{equation} and then
\begin{equation}
\tr F_{H_{\epsilon}(t)}=\tr F_{K}
\end{equation}
for any $\epsilon \in (0, 1]$ and $t\geq 0$. According to Proposition \ref{p}, we observe
\begin{equation}
\begin{split}
\int_{X}|F_{A_{\epsilon}(t),\phi_{\epsilon}(t)}-\frac{1}{r}\tr F_{K}\otimes \Id_{E}|_{H_{\epsilon}}^{2}\frac{\omega^{n}}{n!}&=\int_{X}|F_{H_{\epsilon}(t),\theta}-\frac{1}{r}\tr F_{K}\otimes \Id_{E}|_{H_{\epsilon}(t)}^{2}\frac{\omega^{n}}{n!}\\
&=\int_{X}|\sqrt{-1}\Lambda_{\omega}F_{H_{\epsilon}(t),\theta}-\lambda \Id_{E}|_{H_{\epsilon}(t)}^{2}\frac{\omega^{n}}{n!}\\ &\leq \int_{X}|\sqrt{-1}\Lambda_{\omega}F_{H_{\epsilon},\theta}-\lambda \Id_{E}|_{H_{\epsilon}}^{2}\frac{\omega^{n}}{n!}.
\end{split}
\end{equation}

Even though we don't have the uniform $C^{0}$-estimate of $H_{\epsilon}$, by (\ref{t4}) and $(\ref{t2})$, we still have \begin{equation}\sup_{X}|\theta|_{H_{\epsilon},\omega}^{2}\leq \check{C}_{4}\end{equation}
and
\begin{equation}
\int_{X}|F_{H_{\epsilon},\theta}|_{H_{\epsilon},\omega}^{2}\frac{\omega^{n}}{n!}\leq \check{C}_{5},
\end{equation}
where $\check{C}_{4}$ and $\check{C}_{5}$  are positive constants depending only on $\sup_{X}|\theta |_{K, \omega}$, $\sup_{X}|\Lambda_{\omega }F_{K, \theta}|_{K}$ and the geometry of $(X, \omega )$. So, when we apply the small energy regularity for the Hermitian-Yang-Mills flow (\ref{hymf2}) (or equivalently the Yang-Mills-Higgs flow (\ref{ymhf2})) with different initial datas $H_{\epsilon}$, all  the constants in Theorem \ref{thm:5} are independent of $\epsilon$, in fact they depend only on $\sup_{X}|\theta |_{K, \omega}$, $\sup_{X}|\Lambda_{\omega }F_{K, \theta}|_{K}$ and the geometry of $(X,\omega)$.

We can fix a pair of positive constants $\epsilon_{0},\delta_{0}$ in Theorem \ref{thm:5} satisfying $\tilde{C}_{4}\tilde{C}_{5}(\delta_{0}^{2-2n}\epsilon_{0}+\check{C}_{5}\delta_{0}^{2}(i_{X}/2)^{2})<1$. Taking $R=\min \{\frac{1}{2}, \frac{i_{X}}{2}, \frac{1}{2}(r\epsilon_{0})^{\frac{1}{4}}(\sup_{X}|\tr F_{K}|^{2}\text{Vol}(X, \omega))^{-\frac{1}{4}} \}$ and $0< \epsilon \leq \epsilon '$ for some $\epsilon '$ small enough, we have
\begin{equation}
\begin{split}
&R^{2-2n}\int_{P_{R}(u_{0})}|F_{A_{\epsilon}(t),\phi_{\epsilon}(t)}|_{H_{\epsilon}}^{2}\frac{\omega^{n}}{n!}\\
=&R^{2-2n}\int_{P_{R}(u_{0})}(|F_{A_{\epsilon}(t),\phi_{\epsilon}(t)}-\frac{1}{r}\tr F_{K}\otimes \Id_{E}|_{H_{\epsilon}}^{2}+r^{-1}|\tr F_{K}|^{2})\frac{\omega^{n}}{n!}\\
\leq &\epsilon_{0},
\end{split}
\end{equation}
for any $u_{0}\in X\times[1,\infty)$.
By Theorem \ref{thm:5}, for any $\delta\in(0,\delta_{0})$, it holds that $\sup_{P_{\delta R}(u_{0})}e(A_{\epsilon}(t),\phi_{\epsilon}(t))\leq\frac{16}{(\delta R)^{4}}$ and $\sup_{P_{\delta R}(u_{0})}|\hat{\nabla}_{A_{\epsilon}(t)}\phi_{\epsilon}(t)|_{H_{\epsilon}}^{2}\leq C_{11}.$
Therefore  for any $0<\epsilon \leq \epsilon '$, we deduce
\begin{equation}\label{d:1}
\sup_{X\times[1,\infty)}e(A_{\epsilon}(t),\phi_{\epsilon}(t))\leq \hat{C}_{1}
\end{equation}
and
\begin{equation}\label{d:2}
\sup_{X\times[1,\infty)}|\hat{\nabla}_{A_{\epsilon}(t)}\phi_{\epsilon}(t)|^{2}\leq \hat{C}_{2},
\end{equation}
where the constants $\hat{C}_{1},\hat{C}_{2}$ are independent of $\epsilon$. Combining (\ref{d:1}), (\ref{d:2}) and the inequality (\ref{k02}) in Proposition \ref{p:1}, we can find a uniform constant $\check{C}_{6}$ depending only on $\sup_{X}|\theta |_{K, \omega}$, $\sup_{X}|\Lambda_{\omega }F_{K, \theta}|_{K}$ and the geometry of $(X,\omega)$, such that
\begin{equation}\label{k03}
\begin{split}
&(\Delta -\frac{\partial}{\partial t})(|F_{A_{\epsilon}(t)}+[\phi_{\epsilon}(t),\phi_{\epsilon}(t)^{*H_{\epsilon}}]-\frac{1}{r}\tr F_{K}\otimes \Id_{E}|^{2}+2|\partial_{A_{\epsilon}(t)}\phi_{\epsilon}(t)|^{2})\\
\geq & -\check{C}_{6}(|F_{A_{\epsilon}(t)}+[\phi_{\epsilon}(t),\phi_{\epsilon}(t)^{*H_{\epsilon}}]-\frac{1}{r}\tr F_{K}\otimes \Id_{E}|^{2}+2|\partial_{A_{\epsilon}(t)}\phi_{\epsilon}(t)|^{2})
\end{split}
\end{equation}
for any $0<\epsilon \leq \epsilon '$ and $t\in [1, \infty )$.
Applying the mean value inequality of parabolic equation, one can see that there is a constant $\check{C}_{7}$ independent of $\epsilon$ and $t$, such that for $t\geq 2$ and $\epsilon $ small enough, there holds
\begin{equation}\label{k04}
\begin{split}
&\sup_{X}(|F_{H_{\epsilon}(t)}+[\theta ,\theta^{*H_{\epsilon}(t)}]-\frac{1}{r}\tr F_{K}\otimes \Id_{E}|_{H_{\epsilon}(t)}^{2}+2|\partial_{H_{\epsilon}(t)}\theta |_{H_{\epsilon}(t)}^{2})\\
= & \sup_{X}(|F_{A_{\epsilon}(t)}+[\phi_{\epsilon} (t),\phi_{\epsilon} (t)^{*H_{\epsilon}}]-\frac{1}{r}\tr F_{K}\otimes \Id_{E}|_{H_{\epsilon}}^{2}+2|\partial_{A_{\epsilon}(t)}\phi_{\epsilon} (t)|_{H_{\epsilon}}^{2})\\
\leq & \check{C}_{7}\int_{X}(|F_{A_{\epsilon}(t-1)}+[\phi_{\epsilon} (t-1),\phi_{\epsilon} (t-1)^{*H_{\epsilon}}]-\frac{1}{r}\tr F_{K}\otimes \Id_{E}|_{H_{\epsilon}}^{2}+2|\partial_{A_{\epsilon}(t-1)}\phi_{\epsilon} (t-1)|_{H_{\epsilon}}^{2})\frac{\omega^{n}}{n!}\\
\leq& \check{C}_{7}\int_{X}|\sqrt{-1}\Lambda_{\omega}F_{H_{\epsilon},\theta}-\lambda \Id_{E}|_{H_{\epsilon}}^{2}\frac{\omega^{n}}{n!}.
\end{split}
\end{equation}
Take $t=2$, then we know $\{H_{\epsilon}(2)\}$ is a family of approximate Higgs-projectively flat metrics, i.e.
\begin{equation}\label{k04}
\sup_{X}(|F_{H_{\epsilon}(2)}+[\theta ,\theta^{*H_{\epsilon}(2)}]-\frac{1}{r}\tr F_{K}\otimes \Id_{E}|_{H_{\epsilon}(2)}^{2}+2|\partial_{H_{\epsilon}(2)}\theta |_{H_{\epsilon}(2)}^{2})\rightarrow 0
\end{equation}
as $\epsilon \rightarrow 0$.
\end{proof}

\medskip

 Replacing the K\"ahler condition by the Gauduchon condition, one can easily get the following lemma by a similar argument as that in \cite[Proposition 2.2]{NZ1}.
\begin{lemma}\label{l:1}
Let $(X, \omega )$ be a compact Hermitian manifold of dimension $n$, satisfying $\partial\bar{\partial}\omega^{n-1}=\partial\bar{\partial}\omega^{n-2}=0$. Let $(\tilde{E},\bar{\partial}_{\tilde{E}},\tilde{\theta})$ be an approximate Higgs-Hermitian flat Higgs bundle of rank $r$, i.e. there exist a sequence of Hermitian metrics  $\tilde{H}_{\epsilon}$ such that $\sup_{X}|F_{\tilde{H}_{\epsilon}, \tilde{\theta }}|_{\tilde{H}_{\epsilon}}\rightarrow 0$ as $\epsilon \rightarrow 0$. If $\varsigma$ is a non-trivial $\tilde{\theta}$-invariant holomorphic section of $E$, then $\varsigma$ has no zeros.
\end{lemma}

\medskip

\begin{proof}[Proof of Theorem \ref{thm:1}]
The sufficiency is easy to show by directly calculating. Here we only prove the necessity. Suppose $(E,\bar{\partial}_{E},\theta)$ is strictly semi-stable. There is a torsion free coherent Higgs subsheaf $S$ of minimal rank $p$ with \begin{equation}\label{chern01}\frac{1}{\rank S}\deg_{\omega}(S)=\frac{1}{\rank E}\deg_{\omega}(E).\end{equation}   Thus $(S, \theta_{S})$ is a stable torsion free Higgs sheaf, where $\theta_{S}$ is the induced Higgs field. In the following, we prove that $(S, \theta_{S})$ is a Higgs subbundle, i.e. $S$ is locally free.

Consider the following exact sequence of Higgs sheaves:
\begin{equation}
0\rightarrow (S,\theta_{S})\rightarrow(E,\theta)\rightarrow(Q,\theta_{Q})\rightarrow 0.
\end{equation}
Let $\{H_{\epsilon}\}$ be the approximate Higgs-projectively flat metrics which are constructed in Theorem \ref{app01}. We can define the bundle isomorphisms
\begin{equation}
f_{\epsilon}:S\oplus Q\rightarrow E
\end{equation}
on $X\setminus Z$ with respect to $H_{\epsilon}$ such that $Q \simeq S^{\perp_{H_{\epsilon}}}$, where $Z$ is the singularity set of $S$. Then the pull-back holomorphic structure and Higgs field on $S\oplus Q$ are
\begin{equation}
\begin{split}
f_{\epsilon}^{*}(\bar{\partial}_{E})=\left(\begin{array}{cc}
 \bar{\partial}_{S} &  \gamma_{\epsilon} \\
  0  &  \bar{\partial}_{Q}
\end{array}\right),\ \ f_{\epsilon}^{*}(\theta)=\left(\begin{array}{cc}
 \theta_{S} &  \zeta_{\epsilon} \\
  0  &  \theta_{Q}
\end{array}\right),
\end{split}
\end{equation}
where $\gamma_{\epsilon}\in\Omega^{0,1}(Q^{*}\otimes S)$ is the second fundamental form and $\zeta_{\epsilon}\in\Omega^{1,0}(Q^{*}\otimes S)$.
On the basis of  Gauss-Codazzi equation, we know
\begin{equation}
f_{\epsilon}^{*}(F_{H_{\epsilon},\theta})^{1,1}=\left(\begin{split}
  &F_{H_{S,\epsilon},\theta_{S}}^{1,1}-\gamma_{\epsilon}\wedge\gamma_{\epsilon}^{*}+\zeta_{\epsilon}\wedge\zeta_{\epsilon}^{*} &  D_{S\otimes Q^{*}}^{1,0}\gamma_{\epsilon}+\zeta_{\epsilon}\wedge\theta_{Q}^{*}+\theta_{S}^{*}\wedge\zeta_{\epsilon} \\
  &D_{S^{*}\otimes Q}^{0,1}\gamma_{\epsilon}^{*}+\zeta_{\epsilon}^{*}\wedge\theta_{S}+\theta_{Q}\wedge\zeta_{\epsilon}^{*}  &F_{H_{Q,\epsilon},\theta_{Q}}^{1,1}-\gamma_{\epsilon}^{*}\wedge\gamma_{\epsilon}+\zeta_{\epsilon}^{*}\wedge\zeta_{\epsilon}
\end{split}\right).
\end{equation}
By (\ref{k040}), we have
\begin{equation}
\sqrt{-1}(F_{H_{S,\epsilon},\theta_{S}}^{1,1}-\gamma_{\epsilon}\wedge\gamma_{\epsilon}^{*}+\zeta_{\epsilon}\wedge\zeta_{\epsilon}^{*})-\frac{1}{r}\tr F_{K}\otimes \Id_{S}\leq \delta_{\epsilon} \omega \otimes \Id_{S},
\end{equation}
where $\delta_{\epsilon}\rightarrow 0$ as $\epsilon \rightarrow 0$.
Since $\tr([\theta_{S},\theta_{S}^{*}])=0$ and $\frac{\sqrt{-1}}{2\pi}\tr(-\gamma_{\epsilon}\wedge\gamma_{\epsilon}^{*}+\zeta_{\epsilon}\wedge\zeta_{\epsilon}^{*})\geq 0$, we conclude
\begin{equation}\label{tt:1}
0\leq \rank (S)\delta_{\epsilon} \omega -\sqrt{-1}\tr F_{H_{S,\epsilon}}+\frac{\rank(S)}{r}\tr F_{K}
\end{equation}
on $X\setminus Z$.
The fact (\ref{chern01}) gives us
\begin{equation}
\begin{split}
&\int_{X}( \rank (S)\delta_{\epsilon} \omega -\sqrt{-1}\tr F_{H_{S,\epsilon}}+\frac{\rank(S)}{r}\tr F_{K})\wedge \frac{\omega^{n-1}}{(n-1)!}\\
=& \int_{X} n \cdot \rank (S)\delta_{\epsilon} \frac{\omega^{n}}{n!}\rightarrow 0,\\
\end{split}
\end{equation}
and
\begin{equation}
\sqrt{-1}\tr F_{H_{S,\epsilon}}\rightharpoonup \frac{\rank(S)}{r}\tr F_{K}
\end{equation}
in the sense of current, as $\epsilon\rightarrow 0$.

Let $\hat{K}_{S}$ be a smooth Hermitian metric on  the determinant line bundle $\det(S)=(\wedge^{p}S)^{**}$ such that
\begin{equation}
\sqrt{-1}\Lambda_{\omega }F_{\hat{K}_{S}}=\lambda_{S}=\frac{\rank(S)}{r}\sqrt{-1}\Lambda_{\omega }\tr F_{K}.
\end{equation}
 Set $\eta =\sqrt{-1}F_{\hat{K}_{S}}-\frac{\rank(S)}{r}\sqrt{-1}\tr F_{K}$. We have
\begin{equation}
\begin{split}
0\leftarrow & \int_{X}( \sqrt{-1}\tr F_{H_{S,\epsilon}}-\frac{\rank(S)}{r}\tr F_{K})\wedge \eta \wedge  \frac{\omega^{n-2}}{(n-2)!}\\
=& \int_{X}\eta \wedge \eta \wedge  \frac{\omega^{n-2}}{(n-2)!}\\
=& -\int_{X} |\eta |^{2}\frac{\omega^{n}}{n!}.
\end{split}
\end{equation}
Hence $\eta =0$, i.e.
\begin{equation}\label{chern02}
\sqrt{-1}F_{\hat{K}_{S}}=\frac{\rank(S)}{r}\sqrt{-1}\tr F_{K}.
\end{equation}
  Let us consider the Higgs bundle  $(\tilde{E}, \overline{\partial}_{\tilde{E}}, \tilde{\theta })=(\wedge^{p}E\otimes(\det S)^{-1},\theta_{p})$, where $\theta_{p}$ is the induced Higgs field on $\wedge^{p}E\otimes(\det S)^{-1}$. Of course (\ref{k040}) and (\ref{chern02}) imply   the Higgs bundle $(\tilde{E}, \overline{\partial}_{\tilde{E}}, \tilde{\theta })$ is approximate Higgs-Hermitian flat, i.e. there exist a sequence of Hermitian metrics  $\tilde{H}_{\epsilon}$ such that $\sup_{X}|F_{\tilde{H}_{\epsilon}, \tilde{\theta }}|_{\tilde{H}_{\epsilon}}\rightarrow 0$ as $\epsilon \rightarrow 0$. A Higgs morphism $\det(S)\rightarrow \wedge^{p}E$ can be seen as a $\theta_{p}$-invariant section of $\wedge^{p}E\otimes(\det S)^{-1}$. Lemma \ref{l:1} tells us that the non-zero Higgs morphism $\det(S)\rightarrow \wedge^{p}E$ is injective. Then by Lemma 1.20 in \cite{DPM}, we know that $S$ is a Higgs subbundle of $E$.

From the exact sequence of Higgs bundles, one can see that
\begin{equation}
ch_{1}(S, \overline{\partial}_{S})+ch_{1}(Q, \overline{\partial}_{Q})=ch_{1}(E, \overline{\partial}_{E}),\end{equation} and \begin{equation}ch_{2}(S, \overline{\partial}_{S})+ch_{2}(Q, \overline{\partial}_{Q})=ch_{2}(E, \overline{\partial}_{E}).\end{equation}
By (\ref{chern02}) and the Gauss-Codazzi equation, there exist metrics $K_{S}$ and $K_{Q}$ such that
\begin{equation}\label{chern03}
\frac{1}{\rank(S)}\sqrt{-1}\tr F_{K_{S}}=\frac{1}{\rank(Q)}\sqrt{-1}\tr F_{K_{Q}}=\frac{1}{\rank(E)}\sqrt{-1}\tr F_{K},
\end{equation}
and then
\begin{equation}\label{chern04}
\begin{split}
&\frac{1}{\rank(E)}ch_{1}(E, \overline{\partial}_{E})\wedge ch_{1}(E, \overline{\partial}_{E})\\=&\frac{1}{\rank(S)}ch_{1}(S, \overline{\partial}_{S})\wedge ch_{1}(S, \overline{\partial}_{S})+\frac{1}{\rank(Q)}ch_{1}(Q, \overline{\partial}_{Q})\wedge ch_{1}(Q, \overline{\partial}_{Q}).\\
\end{split}
\end{equation}
On the other hand, we know that $(S, \overline{\partial}_{S}, \theta_{S})$ is stable and $(Q, \overline{\partial}_{Q}, \theta_{Q})$ is semi-stable. According to Theorem \ref{thm:3}, we have the following Bogomolov type inequalities
\begin{equation}
(-2ch_{2}(S, \overline{\partial}_{S})+\frac{1}{\rank(S)}ch_{1}(S, \overline{\partial}_{S})\wedge ch_{1}(S, \overline{\partial}_{S}))\cdot [\omega]^{n-2}\geq 0
\end{equation}
and
\begin{equation}
(-2ch_{2}(Q, \overline{\partial}_{Q})+\frac{1}{\rank(Q)}ch_{1}(Q, \overline{\partial}_{Q})\wedge ch_{1}(Q, \overline{\partial}_{Q}))\cdot [\omega]^{n-2}\geq 0.
\end{equation}
So we conclude that both $(S, \overline{\partial}_{S}, \theta_{S})$  and $(Q, \overline{\partial}_{Q}, \theta_{Q})$ satisfy the vanishing  Chern number condition (\ref{CC}).
Set $(E_{1},\theta_{E_{1}})=(S,\theta_{S})$ and then establish the desired filtration by induction on the rank.
\end{proof}

\section{Proof of Theorem \ref{thm:2-2}}

First, under the assumption of Theorem \ref{thm:2-2}, we can prove:
\begin{lemma}\label{nl:1}
Suppose $(E,\bar{\partial}_{E},\theta,H)$ is a Higgs bundle with metric $H$ satisfying $F_{H,\theta}=0$. Then
\begin{equation}
H^{1}_{DR}(X,E)\simeq H^{1}_{Dol}(X,E).
\end{equation}
\end{lemma}
\begin{proof}
We will show that for any $[\beta]\in H_{Dol}^{1}(X,E)$, there is a representative $\tilde{\beta}\in[\beta]$, such that $D_{H,\theta}\tilde{\beta}=0$. Consider the following equation
\begin{equation}
\sqrt{-1}\Lambda_{\omega}D^{'}_{H}D^{''}_{E}\gamma=-\kappa.
\end{equation}
Making use of the continuity method, one can prove that the solvability of this equation is equivalent to the one of
\begin{equation}
\int_{X}\langle\kappa,\vartheta\rangle_{H}\frac{\omega^{n}}{n!}=0,
\end{equation}
for any $\vartheta\in\Gamma(X,E)$ satisfying $D^{''}_{E}\vartheta=D_{H}^{'}\vartheta=0$. By the assumption $\int_{X}\partial [\eta]\wedge\frac{\omega^{n-1}}{(n-1)!}=0$ for any Dolbeault class $[\eta]\in H^{0,1}(X)$ and $D^{''}_{E}\beta =0$, we know
\begin{equation}
\int_{X}\langle\sqrt{-1}\Lambda_{\omega}D^{'}_{H}\beta,\vartheta\rangle_{H}\frac{\omega^{n}}{n!}
=\int_{X}\sqrt{-1}\partial\langle\beta^{0,1} ,\vartheta\rangle_{H}\wedge\frac{\omega^{n-1}}{(n-1)!}
=0.
\end{equation}
Suppose $\gamma\in\Gamma(X,E)$ is a solution of
\begin{equation}
\sqrt{-1}\Lambda_{\omega}D^{'}_{H}D^{''}_{E}\gamma=-\sqrt{-1}\Lambda_{\omega}D^{'}_{H}\beta.
\end{equation}
Let $\tilde{\beta}=\beta+D^{''}_{E}\gamma$, then $\sqrt{-1}\Lambda_{\omega}D_{H}^{'}\tilde{\beta}=0$.  According to (\ref{z:1}), one can easily check that
\begin{equation}
\sqrt{-1}[\Lambda_{\omega},D_{E}^{''}]=(D_{H}^{'})^{*}+\tau^{*},\ \ -\sqrt{-1}[\Lambda_{\omega},D_{H}^{'}]=(D_{E}^{''})^{*}+\bar{\tau}^{*}.
\end{equation}
A simple computation gives
\begin{equation}\label{eq:61}
\begin{split}
0&=\int_{X}\langle\sqrt{-1}[\Lambda_{\omega},F_{H,\theta}]\tilde{\beta},\tilde{\beta}\rangle_{H,\omega}\frac{\omega^{n}}{n!}\\
&=\int_{X}\langle\sqrt{-1}\Lambda_{\omega}D^{''}_{E}D^{'}_{H}\tilde{\beta},\tilde{\beta}\rangle_{H,\omega}\frac{\omega^{n}}{n!}\\
&=\int_{X}\langle\sqrt{-1}[\Lambda_{\omega},D^{''}_{E}]D^{'}_{H}\tilde{\beta},\tilde{\beta}\rangle_{H,\omega}\frac{\omega^{n}}{n!}\\
&=\int_{X}|D^{'}_{H}\tilde{\beta}|^{2}_{H,\omega}\frac{\omega^{n}}{n!}+\int_{X}\langle\tau^{*}D^{'}_{H}\tilde{\beta},\tilde{\beta}\rangle_{H,\omega}\frac{\omega^{n}}{n!}.
\end{split}
\end{equation}
From Proposition 4.1 in \cite{MZ}, we can see that for any $\alpha\in\Omega^{2}(E)$ with $\Lambda_{\omega}\alpha=0$, there holds
\begin{equation}
(\tau+\bar{\tau}^{*})\alpha=\frac{-*(d(\omega^{n-2})\wedge(\alpha^{1,1}-\alpha^{2,0}-\alpha^{0,2}))}{(n-2)!}.
\end{equation}
Let $\alpha=D^{'}_{H}\tilde{\beta}$. Then using Stokes formula and $D^{''}_{E}\tilde{\beta}=0$, we derive
\begin{equation}
\begin{split}
&\int_{X}\langle\tau^{*}D^{'}_{H}\tilde{\beta},\tilde{\beta}\rangle_{H,\omega}\frac{\omega^{n}}{n!}\\
=&\int_{X}\langle\frac{-*(\bar{\partial}(\omega^{n-2})\wedge(\partial_{H}\tilde{\beta}^{0,1}+\theta^{*}(\tilde{\beta}^{1,0})-\partial_{H}\tilde{\beta}^{1,0}))}{(n-2)!},\tilde{\beta}^{0,1}+\tilde{\beta}^{1,0}\rangle_{H,\omega}\frac{\omega^{n}}{n!}\\
=&\frac{-1}{(n-2)!}\int_{X} \langle\bar{\partial}\omega^{n-2}\wedge(\partial_{H}\tilde{\beta}^{0,1}+\theta^{*}(\tilde{\beta}^{1,0})-\partial_{H}\tilde{\beta}^{1,0}),\tilde{\beta}^{0,1}+\tilde{\beta}^{1,0}\rangle_{H}\\
=&\frac{-1}{(n-2)!}\int_{X} \bar{\partial}\omega^{n-2}\wedge \langle\partial_{H}\tilde{\beta}^{0,1},\tilde{\beta}^{0,1}\rangle_{H}\\
&-\frac{1}{(n-2)!}\int_{X} \bar{\partial}\omega^{n-2}\wedge \langle\theta^{*}(\tilde{\beta}^{1,0}),\tilde{\beta}^{0,1}\rangle_{H}\\
&+\frac{1}{(n-2)!}\int_{X} \bar{\partial}\omega^{n-2}\wedge \langle\partial_{H}\tilde{\beta}^{1,0},\tilde{\beta}^{1,0}\rangle_{H}\\
=&\frac{1}{(n-2)!}\int_{X}\bar{\partial}\omega^{n-2}\wedge \langle\tilde{\beta}^{1,0},\bar{\partial}_{E}\tilde{\beta}^{1,0}+\theta(\tilde{\beta}^{0,1})\rangle_{H}\\
=&0.
\end{split}
\end{equation}
Hence $D^{'}_{H}\tilde{\beta}=0$, and $[\tilde{\beta}]\in H^{1}_{DR}(E)$.

\medskip

Conversely, take $[\beta_{1}]\in H^{1}_{DR}(E)$. For any $\vartheta\in\Gamma(X,E)$ satisfying $D^{''}_{E}\vartheta=D_{H}^{'}\vartheta=0$, we have
\begin{equation}
\int_{X}\langle\sqrt{-1}\Lambda_{\omega}D^{'}_{H}\beta_{1},\vartheta\rangle_{H}\frac{\omega^{n}}{n!}
=\int_{X}\sqrt{-1}\partial\langle\beta_{1}^{0,1},\vartheta\rangle_{H}\wedge\frac{\omega^{n-1}}{(n-1)!}
=0.
\end{equation}
Assume $\gamma_{1}\in\Gamma(X,E)$ is a solution of
\begin{equation}
\sqrt{-1}\Lambda_{\omega}D^{'}_{H}D^{''}_{E}\gamma_{1}=-\sqrt{-1}\Lambda_{\omega}D^{'}_{H}\beta_{1}.
\end{equation}
Let $\tilde{\beta}_{1}=\beta_{1}+D\gamma_{1}$, then $\sqrt{-1}\Lambda_{\omega}D_{E}^{''}\tilde{\beta}_{1}=0$ and $\sqrt{-1}\Lambda_{\omega}D_{H}^{'}\tilde{\beta}_{1}=0$. Similar to (\ref{eq:61}), we obtain
\begin{equation}
\begin{split}
0=&\int_{X}\langle\sqrt{-1}[\Lambda_{\omega},DD^{c}_{H}+D_{H}^{c}D]\tilde{\beta}_{1},\tilde{\beta}_{1}\rangle_{H,\omega}\frac{\omega^{n}}{n!}\\
=&\int_{X}\langle\sqrt{-1}\Lambda_{\omega}DD^{c}_{H}\tilde{\beta}_{1},\tilde{\beta}_{1}\rangle_{H,\omega}\frac{\omega^{n}}{n!}.
\end{split}
\end{equation}
So
\begin{equation}\label{new:1}
\begin{split}
-2\int_{X}|D^{''}_{E}\tilde{\beta}_{1}|^{2}_{H,\omega}\frac{\omega^{n}}{n!}-2\int_{X}\langle\bar{\tau}^{*}D^{''}_{E}\tilde{\beta}_{1},\tilde{\beta}_{1}\rangle_{H,\omega}\frac{\omega^{n}}{n!}=0,
\end{split}
\end{equation}
and
\begin{equation}\label{new:2}
\begin{split}
-2\int_{X}|D^{'}_{H}\tilde{\beta}_{1}|^{2}_{H,\omega}\frac{\omega^{n}}{n!}-2\int_{X}\langle\tau^{*}D^{'}_{H}\tilde{\beta}_{1},\tilde{\beta}_{1}\rangle_{H,\omega}\frac{\omega^{n}}{n!}=0.
\end{split}
\end{equation}
Set $A=\int_{X}\langle\bar{\tau}^{*}D^{''}_{E}\tilde{\beta}_{1},\tilde{\beta}_{1}\rangle_{H,\omega}\frac{\omega^{n}}{n!}+\int_{X}\langle\tau^{*}D^{'}_{H}\tilde{\beta}_{1},\tilde{\beta}_{1}\rangle_{H,\omega}\frac{\omega^{n}}{n!}$. Then
 \begin{equation}
\begin{split}
A=&-\frac{1}{(n-2)!}\int_{X}\partial\omega^{n-2}\wedge \langle\bar{\partial}_{E}\tilde{\beta}_{1}^{1,0},\tilde{\beta}_{1}^{1,0}\rangle_{H}
-\frac{1}{(n-2)!}\int_{X}\partial\omega^{n-2}\wedge \langle\theta(\tilde{\beta}_{1}^{0,1}),\tilde{\beta}^{1,0}\rangle_{H}\\
&+\frac{1}{(n-2)!}\int_{X}\partial\omega^{n-2}\wedge \langle\bar{\partial}_{E}\tilde{\beta}_{1}^{0,1},\tilde{\beta}_{1}^{0,1}\rangle_{H}
-\frac{1}{(n-2)!}\int_{X} \bar{\partial}\omega^{n-2}\wedge \langle\partial_{H}\tilde{\beta}^{0,1},\tilde{\beta}^{0,1}\rangle_{H}\\
&-\frac{1}{(n-2)!}\int_{X} \bar{\partial}\omega^{n-2}\wedge \langle\theta^{*}(\tilde{\beta}^{1,0}),\tilde{\beta}^{0,1}\rangle_{H}
+\frac{1}{(n-2)!}\int_{X} \bar{\partial}\omega^{n-2}\wedge \langle\partial_{H}\tilde{\beta}^{1,0},\tilde{\beta}^{1,0}\rangle_{H}.
\end{split}
\end{equation}
It is not hard to find that $\bar{A}=-A$, so it is imaginary. Adding (\ref{new:1}) and (\ref{new:2}), we get $A=0$ and
\begin{equation}
\int_{X}|D^{''}_{E}\tilde{\beta}_{1}|^{2}_{H}\frac{\omega^{n}}{n!}+\int_{X}|D^{'}_{H}\tilde{\beta}_{1}|^{2}_{H}\frac{\omega^{n}}{n!}=0.
\end{equation}
Hence $[\tilde{\beta}_{1}]\in H^{1}_{Dol}(X,E)$.
\end{proof}

\subsection{From Higgs bundle to projectively flat bundle}
First, we construct a map $j:\mathcal{C}_{Dol}(E)\rightarrow \mathcal{C}_{DR}(E, \alpha )$ in rank $2$ case or when the length of the filtration in Theorem \ref{thm:1} is only $2$. Let $(E,\bar{\partial}_{E},\theta)$ be a semi-stable Higgs bundle with vanishing Chern number (\ref{CC}). Now we discuss the following two cases that may occur.

{\bf Case 1.} $(E,\bar{\partial}_{E},\theta)$ is polystable. By Theorem \ref{thm:2-1}, there is a projectively flat connection on $E$ denoted by $D_{E}$ and $(E,D_{E})$ is semi-simple. Define $j$ mapping from the equivalent class of $(E,\bar{\partial}_{E},\theta)$ to the equivalent class of $(E,D_{E})$.

{\bf Case 2.} $(E,\bar{\partial}_{E},\theta)$ is strictly semi-stable. By Theorem \ref{thm:1}, there is a filtration
\begin{equation}
0\subset S\subset E.
\end{equation}
It induces an exact sequence of Higgs bundles
\begin{equation}\label{exact:1}
0\rightarrow (S,\bar{\partial}_{S},\theta_{S})\rightarrow (E,\bar{\partial}_{E},\theta)\rightarrow (Q,\bar{\partial}_{Q},\theta_{Q})\rightarrow 0,
\end{equation}
where $(S,\bar{\partial}_{S}, \theta_{S})$ and $(Q,\bar{\partial}_{Q}, \theta_{Q})$ are stable Higgs bundles with vanishing Chern number (\ref{CC}).

Letting $H_{S}$ and $H_{Q}$ be Hermitian-Einstein metrics on the Higgs bundles $(S,\bar{\partial}_{S}, \theta_{S})$ and $(Q, \bar{\partial}_{Q}, \theta_{Q})$, we have
\begin{equation}
F_{H_{S}, \theta_{S}}=\frac{1}{\rank(S)}\sqrt{-1}\tr F_{H_{S}}\otimes \Id_{S}
\end{equation}
and
\begin{equation}
F_{H_{Q}, \theta_{Q}}=\frac{1}{\rank(Q)}\sqrt{-1}\tr F_{H_{Q}}\otimes \Id_{Q}.
\end{equation}
By (\ref{chern03}), we know that
\begin{equation}\label{chern06}
\frac{1}{\rank(S)}\sqrt{-1}\tr F_{H_{S}}=\frac{1}{\rank(Q)}\sqrt{-1}\tr F_{H_{Q}}=\frac{1}{\rank(E)}\sqrt{-1}\tr F_{K},
\end{equation}
where $K$ is a Hermitian metric on $E$ such that $\Lambda_{\omega }F_{K} $ is constant.
 For simplicity, we denote the Hitchin-Simpson connections on $(S, \bar{\partial}_{S}, \theta_{S})$ and $(Q, \bar{\partial}_{Q}, \theta_{Q})$ with respect to the Hermitian-Einstein metrics by $D_{S}$ and $D_{Q}$. Due to (\ref{chern06}), the induced Higgs bundle $(Q^{\ast} \otimes S, \bar{\partial}_{Q^{\ast} \otimes S}, \theta_{Q^{\ast} \otimes S})$ is Higgs-Hermitian flat, i.e.
 \begin{equation}
 F_{D_{S}\otimes D_{Q^{\ast}}}=0.
 \end{equation}

For the exact sequence (\ref{exact:1}), it determines a Higgs extension class in $H^{1}_{Dol}(Q^{*}\otimes S)$. Choose a $C^{\infty}$ splitting $f:S\oplus Q\rightarrow E$. The pull-back of $\bar{\partial}_{E}$ and $\theta$ on $S\oplus Q$ can be written as:
\begin{equation}
f^{*}(\bar{\partial}_{E})=\left(\begin{split}
  &\bar{\partial}_{S}  &\gamma \\
  &0   &\bar{\partial}_{Q}
\end{split}\right),\ \
f^{*}(\theta)=\left(\begin{split}
  &\theta_{S}  &\xi \\
  &0   &\theta_{Q}
\end{split}\right),
\end{equation}
where $\gamma\in\Omega^{0,1}(Q^{*}\otimes S)$ is the second fundamental form and $\xi\in\Omega^{1,0}(Q^{*}\otimes S)$. Setting $\beta=\gamma+\xi$, one can check that $D^{''}_{Q^{*}\otimes S}\beta=0$. The Higgs extension class can be presented by $\beta$. It is well-known that there is a one-to-one correspondence between the equivalence class of Higgs extensions of $(Q,\theta_{Q})$ by $(S,\theta_{S})$ and the elements in $H^{1}_{Dol}(Q^{*}\otimes S)$.

Let $H_{Q^{*}\otimes S}$ be the Higgs-Hermitian flat metric on $Q^{*}\otimes S$, $D_{Q^{*}\otimes S}$ be the related flat connection. Because of Lemma \ref{nl:1}, there is a $\tilde{\beta}=\beta+D^{''}\gamma\in[\beta]$ satisfying $D_{Q^{*}\otimes S}\tilde{\beta}=0$. Define $\tilde{f}:S\oplus Q\rightarrow E$ to be
\begin{equation}
\tilde{f}=f\circ \left(\begin{split}
  &\Id_{S}  &\gamma \\
  &0  &\Id_{Q}
\end{split}\right),
\end{equation}
then
\begin{equation}
\tilde{f}^{*}(D_{E}^{''})=\left(\begin{split}
  &D_{S}^{''}  &\tilde{\beta} \\
  &0  &D_{Q}^{''}
\end{split}\right).
\end{equation}
Define the  connection $D_{E}$ on $E$ by
\begin{equation}
\tilde{f}^{*}(D_{E})=\left(\begin{split}
  &D_{S}  &\tilde{\beta} \\
  &0  &D_{Q}
\end{split}\right),
\end{equation}
then $(E,D_{E})$ is projectively flat, i.e.
\begin{equation}
F_{D_{E}}=\frac{1}{\rank(S)}\sqrt{-1}\tr F_{H_{S}}\otimes \Id_{E},
\end{equation}
 and
\begin{equation}
0\rightarrow (S,D_{H_{s},\theta_{S}})\rightarrow (E,D_{E})\rightarrow (Q,D_{H_{Q},\theta})\rightarrow 0
\end{equation}
is an exact sequence of projective flat bundles. Define $j$ mapping from the equivalent class of $(E,\bar{\partial}_{E},\theta)$ to the equivalent class of $(E,D_{E})$.

We claim this map is well-defined. It is sufficient to prove that the map $j$ is independent of the choice of filtrations. Suppose we have another filtration
\begin{equation}
0\subset \tilde{S}\subset E.
\end{equation}
The induced exact sequence is
\begin{equation}
0\rightarrow (\tilde{S},\bar{\partial}_{\tilde{S}},\theta_{\tilde{S}})\rightarrow (E,\bar{\partial}_{E},\theta)\rightarrow (\tilde{Q},\bar{\partial}_{\tilde{Q}},\theta_{\tilde{Q}})\rightarrow 0.
\end{equation}
Choose a suitable splitting $\tilde{g}:\tilde{S}\oplus \tilde{Q}\rightarrow E$. The pull-back of $D_{E}^{''}$ is
\begin{equation}
\tilde{g}^{*}(D_{E}^{''})=\left(\begin{split}
  &D_{\tilde{S}}^{''}   &\tilde{\rho} \\
  &0   &D_{\tilde{Q}}^{''}
\end{split}\right),
\end{equation}
and $D^{'}\tilde{\rho}=0$. Define the projectively flat connection on $E$ by
\begin{equation}
\tilde{g}^{*}(\tilde{D}_{E})=\left(\begin{split}
  &D_{\tilde{S}}  &\tilde{\rho} \\
  &0  &D_{\tilde{Q}}
\end{split}\right).
\end{equation}
Now we need to show $(E,D_{E})\simeq(E,\tilde{D}_{E})$. First, we have a Higgs isomorphism
\begin{equation}
P=\tilde{g}^{-1}\circ \tilde{f}:(S\oplus Q,\left(\begin{split}
  &D_{S}^{''}  &\tilde{\beta} \\
  &0  &D_{Q}^{''}
\end{split}\right))\rightarrow (\tilde{S}\oplus \tilde{Q},\left(\begin{split}
  &D_{\tilde{S}}^{''}  &\tilde{\rho} \\
  &0  &D_{\tilde{Q}}^{''}
\end{split}\right)).
\end{equation}
That is,
\begin{equation}
P\circ\left(\begin{split}
  &D_{S}^{''}  &\tilde{\beta} \\
  &0  &D_{Q}^{''}
\end{split}\right)=\left(\begin{split}
  &D_{\tilde{S}}^{''}  &\tilde{\rho} \\
  &0  &D_{\tilde{Q}}^{''}
\end{split}\right)\circ P.
\end{equation}
Let
\begin{equation}
P=\left(\begin{split}
  &P_{1}^{1}  &P_{1}^{2} \\
  &P_{2}^{1}  &P_{2}^{2}
\end{split}\right),
\end{equation}
then
\begin{equation}
\left(\begin{split}
  &P_{1}^{1}\circ D_{S}^{''}  &P_{1}^{1}\circ\tilde{\beta}+P_{1}^{2}\circ D_{Q}^{''} \\
  &P_{2}^{1}\circ D_{S}^{''}  &P_{2}^{1}\circ\tilde{\beta}+P_{2}^{2}\circ D_{Q}^{''}
\end{split}\right)=\left(\begin{split}
  &D_{\tilde{S}}^{''}\circ P_{1}^{1}+\tilde{\rho}\circ P_{2}^{1}  &D_{\tilde{S}}^{''}\circ P_{1}^{2}+\tilde{\rho}\circ P_{2}^{2} \\
  &\qquad D_{\tilde{Q}}^{''}\circ P_{2}^{1}  &D_{\tilde{Q}}^{''}\circ P_{2}^{2}\qquad
\end{split}\right).
\end{equation}
Comparing both sides of this equation, we get
\begin{equation}\label{neweq:1}
\left\{\begin{split}
&D^{''}(P_{2}^{1})=0;\\
&D^{''}(P_{1}^{1})+\tilde{\rho}\circ P_{2}^{1}=0;\\
&D^{''}(P_{2}^{2})-P_{2}^{1}\circ\tilde{\beta}=0;\\
&D^{''}(P_{1}^{2})+\tilde{\rho}\circ P_{2}^{2}-P_{1}^{1}\circ\tilde{\beta}=0.
\end{split}\right.
\end{equation}
Using $D^{''}(P_{2}^{1})=0$, we know $\sqrt{-1}\Lambda_{\omega}D^{'}D^{''}(P_{2}^{1})=0$. This means $D^{'}(P_{2}^{1})=0$.
By $D^{'}\tilde{\rho}=D^{'}\tilde{\beta}=0$, one can obtain
\begin{equation}
\sqrt{-1}\Lambda_{\omega}D^{'}D^{''}(P_{1}^{1})=0, \ \sqrt{-1}\Lambda_{\omega}D^{'}D^{''}(P_{2}^{2})=0.
\end{equation}
So $D^{'}(P_{1}^{1})=D^{''}(P_{1}^{1})=0$ and $D^{'}(P_{2}^{2})=D^{''}(P_{2}^{2})=0$. From equation (\ref{neweq:1}) and  $D^{'}\tilde{\rho}=D^{'}\tilde{\beta}=0$, we also have $\sqrt{-1}\Lambda_{\omega}D^{'}D^{''}(P_{1}^{2})=0$. Then $D^{'}(P_{1}^{2})=0$ and $D^{''}(P_{1}^{2})=0$. Together with all of the above, we see
 \begin{equation}
P\circ\left(\begin{split}
&D_{S}^{'}  &0 \\
&0  &D_{Q}^{'}
\end{split}\right)=\left(\begin{split}
  &D_{\tilde{S}}^{'}  &0 \\
  &0  &D_{\tilde{Q}}^{'}
\end{split}\right)\circ P.
\end{equation}
Hence, $(E,D_{E})\simeq(E,\tilde{D}_{E})$.

\subsection{From projectively flat bundle to Higgs bundle}
Now we construct a map $i:\mathcal{C}_{DR}(E, \alpha )\rightarrow \mathcal{C}_{Dol}(E )$. Letting $(E,D)\in \mathcal{C}_{DR}(E, \alpha )$, by Theorem \ref{thm:2-1},  we only need to consider the case that $(E,D)$ is not semi-simple. Let $(S,D_{S})$ be a $D$-invariant subbundle of $(E,D)$ with minimal rank, then we have the following exact sequence of projectively flat bundles
\begin{equation}
0\rightarrow (S,D_{S})\rightarrow (E,D)\rightarrow (Q,D_{Q})\rightarrow 0.
\end{equation}
It is obvious that $(S,D_{S})$ is a simple projectively flat bundle. In the following, we suppose that $\rank (E)=2$ or $(Q,D_{Q})$ is simple. According to Theorem \ref{thm:2-1}, we can get Higgs structures $(\bar{\partial}_{S},\theta_{S})$ and $(\bar{\partial}_{Q},\theta_{Q})$ on $S$ and $Q$ by choosing harmonic metrics. It is straightforward to find that the induced Higgs bundle $(Q^{*}\otimes S, \bar{\partial}_{Q^{*}\otimes S},\theta_{Q^{*}\otimes S})$ is Higgs-Hermitian flat, i.e. $F_{D_{Q^{*}\otimes S}}=0$, where $D_{Q^{*}\otimes S}$ is the corresponding Hitchin-Simpson connection.

 Choose a $C^{\infty}$ splitting $f:S\oplus Q\rightarrow E$, then the pull-back of $D$ can be expressed as
\begin{equation}
f^{*}(D)=\left(\begin{split}
  &D_{S}  &\beta \\
  &0  &D_{Q}
\end{split}\right),
\end{equation}
and one can check that $D_{Q^{*}\otimes S}\beta=0$. Thus $\beta$ determines a flat extension class in $H^{1}_{DR}(Q^{*}\otimes S)$ of the flat bundle $(Q^{*}\otimes S, D_{Q^{*}\otimes S})$. By Lemma \ref{nl:1}, there is a $\tilde{\beta}=\beta+D_{Q^{*}\otimes S}\gamma\in [\beta]$ satisfying $D_{Q^{*}\otimes S}^{''}\tilde{\beta}=0$. Define $\tilde{f}:S\oplus Q\rightarrow E$ to be
\begin{equation}
\tilde{f}=f\circ \left(\begin{split}
  &\Id_{S}  &\gamma \\
  &0  &\Id_{Q}
\end{split}\right),
\end{equation}
then
\begin{equation}
\tilde{f}^{*}(D_{E})=\left(\begin{split}
  &D_{S}  &\tilde{\beta} \\
  &0  &D_{Q}
\end{split}\right).
\end{equation}
Define the holomorphic structure $\bar{\partial}_{E}$ and Higgs field $\theta_{E}$ on $E$ by
\begin{equation}
\tilde{f}^{*}(\bar{\partial}_{E})=\left(\begin{split}
  &\bar{\partial}_{S}  &\tilde{\beta}^{0,1} \\
  &0  &\bar{\partial}_{Q}
\end{split}\right)  \quad \mbox{and} \quad
\tilde{f}^{*}(\theta_{E})=\left(\begin{split}
  &\theta_{S}  &\tilde{\beta}^{1,0} \\
  &0  &\theta_{Q}
\end{split}\right).
\end{equation}
Then $(E,\bar{\partial}_{E},\theta_{E})$ is a semi-stable Higgs bundle with vanishing Chern number (\ref{CC}) and
\begin{equation}
0\rightarrow (S,\bar{\partial}_{S},\theta_{S})\rightarrow (E,\bar{\partial}_{E},\theta_{E})\rightarrow (Q,\bar{\partial}_{Q},\theta_{Q})\rightarrow 0
\end{equation}
is an exact sequence of Higgs bundles. Define $i$ mapping from the equivalent class of $(E,D)$ to the equivalent class of $(E,\bar{\partial}_{E},\theta_{E})$.  We can  show that this definition is well-defined. Because the proof is very similar as that for the map $j$, we omit it here.

\medskip

\begin{proof}[Proof of Theorem \ref{thm:2-2}]
From above, we define two maps $i$ and $j$. It is easy to see $i\circ j=\Id_{\mathcal{C}_{Dol}(E)}$ and $j\circ i=\Id_{\mathcal{C}_{DR}(E, \alpha )}$. So $i$ and $j$ are one-to-one maps, which finishes the proof.
\end{proof}

\section{Proof of Theorem \ref{thm:2-3}}
In this section, we assume that $(X,\omega)$ is a K\"ahler manifold. Let $\mathfrak{E}=(E,\bar{\partial}_{E},\theta)$ be a semi-stable Higgs bundle over $X$ with vanishing Chern number (\ref{CC}). By Theorem \ref{thm:1}, there is  a filtration
\begin{equation}
0=\mathfrak{E}_{0}\subset \mathfrak{E}_{1}\subset\cdots\subset \mathfrak{E}_{l}=\mathfrak{E}
\end{equation}
by Higgs subbundles such that the quotients $\mathfrak{Q}_{k}=(Q_{k}, \overline{\partial}_{Q_{k}}, \theta_{k})=\mathfrak{E}_{k}/\mathfrak{E}_{k-1}$ are stable Higgs bundles with vanishing Chern number (\ref{CC}) and $\frac{1}{\rank(Q_{k})}c_{1}(Q_{k}, \overline{\partial}_{Q_{k}})=\frac{1}{\rank(E)}c_{1}(E, \overline{\partial}_{E})=\frac{1}{2\pi }[\alpha ]$. Let $H_{k}$ be the Hermitian-Einstein metric on the stable Higgs bundle $(Q_{k}, \overline{\partial}_{Q_{k}}, \theta_{k})$, then we have
\begin{equation}
\sqrt{-1}F_{H_{k}, \theta_{k}}=\alpha \otimes \Id_{Q_{K}}.
\end{equation}
So every induced Higgs bundle $(Q_{k}^{\ast }\otimes Q_{j}, \overline{\partial}_{Q_{k}^{\ast }\otimes Q_{j}}, \theta_{Q_{k}^{\ast }\otimes Q_{j}})$ is Higgs-Hermitian flat. In the following, we denote
\begin{equation}
D_{Q_{k}}=D_{Q_{k}}^{'}+D_{Q_{k}}^{''}, \quad D_{Q_{k}}^{''}=\bar{\partial}_{Q_{k}}+\theta_{k}, \quad and \quad D_{Q_{k}}^{'}=\partial_{H_{k}}+\theta_{k}^{\ast H_{k}}.
\end{equation}

 Take a smooth splitting $f:\oplus_{i=1}^{l}Q_{i}\rightarrow E$, then the pull back of $D_{E}^{''}$ can be written as
\begin{equation}
\begin{split}
f^{*}(D_{E}^{''})=\left(\begin{array}{ccccc}
  D_{Q_{1}}^{''} &\cdots & \beta_{1}^{l} \\
   \vdots  &\ddots & \vdots\\
  0  & \cdots & D_{Q_{l}}^{''}
\end{array}\right).
\end{split}
\end{equation}
In the following, we  denote the induced  Hitchin-Simpson connections on $Q_{j}^{\ast} \otimes Q_{i}$ by $D'+D''$ for simplicity. The condition $(D_{E}^{''})^{2}=0$ implies
\begin{equation}\label{eq:7-1}
\left\{\begin{split}
&(D_{Q_{i}}^{''})^{2}=0, \ \ i=1,\cdots,l;\\
&D^{''}(\beta_{i}^{i+1})=0,\ \ i=1,\cdots,l-1;\\
&D^{''}(\beta_{i}^{j})+\sum_{k=i+1}^{j-1}\beta_{i}^{k}\wedge\beta_{k}^{j}=0,\ \ 1\leq i\leq j-2\leq l-2.
\end{split}\right.
\end{equation}

\begin{lemma}\label{l:7-1}
There exists a smooth splitting $\tilde{f}:\oplus_{i=1}^{l}Q_{i}\rightarrow E$, such that $\tilde{\beta}_{i}^{j}$ satisfies the equation (\ref{eq:7-1}) and $D^{'}\tilde{\beta}_{i}^{j}=0$ for $1\leq i<j\leq l$.
\end{lemma}
\begin{proof}
Let $\tilde{\beta}_{i}^{i+1}=\beta_{i}^{i+1}+D^{''}\gamma_{i}^{i+1}$, where $\gamma_{i}^{i+1}\in \Gamma(X,Q_{i+1}^{*}\otimes Q_{i})$ satisfies
\begin{equation}
\sqrt{-1}\Lambda_{\omega}D^{'}D^{''}\gamma_{i}^{i+1}=-\sqrt{-1}\Lambda_{\omega}D^{'}\beta_{i}^{i+1}.
\end{equation}
Since $(Q_{i+1}^{\ast }\otimes Q_{i}, \overline{\partial}_{Q_{i+1}^{\ast }\otimes Q_{i}}, \theta_{Q_{i+1}^{\ast }\otimes Q_{i}})$ is Higgs-Hermitian flat,  by (\ref{eq:61}), we have \begin{equation}D^{'}\tilde{\beta}_{i}^{i+1}=0.\end{equation} Let $h_{i}^{i+1}:\oplus_{k=1}^{l}Q_{k}\rightarrow \oplus_{k=1}^{l}Q_{k}$ and $h_{i}^{i+1}=\oplus_{k=1}^{l}\Id_{Q_{k}}\oplus\gamma_{i}^{i+1}$. Define $f_{1}=f\circ h_{1}^{2}\circ\cdots \circ h_{l-1}^{l}$. Then the straightforward computations show us
\begin{equation}
\begin{split}
f_{1}^{*}(D_{E}^{''})=\left(\begin{array}{ccccc}
  D_{Q_{1}}^{''} & \tilde{\beta}_{1}^{2} & \cdots &\cdots & * \\
   0 & D_{Q_{2}}^{''} & \tilde{\beta}_{2}^{3} & \cdots & *   \\
   \vdots & \ddots&\ddots &\ddots & \vdots\\
   0 & \cdots &0 & D_{Q_{l-1}}^{''} & \tilde{\beta}_{l-1}^{l}\\
  0  & \cdots& \cdots & 0 & D_{Q_{l}}^{''}
\end{array}\right).
\end{split}
\end{equation}

Inductively, suppose we construct a splitting $f_{p}$, such that $\beta_{i}^{j}$ satisfies the equation (\ref{eq:7-1}) and $D^{'}\beta_{i}^{j}=0$ for $1\leq j-i\leq p$. Let $\tilde{\beta}_{i}^{i+p+1}=\beta_{i}^{i+p+1}+D^{''}\gamma_{i}^{i+p+1}$, where $\gamma_{i}^{i+p+1}\in \Gamma(X,Q_{i+p+1}^{*}\otimes Q_{i})$ satisfies
\begin{equation}\label{eq:7-2}
\sqrt{-1}\Lambda_{\omega}D^{'}D^{''}\gamma_{i}^{i+p+1}=-\sqrt{-1}\Lambda_{\omega}D^{'}\beta_{i}^{i+p+1}.
\end{equation}
This equation can be solved in K\"ahler manifolds case, because we have the following integrability condition
\begin{equation}
\int_{X}\langle\sqrt{-1}\Lambda_{\omega}D^{'}\beta_{i}^{i+p+1},\vartheta\rangle_{H}\frac{\omega^{n}}{n!}
=0
\end{equation}
for any $\vartheta \in \Gamma (Q_{i+p+1}^{\ast } \otimes Q_{i}) $ satisfying $D^{''}\vartheta=D^{'}\vartheta=0$, where $H$ is the induced metric with respect to the Hermitian-Einstein metrics $H_{k}$.
Note that $(Q_{i+p+1}^{\ast }\otimes Q_{i}, \overline{\partial}_{Q_{i+p+1}^{\ast }\otimes Q_{i}}, \theta_{Q_{i+p+1}^{\ast }\otimes Q_{i}})$ is Higgs-Hermitian flat. Calculating directly deduces
\begin{equation}
\begin{split}
0&=\int_{X}\langle\sqrt{-1}[\Lambda_{\omega},(D^{''}+D^{'})^{2}]\tilde{\beta}_{i}^{i+p+1},\tilde{\beta}_{i}^{i+p+1}\rangle_{H,\omega}\frac{\omega^{n}}{n!}\\
&=\int_{X}\langle\sqrt{-1}\Lambda_{\omega}D^{''}D^{'}\tilde{\beta}_{i}^{i+p+1},\tilde{\beta}_{i}^{i+p+1}\rangle_{H,\omega}\frac{\omega^{n}}{n!}\\
&=\int_{X}\langle\sqrt{-1}[\Lambda_{\omega},D^{''}]D^{'}\tilde{\beta}_{i}^{i+p+1},\tilde{\beta}_{i}^{i+p+1}\rangle_{H,\omega}\frac{\omega^{n}}{n!}\\
&=\int_{X}|D^{'}\tilde{\beta}_{i}^{i+p+1}|^{2}_{H,\omega}\frac{\omega^{n}}{n!},
\end{split}
\end{equation}
where in the second equality we have used the condition $D^{''}\tilde{\beta}_{i}^{i+p+1}=D^{''}\beta_{i}^{i+p+1}=-\sum_{k=i+1}^{i+p}\beta_{i}^{k}\wedge \beta_{k}^{i+p+1}$ and $D^{'}\beta_{i}^{j}=0$ for $1\leq j-i\leq p$. Thus $D^{'}\tilde{\beta}_{i}^{i+p+1}=0$.
Set $h_{i}^{i+p+1}=\oplus_{k=1}^{l}\Id_{Q_{k}}\oplus\gamma_{i}^{i+p+1}$ and define $f_{p+1}=f_{p}\circ h_{1}^{2+p}\circ\cdots \circ h_{l-p-1}^{l}$. Then $f_{p+1}$ is a splitting satisfying equation (\ref{eq:7-1}) and $D^{'}\beta_{i}^{j}=0$ for $1\leq j-i\leq p+1$. Finally $\tilde{f}=f_{l-1}$ is a splitting satisfying the conditions in this lemma.
\end{proof}

\begin{remark}\label{rmk:7-1}
In non-K\"ahler manifolds case, the equation (\ref{eq:7-2}) may have no solutions because the right hand side of the equation may not satisfy the integrability condition.
\end{remark}

\begin{proof}[Proof of Theorem \ref{thm:2-3}]
According to Lemma \ref{l:7-1}, we can define a projectively flat connection on $E$ by
\begin{equation}
\begin{split}
\tilde{f}^{*}(D_{E})=\left(\begin{array}{ccccc}
  D_{Q_{1}} & \tilde{\beta}_{1}^{2} & \cdots &\cdots & \tilde{\beta}_{1}^{l} \\
   0 & D_{Q_{2}} & \tilde{\beta}_{2}^{3} & \cdots & \tilde{\beta}_{2}^{l}   \\
   \vdots & \ddots&\ddots &\ddots & \vdots\\
   0 & \cdots &0 & D_{Q_{l-1}} & \tilde{\beta}_{l-1}^{l}\\
  0  & \cdots& \cdots & 0 & D_{Q_{l}}
\end{array}\right).
\end{split}
\end{equation}
Then we can define a map $j:\mathcal{C}_{Dol}(E)\rightarrow \mathcal{C}_{DR}(E, \alpha )$. Next, we will show this map is well-defined. Suppose we have another filtration
\begin{equation}
0=\tilde{\mathfrak{E}}_{0}\subset \tilde{\mathfrak{E}}_{1}\subset\cdots\subset \tilde{\mathfrak{E}}_{l}=\mathfrak{E}.
\end{equation}
Let $\tilde{\mathfrak{Q}}_{i}=(\tilde{Q}_{i}, \overline{\partial}_{\tilde{Q}_{i}}, \tilde{\theta}_{i})=\tilde{\mathfrak{E}}_{i}/\tilde{\mathfrak{E}}_{i-1}$ and choose a suitable splitting $\tilde{g}:\oplus_{i=1}^{l}\tilde{Q}_{i}\rightarrow E$, such that it satisfies the equations in Lemma \ref{l:7-1}. By the same way, we can define a projectively flat connection $\tilde{D}_{E}$ on $E$. Assume the pull-back of $D_{E}^{''}$ is
\begin{equation}
\begin{split}
\tilde{g}^{*}(D_{E}^{''})=\left(\begin{array}{ccccc}
  D_{\tilde{Q}_{1}}^{''} &\cdots & \tilde{\rho}_{1}^{l} \\
   \vdots  &\ddots & \vdots\\
  0  & \cdots & D_{\tilde{Q}_{l}}^{''}
\end{array}\right).
\end{split}
\end{equation}
Set $P=\tilde{g}^{-1}\circ \tilde{f}=(P_{i}^{j})$, where $P_{i}^{j}\in \Gamma(X,Q_{j}^{*}\otimes \tilde{Q}_{i})$. Then
\begin{equation}
P\circ\left(\begin{array}{ccccc}
  D_{Q_{1}}^{''} &\cdots & \tilde{\beta}_{1}^{l} \\
   \vdots  &\ddots & \vdots\\
  0  & \cdots & D_{Q_{l}}^{''}
\end{array}\right)=\left(\begin{array}{ccccc}
  D_{\tilde{Q}_{1}}^{''} &\cdots & \tilde{\rho}_{1}^{l} \\
   \vdots  &\ddots & \vdots\\
  0  & \cdots & D_{\tilde{Q}_{l}}^{''}
\end{array}\right)\circ P.
\end{equation}
This means
\begin{equation}
\left\{\begin{split}
&D^{''}(P_{l}^{1})=0;\\
&D^{''}(P_{i}^{1})+\sum_{k=i+1}^{l}\tilde{\rho}_{i}^{k}\circ P_{k}^{1}=0,\ \ 1\leq i\leq l-1;\\
&D^{''}(P_{l}^{j})-\sum_{k=1}^{j-1}P_{l}^{k}\circ\tilde{\beta}_{k}^{j}=0,\ \ 2\leq j<l;\\
&D^{''}( P_{i}^{j})+\sum_{k=i+1}^{l}\tilde{\rho}_{i}^{k}\circ P_{k}^{j}-\sum_{k=1}^{j-1}P_{i}^{k}\circ\tilde{\beta}_{k}^{j}=0,\ \ 1\leq i\leq l-1,\ 2\leq j\leq l.
\end{split}\right.
\end{equation}
Note that $(Q_{j}^{\ast }\otimes \tilde{Q}_{i}, \overline{\partial}_{Q_{j}^{\ast }\otimes \tilde{Q}_{i}}, \theta_{Q_{j}^{\ast }\otimes \tilde{Q}_{i}})$ is Higgs-Hermitian flat for any $i$ and $j$. Using $D^{''}(P_{l}^{1})=0$, similar to the argument in Section 8, we can show $D^{'}(P_{l}^{1})=0$. By induction, we can prove $D^{'}P_{i}^{j}=0$ for $1\leq i,j\leq l$, which implies $(E,D_{E})\simeq (E,\tilde{D}_{E})$. Therefore the map $j$ is well-defined.

Conversely, we can also define a map $i:\mathcal{C}_{DR}(E, \alpha )\rightarrow \mathcal{C}_{Dol}(E)$ by the same way. It is obvious that $i\circ j=\Id_{\mathcal{C}_{Dol}}$ and $j\circ i=\Id_{\mathcal{C}_{DR}(E, \alpha )}$. So $j$ is a one-to-one map between $\mathcal{C}_{Dol}(E)$ and $\mathcal{C}_{DR}(E, \alpha )$.
\end{proof}

\end{document}